\documentclass[12pt,a4paper]{amsart}
\usepackage[english]{babel}
\usepackage[autostyle]{csquotes}
\usepackage{amssymb,latexsym,amsfonts,amsthm,upref,amsmath}
\usepackage[margin=1in]{geometry}
\usepackage[colorlinks=true]{hyperref}
\hypersetup{citecolor=blue, urlcolor=blue, colorlinks=false, linkcolor=red, citecolor=green, filecolor=magenta}

\newtheorem{thm}{Theorem}
\newtheorem{lem}{Lemma}[subsection]
\newtheorem{prop}{Proposition}[subsection]
\newtheorem{cor}{Corollary}[subsection]

\newtheorem{ax}{Remark}[subsection]

\numberwithin{equation}{section}

\begin{document}

\title{Quasi-particle bases of principal subspaces for the affine Lie algebras of types  \texorpdfstring{$B_\MakeLowercase{l}^{(1)}$}{B_l^{(1)}} and \texorpdfstring{$C_\MakeLowercase{l}^{(1)}$}{C_l^{(1)}}}

\author{Marijana Butorac}

\address{University of Rijeka, Department of Mathematics, Radmile Matej\v{c}i\'{c} 2, 51000 Rijeka, Croatia}

\email{mbutorac@math.uniri.hr}


\begin{abstract} 
Generalizing our earlier work, we construct quasi-particle bases of principal subspaces of standard module $L_{X_l^{(1)}}(k\Lambda_0)$ and generalized Verma module $N_{X_l^{(1)}}(k\Lambda_0)$ at level $k\geq 1$ in the case of affine Lie algebras of types $B_l^{(1)}$ and $C_l^{(1)}$. As a consequence, from quasi-particle bases, we obtain the graded dimensions of these subspaces. \end{abstract}

\maketitle

\section*{Intorduction}
\par Let $\mathfrak{g}$ be a simple complex Lie algebra of type $X_l$, with a Cartan subalgebra $\mathfrak{h}$, the set of simple roots $\Pi=\left\{\alpha_1, \ldots, \alpha_l\right\}$ and the triangular decomposition $\mathfrak{g}=\mathfrak{n}_{-}\oplus\mathfrak{h}\oplus\mathfrak{n}_{+}$, where $\mathfrak{n}_{+}$ is a direct sum of its one dimensional subalgebras corresponding to the positive roots. 
Denote by $\mathcal{L}({\mathfrak{n}}_{+})$ a subalgebra of untwisted affine Lie algebra $\widehat{\mathfrak{g}}$ of type $X_l^{(1)}$ 
\begin{equation*}\mathcal{L}(\mathfrak{n}_{+})=\mathfrak{n}_{+} \otimes \mathbb{C}[t,t^{-1}].
\end{equation*}
Let $V$ be a highest $\widehat{\mathfrak{g}}$-module with highest weight $\Lambda$ and highest weight vector $v_{\Lambda}$. We define the principal subspace $W_{V}$ of $V$ as
\begin{equation*} W_{V}=
U(\mathcal{L}(\mathfrak{n}_{+}))v_{\Lambda}.
\end{equation*}
In this paper we study principal subspaces of the generalized Verma module $N_{X_l^{(1)}}(k{\Lambda}_{0})$  and its irreducible quotient $L_{X_l^{(1)}}(k{\Lambda}_{0})$  at level $k\geq 1$, defined over the affine Lie algebras of type $B_l^{(1)}$ and $C_l^{(1)}$.
\par The study of principal subspaces of standard (i.e., integrable highest weight) modules of the simply laced affine Lie algebras 
and its connection to Rogers-Ramanujan identities was initiated in the work of B. L. Feigin and A. V. Stoyanovsky \cite{FS} and has been further de\-veloped in \cite{AKS}, \cite{Cal1}--\cite{Cal2}, \cite{CalLM1}--\cite{CalLM4}, \cite{CLM1}--\cite{CLM2}, \cite{G}, \cite{Ko}, \cite{KP}, \cite{MiP}, \cite{S1}--\cite{S2}. 
\par Quasi-particle descriptions of principal subspaces of standard modules for untwisted affine Kac-Moody algebras ori\-ginate from the work of Feigin and Stoyanovsky \cite{FS} and G. Georgiev \cite{G}. In order to compute the character formulas of standard $A^{(1)}_1$-modules, Feigin and Stoyanovsky have constructed monomial bases of principal subspaces of the standard modules in terms of the expansion coefficients of a certain vertex operators (cf. \cite{DL}, \cite{LL}). These monomial bases have an interesting physical interpretation, as quasi-particles, whose energies comply the difference-two condition (see in particular \cite{DKKMM}, \cite{FS}, \cite{G}). 
\par Later on, Georgiev in \cite{G} generalized the construction of quasi-particle bases to principal subspaces of certain standard modules in the $ADE$ type. His bases were built of quasi-particles $x_{r\alpha_i}(m)$ of color $i$ ($1\leq i \leq l$), charge $r\geq 1$ and energy $-m$   
\begin{equation*}
x_{r\alpha_{i}}(m)=\text{Res}_z \left\{ z^{m+r-1}\underbrace{x_{\alpha_{i}}(z) \cdots  x_{\alpha_{i}}(z)}_{\text{r factors}}\right\},
\end{equation*}
where $x_{\alpha_{i}}(z)=\sum_{j\in \mathbb Z} x_{\alpha_i}(j) z\sp{-j-1}$ are vertex operators associated to elements $x_{\alpha_i} \in N_{A_l^{(1)}}(k\Lambda_0)$. From this bases Georgiev obtained character formulas, which are in the case of $A_1^{(1)}$ character formulas first obtained by  J. Lepowsky and M. Primc in \cite{LP}.   
\par In our previous paper \cite{Bu} we have used ideas of Georgiev to construct quasi-particle bases of principal subspaces of level $k \geq 1$ standard module $L_{B_2^{(1)}}(k{\Lambda}_{0})$ and generalized Verma module $N_{B_2^{(1)}}(k{\Lambda}_{0})$ of an affine Lie algebra of type $B_2^{(1)}$.  From the graded dimensions (characters) of principal subspaces of generalized Verma module we obtained a new identity of Rogers-Ramanujan's type.
\par Our present work is a generalization of \cite{Bu} to the case of $B_l^{(1)}$, $l\geq 3$, and to the case of $C_l^{(1)}$, $l\geq 3$. 
Our methods for these cases are the same as the methods that we used in \cite{Bu}. First, using relations among vertex operators associated with the simple roots $\alpha_i \in\Pi$, we find spanning sets of principal subspaces, which are built of quasi-particles of colors $i$, $1 \leq i \leq l$, and charges $r \geq 1$ acting on the highest weight vectors. 
\par In the case of affine Lie algebra of type $B_l^{(1)}$ these quasi-particle monomials are of form 
\begin{equation*}
b(\alpha_{l})b(\alpha_{l-1})\cdots b(\alpha_{1}),
\end{equation*}
where $b(\alpha_i)$ is a product of quasi-particles corresponding to simple root $\alpha_i \in\Pi$. From combinatorial point of view, difference conditions for energies of quasi-particles of colors $i$, $1 \leq i \leq l-2$, are identical with the difference conditions for energies of Georgiev's quasi-particles in the case of standard $A_{l-1}^{(1)}$-modules of level $k$ and difference conditions for energies of quasi-particles of colors $l-1$ and $l$ are the same as difference conditions for energies for level $k$ given in \cite{Bu}. 
\par In the case of $C_l^{(1)}$, quasi-particle monomials in the spanning sets are of form 
\begin{equation*}
 b(\alpha_{1})\cdots b(\alpha_{l-1})b(\alpha_{l}),
\end{equation*}
where difference conditions for energies of quasi-particles colored with colors $l$ and $l-1$ are identical as difference conditions for energies of quasi-particles for level $k$ $B_2^{(1)}$-modules, and difference conditions for energies of quasi-particles of colors $i$, $1 \leq i \leq l-2$, are identical with difference conditions for energies of quasi-particles in the case of standard $A_{l-1}^{(1)}$-modules of level $2k$.  
\par For the purpose of proving the linear independence of spanning sets, we use a projection of principal subspaces on the tensor product of $\mathfrak{h}$-weight subspaces of standard modules defined in \cite{Bu}. The projection enables the usage of certain coefficients of intertwining operators, simple current operators and \enquote{Weyl group translation} operator defined on the level one standard modules. We prove linear independence by induction on the order on quasi-particle monomials. Important argument in the proof will be the linear independence of quasi-particle vectors from \cite{Bu} for the $B_2^{(1)}$ case and linear independence of $A_{l-1}^{(1)}$ monomial vectors obtained in \cite{G}.
\par The main results of this paper are character formulas for principal subspaces of standard modules $L_{X_l^{(1)}}(k{\Lambda}_{0})$ (Theorem \ref{uvodBl1} and Theorem \ref{uvodCl1}) and principal subspaces of generalized Verma modules $N_{X_l^{(1)}}(k{\Lambda}_{0})$ (Theorem \ref{uvodBl2} and Theorem \ref{uvodCl2}). As a consequence, we also obtained two new identities, which are generalization of identity from \cite{Bu}. The first one was obtained from character formulas of principal subspace of $N_{B_l^{(1)}}(k{\Lambda}_{0})$ in the $B_l^{(1)}$ case
\begin{thm}
\begin{equation*}
\prod_{m > 0}\frac{1}{(1-q^my_1)}\frac{1}{(1-q^my_1y_2)}\cdots \frac{1}{(1-q^my_1\cdots y_l)}\frac{1}{(1-q^my_1y_2^2\cdots y_l^2)}\cdots \frac{1}{(1-q^my_1y_2\cdots y_l^2)} 
\end{equation*}
\begin{equation*}
\frac{1}{(1-q^my_2)}\frac{1}{(1-q^my_2y_3)}\cdots \frac{1}{(1-q^my_2\cdots y_l)}\frac{1}{(1-q^my_2y_3^2\cdots y_l^2)}\cdots \frac{1}{(1-q^my_2y_3\cdots y_l^2)}
\end{equation*}
\begin{equation*}
\ \ \ \ \ \ \ \ \cdots 
\end{equation*}
\begin{equation*}
\frac{1}{(1-q^{l-1})}\frac{1}{(1-q^my_{l-1}y_l)} \frac{1}{(1-q^my_{l-1}y_l^2)} \frac{1}{(1-q^my_l)} 
\end{equation*}
\begin{equation*}
= \sum_{\substack{r^{(1)}_{1}\geq \ldots \geq r^{(u_1)}_{1}\geq 0\\ u_1\geq0 }}
\frac{q^{r^{(1)^{2}}_{1}+\cdots +r^{(u_1)^{2}}_{1}}}{(q)_{r^{(1)}_{1}-r^{(2)}_{1}}\cdots (q)_{r^{(u_1)}_{1}}}y^{r_1}_{1} 
\end{equation*}
\begin{equation*}
 \sum_{\substack{r^{(1)}_{2}\geq \ldots \geq r^{(u_2)}_{2}\geq 0\\ u_2\geq 0}}
\frac{q^{r^{(1)^{2}}_{2}+\cdots +r^{(u_2)^{2}}_{2}-r_1^{(1)}r_2^{(1)}-\cdots -r_1^{(u_2)}r_2^{(u_2)}}}{(q)_{r^{(1)}_{2}-r^{(2)}_{2}}\cdots (q)_{r^{(u_2)}_{2}}}y^{r_2}_{2} 
\end{equation*}
\begin{equation*}
\ \ \ \ \ \ \ \ \cdots 
\end{equation*}
\begin{equation*}
\sum_{\substack{r^{(1)}_{l-1}\geq \ldots \geq r^{(u_{l-1})}_{l-1}\geq 0\\ u_{l-1}\geq0}}
\frac{q^{r^{(1)^{2}}_{l-1}+\cdots +r^{(u_{l-1})^{2}}_{l-1}-r_{l-2}^{(1)}r_{l-1}^{(1)}-\cdots -r_{l-2}^{(u_{l-1})}r_{l-1}^{(u_{l-1})}}}{(q)_{r^{(1)}_{l-1}-r^{(2)}_{l-1}}\cdots (q)_{r^{(u_{l-1})}_{l-1}}}y^{r_{l-1}}_{l-1}
\end{equation*}
\begin{equation*}
\sum_{\substack{r^{(1)}_{l}\geq \ldots \geq r^{(2u_{l})}_{l}\geq  0\\ u_{l}\geq 0}}
\frac{q^{r^{(1)^{2}}_{l}+\cdots +r^{(2u_{l})^{2}}_{l}-r_{l-1}^{(1)}(r_{l}^{(1)}+r_{l}^{(2)})
-\cdots -r_{l-1}^{(u_{l})}(r_{l}^{(2u_{l}-1)}+r_{l}^{(2u_{l})})}}{(q)_{r^{(1)}_{l}-r^{(2)}_{l}}\cdots (q)_{r^{(2u_{l})}_{l}}}
y^{r_l}_{l}.
\end{equation*}
\end{thm}
In the $C_l^{(1)}$ case we get the following identity
\begin{thm}
\begin{equation*}
\prod_{m > 0} \frac{1}{(1-q^my_1)}\frac{1}{(1-q^my_1y_2)}\cdots \frac{1}{(1-q^my_1\cdots y_{l-1})}\frac{1}{(1-q^my_1y_2^2\cdots y_l^2)}
\end{equation*}
\begin{equation*}
\ \ \ \ \ \ \ \cdots \frac{1}{(1-q^my_1y_2\cdots y_l^2)}\frac{1}{(1-q^my^2_1y^2_2\cdots y_l )} 
\end{equation*}
\begin{equation*}
\frac{1}{(1-q^my_2)}\frac{1}{(1-q^my_2y_3)}\cdots \frac{1}{(1-q^my_2\cdots y_{l-1})}\frac{1}{(1-q^my_2y_3^2\cdots y_l^2)} 
\end{equation*}
\begin{equation*}
\ \ \ \ \ \ \ \cdots \frac{1}{(1-q^my_2y_3\cdots y_l^2)}\frac{1}{(1-q^my^2_2y^2_3\cdots y_l )} 
\end{equation*}
\begin{equation*}
\ \ \ \ \ \ \ \ \cdots 
\end{equation*}
\begin{equation*}
\frac{1}{(1-q^{l-1})}\frac{1}{(1-q^my_{l-1}y^2_l)} \frac{1}{(1-q^my^2_{l-1}y_l)} \frac{1}{(1-q^my_l)}
\end{equation*}
\begin{equation*}
= \sum_{\substack{r^{(1)}_{1}\geq \ldots \geq r^{(2u_1)}_{1}\geq 0\\ u_1\geq0 }}
\frac{q^{r^{(1)^{2}}_{1}+\cdots +r^{(2u_1)^{2}}_{1}}}{(q)_{r^{(1)}_{1}-r^{(2)}_{1}}\cdots (q)_{r^{(2u_1)}_{1}}}y^{r_1}_{1}
\end{equation*}
\begin{equation*}
\sum_{\substack{r^{(1)}_{2}\geq \ldots \geq r^{(2u_2)}_{2}\geq 0\\ u_2\geq 0}}
\frac{q^{r^{(1)^{2}}_{2}+\cdots +r^{(2u_2)^{2}}_{2}-r_1^{(1)}r_2^{(1)}-\cdots -r_1^{(2u_2)}r_2^{(2u_2)}}}{(q)_{r^{(1)}_{2}-r^{(2)}_{2}}\cdots (q)_{r^{(2u_2)}_{2}}}y^{r_2}_{2}
\end{equation*}
\begin{equation*}
\ \ \ \ \ \ \ \ \cdots  
\end{equation*}
\begin{equation*}
\sum_{\substack{r^{(1)}_{l-1}\geq \ldots \geq r^{(2u_{l-1})}_{l-1}\geq 0\\ u_{l-1}\geq0}}
\frac{q^{r^{(1)^{2}}_{l-1}+\cdots +r^{(2u_{l-1})^{2}}_{l-1}-r_{l-2}^{(1)}r_{l-1}^{(1)}-\cdots -r_{l-2}^{(2u_{l-1})}r_{l-1}^{(2u_{l-1})}}}{(q)_{r^{(1)}_{l-1}-r^{(2)}_{l-1}}\cdots (q)_{r^{(2u_{l-1})}_{l-1}}}y^{r_{l-1}}_{l-1}  
\end{equation*}
\begin{equation*}
\sum_{\substack{r^{(1)}_{l}\geq \ldots \geq r^{(u_{l})}_{l}\geq  0\\ u_{l}\geq 0}}
\frac{q^{r^{(1)^{2}}_{l}+\cdots +r^{(u_{l})^{2}}_{l}-r_{l}^{(1)}(r_{l-1}^{(1)}+r_{l-1}^{(2)})
-\cdots -r_{l}^{(u_{l})}(r_{l-1}^{(2u_{l}-1)}+r_{l-1}^{(2u_{l})})}}{(q)_{r^{(1)}_{l}-r^{(2)}_{l}}\cdots (q)_{r^{(u_{l})}_{l}}}
y^{r_l}_{l}. 
\end{equation*}
\end{thm}
\par The plan of the paper is as follows. In Section \ref{s:pril} we recall some fundamental results concerning affine Lie algebras and their modules. Next, we introduce a notion of a quasi-particle and relations among quasi-particles of the same color. In Section \ref{s:psqpm}, we recall the definition of the principal subspace. In Section \ref{Bl3} we construct quasi-particle bases of principal subspaces of standard module $L_{B_l^{(1)}}(k\Lambda_0)$ and genrealized Verma module $N_{B_l^{(1)}}(k\Lambda_0)$ of $B_l^{(1)}$. We will start with finding relations among quasi-particles of different colors. Using these relations along with relations among quasi-particles of the same color we will construct the spanning sets of principal subspaces. Then we will introduce operators which we use in the proof of linear independance. At the end of this section we will find character formulas.  Section \ref{Cl4} is devoted to the construction of bases of principal subspaces in the case of $C_l^{(1)}$. 	

\section{Preliminaries}\label{s:pril}
In this paper we are interested in principal subspaces of two different types of affine Lie algebras, so it will be convenient to introduce principal subspace (and latter quasi-particle monomials) for modules of a general untwisted affine Lie algebra. 

\subsection{Modules of affine Lie algebra} \label{ss:modules}
Let $\mathfrak{g}$ be a complex simple Lie algebra of type $X_l$, $\mathfrak{h}$ a Cartan subalgebra of $\mathfrak{g}$ and $R$ the corresponding root system. Let $\Pi=\left\{\alpha_{1},...,\alpha_{l}\right\}$ be a set of simple roots and let $\theta$ denote the maximal root. Denote with $R_{+}$ ($R_{-}$) the set of positive (negative) roots. Then we have the triangular decomposition 
$\mathfrak{g}=\mathfrak{n}_{-}\oplus\mathfrak{h}\oplus\mathfrak{n}_{+}.$ We use $\left\langle \cdot , \cdot \right\rangle$ to denote the standard symmetric invariant nondegenerate bilinear form on $\mathfrak{g}$ which enables us to identify $\mathfrak{h}$ with its dual $\mathfrak{h}^*$. We normalize this form so that $\left\langle \alpha,\alpha \right\rangle=2$ for every long root 
$\alpha \in R$. For $\alpha \in R$ let $\alpha^{\vee}=\frac{2\alpha}{\left\langle \alpha, \alpha\right\rangle}$ denote the corresponding
coroot. Denote by $Q=\sum^{l}_{i=1}\mathbb{Z}\alpha_{i}$ and $P=\sum^{l}_{i=1}\mathbb{Z}\omega_{i}$ the root and weight lattices, 
 where $\omega_1, \ldots ,\omega_l$ are the fundamental weights of $\mathfrak{g}$, that is $\left\langle \omega_{i} , \alpha_{j}\sp\vee \right\rangle=\delta_{i,j}$, $i,j=1,\ldots , l$. For later use, we set $\omega_{0}=0$. 
\par The associated affine Lie algebra of type $X_l^{(1)}$ is the Lie algebra
\begin{equation*}
\widehat{\mathfrak{g}}= \mathfrak{g}\otimes \mathbb{C}[t,t^{-1}]\oplus \mathbb{C}c,
\end{equation*}
where $c$ is a non-zero central element (cf. \cite{K}). For every $x \in\mathfrak{g}$ and $j \in \mathbb{Z}$, we write $x(j)$ for elements $x\otimes t^{j}$. Commutation relations are then given by
\begin{equation*}
\left[c, x(j)\right]=0,
\end{equation*}
\begin{equation*}
\left[x(j_1),y(j_2)\right]= \left[x, y\right](j_1+j_2) + \left\langle x, y \right\rangle j_1 \delta_{j_1+j_2,0}c,
\end{equation*}
for any $x, y \in \mathfrak{g}, \ j,j_1,j_2 \in \mathbb{Z}$. We introduce the following subalgebras of $\widehat{\mathfrak{g}}$
\begin{align*} 
\widehat{\mathfrak{g}}_{\geq 0}=\bigoplus_{n\geq0} \mathfrak{g}\otimes t^{n}\oplus \mathbb{C}c ,\ \ 
\widehat{\mathfrak{g}}_{< 0}=\bigoplus_{n< 0} \mathfrak{g}\otimes t^{n},
\end{align*}
\begin{align*}
 \mathcal{L}(\mathfrak{n}_{+})=\mathfrak{n}_{+} \otimes \mathbb{C}[t,t^{-1}],
\end{align*}
\begin{align*}    
\mathcal{L}(\mathfrak{n}_{+})_{\geq 0}=\mathcal{L}(\mathfrak{n}_{+}) \otimes \mathbb{C}[t] \ \ \text{and} \ \ \mathcal{L}(\mathfrak{n}_{+})_{< 0}=\mathcal{L}(\mathfrak{n}_{+}) \otimes t^{-1}\mathbb{C}[t].
\end{align*}
\par By adjoining the degree operator $d$ to the Lie algebra $\widehat{\mathfrak{g}}$, such that
\begin{equation}\label{eq:S11}
\left[d, x (j)\right]=jx(j),\ \    [d, c] = 0,
\end{equation}
one obtains the affine Kac-Moody algebra \begin{equation*}\widetilde{\mathfrak{g}}= \widehat{\mathfrak{g}} \oplus \mathbb{C}d,\end{equation*} (cf. \cite{K}).
\par Set $\widetilde{\mathfrak{h}}= \mathfrak{h} \oplus  \mathbb{C}c \oplus \mathbb{C}d$. The form $\left\langle  \cdot, \cdot 
\right\rangle$ on $\mathfrak{h}$ extends naturally to $\widetilde{\mathfrak{h}}$. We shall identify $\widetilde{\mathfrak{h}}$ with its dual space $\widetilde{\mathfrak{h}}^*$ via this form. We define $\delta \in \widetilde{\mathfrak{h}}^{\ast}$ by $\delta(d)=1$, $\delta(c)=0$ and $\delta(h)=0$, for every $h \in \mathfrak{h}$. Set $\alpha_0=\delta-\theta$ and $\alpha_0\sp\vee=c-\theta\sp\vee$. Then  $\left\{\alpha_{0}\sp\vee,\alpha_{1}\sp\vee,\ldots , \alpha_{l}\sp\vee\right\}$ is a set of simple coroots of $\widetilde{\mathfrak{g}}$. 
\par Define fundamental weights of $\widetilde{\mathfrak{g}}$ by $\left\langle \Lambda_{i} , \alpha_{j}\sp\vee \right\rangle=\delta_{i,j}$ for $i,j=0,1,\ldots , l$ and $\Lambda_i\left(d\right)=0$. Denote by $L(\Lambda_0)$, $L(\Lambda_1)$,$\ldots$, $L(\Lambda_l)$ standard $\widetilde{\mathfrak{g}}$-modules, that is integrable highest weight $\widetilde{\mathfrak{g}}$-modules with highest weights $\Lambda_0$, $\Lambda_1, \ldots$, $\Lambda_l$ and with highest weight vectors $v_{\Lambda_0}, v_{\Lambda_1}, \ldots , v_{\Lambda_l}$.  
\par The object of our study is $\widehat{\mathfrak{g}}$-module $N_{X_l^{(1)}}(k\Lambda_{0})$ and its irreducible quotient $L_{X_l^{(1)}}(k\Lambda_{0})$, where level $k$ is a positive integer. The generalized Verma module $N_{X_l^{(1)}}(k\Lambda_{0})$ is defined as the induced $\widehat{\mathfrak{g}}$-module
\begin{equation*} 
N_{X_l^{(1)}}(k\Lambda_{0})= 
U(\widehat{\mathfrak{g}})\otimes_{U(\widehat{\mathfrak{g}}_{\geq 0})} \mathbb{C}v_{k\Lambda_{0}},
\end{equation*}
where $\mathbb{C}v_{k\Lambda_{0}}$ is 1-dimensional $\widehat{\mathfrak{g}}_{\geq 0}$-module, such that \begin{equation*}cv_{ k\Lambda_{0}}=kv_{k\Lambda_{0}}\end{equation*} and \begin{equation*}(\mathfrak{g}\otimes t^{j})v_{ k\Lambda_{0} }=0,\end{equation*} for every $j\geq 0$. From the Poincar\'{e}-Birkhoff-Witt theorem, we have  
\begin{equation*} 
N_{X_l^{(1)}}(k\Lambda_{0})\cong  U(\widehat{\mathfrak{g}}_{<0})\otimes_{\mathbb{C}} \mathbb{C}v_{ k\Lambda_{0} }
\end{equation*} 
as vector spaces. Set \begin{equation*}v_{N_{X_l^{(1)}}(k\Lambda_{0})}=1 \otimes v_{ k\Lambda_{0} }.\end{equation*}
We view $\widehat{\mathfrak{g}}$-modules $N_{X_l^{(1)}}(k\Lambda_{0})$ and $L_{X_l^{(1)}}(k\Lambda_{0})$ as $\widetilde{\mathfrak{g}}$-modules, where $d$ acts as \begin{equation}\label{eq:S13}
dv_{N_{X_l^{(1)}}(k\Lambda_{0})}=0
\end{equation} (see \cite{LL}).
\par 
Throughout this paper, we will write $x(m)$ for the action of $x\otimes t^m$ on any $\widehat{\mathfrak{g}}$-module, where $x \in \mathfrak{g}$ and $j \in \mathbb{Z}$.

\subsection{Definition of quasi-particles}
For every positive integer $k$, the generalized Verma module $N_{X_l^{(1)}}(k\Lambda_{0})$ has a structure of vertex operator algebra (see \cite{LL}, \cite{Li1}, \cite{MP}), where $v_{N_{X_l^{(1)}}(k\Lambda_{0})}$ is the vacuum vector. For $x \in \mathfrak{g}$
\begin{equation*}Y(x(-1)v_{N_{X_l^{(1)}}(k\Lambda_{0})}, z)=x(z)=\sum_{m \in \mathbb{Z}} x(m)z^{-m-1}
\end{equation*}
is vertex operator associated with the vector $x(-1)v_{N_{X_l^{(1)}}(k\Lambda_{0})} \in N_{X_l^{(1)}}(k\Lambda_{0})$. In addition, on the irreducible $\widehat{\mathfrak{g}}$ module $L_{X_l^{(1)}}(k\Lambda_0)$ we have the structure of a simple vertex operator algebra and all the level $k$ standard modules are modules for this vertex operator algebra (cf. \cite{LL}, \cite{MP}).
\begin{ax} Later, we will realize standard modules of level $k > 1$ as submo\-dules of tensor products of standard modules of level $1$. Vertex operators $x(z)$, where $x \in \mathfrak{g}$, act on the tensor product of standard modules of level $1$ as Lie algebra elements (cf. \cite{LL}).
\end{ax}
\par For $\alpha_i \in \Pi$ and $r>0$, we have 
\begin{equation*} 
 x_{r\alpha_{i}}(z):= x_{\alpha_i}(z)^r=Y(\left(x_{\alpha_i}(-1)\right)^rv_{N_{X_l^{(1)}}(k\Lambda_{0})},z).
 \end{equation*}
\par For given $i \in \left\{1,\ldots, l\right\}$, $r \in \mathbb{N}$ and $m\in \mathbb{Z}$ define a quasi-particle 
of color $i$, charge $r$ and energy $-m$ by
\begin{equation}\label{eq:S21}
x_{r\alpha_{i}}(m)=\textup{Res}_z \left\{ z^{m+r-1}\underbrace{x_{\alpha_{i}}(z) \cdots x_{\alpha_{i}}(z)}_{\textup{r factors}}\right\}.
\end{equation}
 We shall say that vertex operator $x_{r\alpha_{i}}(z)$ represents the generating function for quasi-particles of color $i$ and charge $r$.
\par From (\ref{eq:S21}) 
it follows
\begin{equation*} 
x_{r\alpha_{i}}(m)=\sum_{\substack{m_{1},\ldots, m_{r}\in \mathbb{Z} \\ m_{1}+\cdots +m_{r}=m}}x_{\alpha_{i}}(m_{r})\cdots x_{\alpha_{i}}(m_{1}),
\end{equation*}
where the family  of operators \begin{equation*}\left(x_{\alpha_{i}}(m_{r})\cdots  x_{\alpha_{i}}(m_{1})\right)_{\substack{m_{1},\ldots ,m_{r}\in
\mathbb{Z} \\ m_{1}+ \cdots + m_{r}=m}}\end{equation*} on the highest weight module is a summable family (cf. \cite{LL}).  
\par We shall usually denote a product of quasi-particles of color $i$ by $b(\alpha_i)$. We say that monomial $b(\alpha_i)$ is a monochromatic monomial colored with color-type $r_i$, if the sum of all quasi-particle charges in monomial $b(\alpha_i)$ is $r_i$. We say that a monochromatic quasi-particle monomial 
\begin{equation*} b(\alpha_i)=x_{n_{r_{i}^{(1)},i}\alpha_{i}}(m_{r_{i}^{(1)},i})\cdots  x_{n_{1,i}\alpha_{i}}(m_{1,i}),\end{equation*}
is of color-type $r_i$, charge-type
\begin{equation}\label{part1}\left(n_{r_{i}^{(1)},i}, \ldots ,  n_{1,i}\right)\end{equation}
where
\begin{equation*} 0 \leq n_{r_{i}^{(1)},i} \leq \cdots \leq  n_{1,i},\end{equation*}
and dual-charge-type
\begin{equation}\label{part2} \left(r^{(1)}_{i}, r^{(2)}_{i}, \ldots , r^{(s)}_{i}\right),\end{equation}
where \begin{equation*} r^{(1)}_{i} \geq  r^{(2)}_{i} \geq \cdots \geq  r^{(s)}_{i} \geq  0 \ \ \text{and} \ \ \ s \geq 1,\end{equation*} 
if (\ref{part1}) and (\ref{part2}) are conjugate partitions of $r_i$ (cf. \cite{Bu}, \cite{G}).
\par Since quasi-particles of the same color commute, we arrange quasi-particles of the same color and the same charge so 
 that the values $m_{p,i}$, for $1 \leq p \leq  r_i^{(1)}$, form a decreasing sequence of integers from right to left. 
 
\subsubsection{\textbf{Relations among quasi-particles of the same color}}
Relations among particles of the same color, that is, expressions for the  products of the form $x_{n\alpha}(m)x_{n'\alpha}(m')$, where $\alpha=\alpha_i$, $n,n' \in \mathbb{N}$ and $m,m' \in \mathbb{Z}$, can be divided into two sets. The first set of relations is described by the following proposition. 
\begin{prop} [cf. \cite{LL}, \cite{Li1}, \cite{MP}]
Let $k \in \mathbb{N}$. Then we have the following relations on the standard module $L_{X_l^{(1)}}(k\Lambda_{0})$:\begin{align}\label{al:S15}
x_{\alpha}(z)^{k+1}&=0,\\
\label{al:S16}
x_{\beta}(z)^{2k+1}&=0,
\end{align}
where $\alpha \in R$ is a long root and $\beta \in R$ is a short root.
\end{prop}
\par The second set of relations was proved in \cite{F}, \cite{FS}, \cite{G} and \cite{JP}:
\begin{lem}\label{lem:S21}
For fixed $M,j\in \mathbb{Z}$ and $1\leq n\leq n'$ any of $2n$ monomials from the set
\begin{multline*} A=\{x_{n\alpha}(j)x_{n'\alpha}(M-j),x_{n\alpha}(j+1)x_{n'\alpha}(M-j-1),
\ldots ,\\
 \ldots , x_{n\alpha}(j+2n-1)x_{n'\alpha}(M-j-2n+1)\}\end{multline*}
can be expressed as a linear combination of monomials from the set \begin{equation*} \left\{x_{n\alpha}(m)x_{n'\alpha}(m'):m+m'=M\right\} \setminus A\end{equation*} and monomials which have as a factor quasi-particle $x_{(n'+1)\alpha}(j')$, $j' \in \mathbb{Z}$.
\end{lem}
\begin{cor}\label{cor:S21}
Fix $n\in \mathbb{N}$ and $j \in \mathbb{Z}$. The elements from the set \begin{equation*}A_1=\{x_{n\alpha}(m)x_{n\alpha}(m'):m'-2n< m \leq   m'\}\end{equation*}
can be expressed as linear combinations of monomials $x_{n\alpha}(m)x_{n\alpha}(m')$, such that $$m\leq m'-2n$$ 
and monomials with quasi-particle $x_{(n+1)\alpha_i}(j')$, $j' \in \mathbb{Z}$.
\end{cor}
\par In order to find relations among quasi-particles, which are differently colored, we will use the commutator formula among vertex operators: 
\begin{align}\label{al:S14} 
[Y(x_{\alpha}(-1)v_{N(k\Lambda_{0})},z_1), Y(x_{\beta}(-1)^rv_{N(k\Lambda_{0})},z_2)]\\
\nonumber
= \sum_{j \geq 0} \frac{(-1)^j}{j!} \left(\frac{d}{dz_1} 
 \right)^j z^{-1}_2
\delta\left(\frac{z_1}{z_2}\right)Y(x_{\alpha}(j)x_{\beta}(-1)^rv_{N(k\Lambda_{0})},z_2),
\end{align}
where $\alpha, \beta \in R$, (cf. \cite{FHL}).

\section{Principal subspaces and Quasi-particle monomials}\label{s:psqpm}
\subsection{Principal subspace}\label{ss:ps}
Let $k \in \mathbb{N}$ and let $\Lambda = k \Lambda_0$. Set $v_{L_{X_l^{(1)}}(k\Lambda_0)}$ to be the highest weight vector of the standard module $L_{X_l^{(1)}}(k\Lambda_0)$. As in \cite{FS} and \cite{G}, we define the principal subspace $W_{L_{X_l^{(1)}}(k\Lambda_0)}$ of the standard module $L_{X_l^{(1)}}(k\Lambda_0)$ as
\begin{equation*}
    W_{L_{X_l^{(1)}}(k\Lambda_0)}= U(\mathcal{L}(\mathfrak{n}_+))v_{L_{X_l^{(1)}}(k\Lambda_0)},
\end{equation*}
and the principal subspace $W_{N_{X_l^{(1)}}(k\Lambda_{0})}$ of the generalized Verma module $N_{X_l^{(1)}}(k\Lambda_{0})$ as
\begin{equation*}
W_{N_{X_l^{(1)}}(k\Lambda_{0})}= U(\mathcal{L}(\mathfrak{n}_{+}))v_{N_{X_l^{(1)}}(k\Lambda_{0})},
\end{equation*}
where $U(\mathcal{L}(\mathfrak{n}_+))$ is the universal enveloping algebra of Lie algebra $\mathcal{L}(\mathfrak{n}_+)$. 

\subsection{Quasi-particle monomials}\label{ss:qpm}
In Subsections \ref{Bl32} and \ref{Cl42}, we construct bases of $W_{L_{X_l^{(1)}}(k\Lambda_0)}$  and $W_{N_{X_l^{(1)}}(k\Lambda_0)}$ consisting of vectors of the form  $bv_{L_{X_l^{(1)}}(k\Lambda_0)}$ and $bv_{N_{X_l^{(1)}}(k\Lambda_0)}$, where monomials $b$ are composed of monochromatic monomials $b(\alpha_i)$, where $i = 1,\ldots,l$. Here we extend definitions of monochromatic monomials to polychromatic monomials. We use the same terminology for the products of generating functions.  
\par For (polychromatic) monomial 
\begin{equation*}b= b(\alpha_{l})\cdots b(\alpha_{1}),\end{equation*}
\begin{equation*}=x_{n_{r_{l}^{(1)},l}\alpha_{l}}(m_{r_{l}^{(1)},l}) \cdots  x_{n_{1,l}\alpha_{l}}(m_{1,l})\cdots 
x_{n_{r_{1}^{(1)},1}\alpha_{1}}(m_{r_{1}^{(1)},1}) \cdots  x_{n_{1,1}\alpha_{1}}(m_{1,1}),\end{equation*} 
we will say it is of charge-type 
\begin{equation*}
\mathfrak{R}'=\left(n_{r_{l}^{(1)},l}, \ldots ,n_{1,l};\ldots ;n_{r_{l}^{(1)},l}, \ldots ,n_{1,1}\right),\end{equation*} where
\begin{equation*} 
0 \leq n_{r_{i}^{(1)},i}\leq \ldots \leq  n_{1,i},
\end{equation*}
dual-charge-type
\begin{equation*}
\mathfrak{R}= \left(r^{(1)}_{l},\ldots , r^{(s_{l})}_{l};\ldots ;r^{(1)}_{1},\ldots , r^{(s_{1})}_{1} \right),\end{equation*}
where
\begin{equation*} 
r^{(1)}_{i}\geq r^{(2)}_{i}\geq \ldots \geq  r^{(s_{i})}_{i}\geq 0 
\end{equation*}
and color-type
\begin{equation*} \left(r_{l},\ldots, r_{1}\right),\end{equation*}
where 
\begin{equation*} 
r_i=\sum_{p=1}^{r_{i}^{(1)}}n_{p,i}=\sum^{s_{i}}_{t=1}r^{(t)}_{i} \ \ \text{and} \ \ s_{i}\in \mathbb{N},
\end{equation*}
if for every color $i$, $1 \leq i \leq l$,
\begin{equation*}\left(n_{r_{i}^{(1)},i}, \ldots ,n_{1,i}\right)$$ and $$\left(r^{(1)}_{i}, r^{(2)}_{i}, \ldots , 
r^{(s)}_{i}\right)\end{equation*}
are mutually conjugate partitions of $r_i$ (cf. \cite{Bu}, \cite{G}). 
\begin{ax} In the case of affine Lie algebra of type $C_l^{(1)}$ the bases of $W_{L_{C_l^{(1)}}(k{\Lambda}_{0})}$ and $W_{N_{C_l^{(1)}}(k{\Lambda}_{0})}$ will generate polychromatic monomials
\begin{equation*} b(\alpha_{1})\cdots b(\alpha_{l})
\end{equation*}
whose charge-types, dual-charge types and color-types are defined similarly.
\end{ax}
\par We compare the (polychromatic) monomials as in \cite{Bu} and \cite{G}. We state 
 \begin{equation}\label{eq:S23}
b< \overline{b}
\end{equation}
if one of the following conditions holds:
\begin{enumerate}
    \item
$\left(n_{r_{l}^{(1)},l}, \ldots  ,n_{1,1}\right)<
\left(\overline{n}_{\overline{r}_{l}^{(1)},l}, \ldots ,\overline{n}_{1,1}\right)$,\\ i.e., if there is $u \in \mathbb{N}$, such that   $n_{1,i}=\overline{n}_{1,i}, n_{2,i}=\overline{n}_{2,i},\ldots , n_{u-1,i}=\overline{n}_{u-1,i},$ and 
$u=\overline{r}_{i}^{(1)}+1 \ \ \text{or} \ \ n_{u,i}<\overline{n}_{u,i}$;
\item $\left(n_{r_{l}^{(1)},l}, \ldots , n_{1,1}\right)=
\left(\overline{n}_{\overline{r}_{l}^{(1)},l}, \ldots , \overline{n}_{1,1}\right)$,\\ 
$\left(m_{r_{l}^{(1)},l},\ldots , m_{1,1}\right)
<
\left(\overline{m}_{\overline{r}_{l}^{(1)},l}, \ldots
,\overline{m}_{1,1}\right)$\\
i.e. if the\-re is $u \in \mathbb{N}$, $1\leq u \leq r_i$, such that $m_{1,i}=\overline{m}_{1,i},
m_{2,j}=\overline{m}_{2,j},
\ldots m_{u-1,i}=\overline{m}_{u-1,i}$ and $m_{u,i}<\overline{m}_{u,i}$.
\end{enumerate}
\begin{ax}
Similarly definition is for the $C_l^{(1)}$ case. First we compare the charge-types and if the charge-types are the same, we compare the sequences of energies, starting from color $i=l$. 
\end{ax}

\subsection{Characters of principal subspaces}
We extend the definition of character of the principal subspaces $W_{L_{X_l^{(1)}}(k{\Lambda}_{0})}$ and $W_{N_{X_l^{(1)}}(k{\Lambda}_{0})}$ from \cite{Bu}. 
\par Denote by $\text{ch} \ W_{L_{X_l^{(1)}}(k\Lambda_{0})}$ the  characters of $W_{L_{X_l^{(1)}}(k\Lambda_{0})}$:
 \begin{align}\label{S75}
 \text{ch} \ W_{L_{X_l^{(1)}}(k\Lambda_{0})}=\sum_{m,r_1,\ldots, r_l\geq 0} 
\text{dim} \ {W_{L_{X_l^{(1)}}(k\Lambda_{0})}}_{(m,r_1,\ldots, r_l)}q^{m}y^{r_1}_{1}\cdots y^{r_l}_{l},
\end{align}
where $q,\ y_1,\ldots y_l$ are formal variables and 
\begin{align*}
{W_{L_{X_l^{(1)}}(k\Lambda_{0})}}_{(m,r_1, \ldots, r_l)}={W_{L_{X_l^{(1)}}(k\Lambda_{0})}}_{-m\delta +r_1\alpha_1 +\ldots + r_l\alpha_l}
\end{align*} 
is the $\widetilde{\mathfrak{h}}$-weight subspace of weight $-m\delta +r_1\alpha_1 +\ldots + r_l\alpha_l$.
\par In the same way we define the character of the principal subspace $W_{N_{X_l^{(1)}}(k\Lambda_{0})}$.

\section{The case  \texorpdfstring{$B_\MakeLowercase{l}^{(1)}$}{Bl{(1)}}}
\label{Bl3}
\subsection{Principal subspaces for affine Lie algebra of type \texorpdfstring{$B_\MakeLowercase{l}^{(1)}$}{Bl{(1)}}}
Let $\mathfrak{g}$ be of the type $B_l$, $l \geq 2$. The root system $R$ of $\mathfrak{g}$ will be identified as a subset $\mathbb{R}^l$, where $\left\{\epsilon_1,\ldots, \epsilon_l\right\}$ denotes the usual orthonormal basis of the $\mathbb{R}^l$. We have the base of $R$: 
$$\Pi =\left\{\alpha_1= \epsilon_1-\epsilon_2  , \ldots, \alpha_{l-1}=\epsilon_{l-1}-\epsilon_l ,\alpha_l= \epsilon_l\right\},$$ the set of positive roots: 
$$R_{+}=\left\{ \epsilon_{i}- \epsilon_{j}: i < j \right\} \cup \left\{ \epsilon_{i}+ \epsilon_{j}: i \neq j \right\} \cup \left\{\epsilon_{i}: 1 \leq i \leq l\right\}$$ and the highest root 
$$\theta= \epsilon_1+\epsilon_2=\alpha_1+2\alpha_2 + \cdots + 2\alpha_l.$$
For each root $\alpha \in R_+$ we have a root vector $x_{\alpha} \in \mathfrak{g}$. We define a one-dimensional subalgebras of $\mathfrak{g}$ 
\begin{equation*}
\mathfrak{n}_{\alpha} = \mathbb{C}x_{\alpha}, \  \  \alpha \in R_+,
\end{equation*}
with the corresponding subalgebras of $\widehat{\mathfrak{g}}$
\begin{equation*}
\mathcal{L}(\mathfrak{n}_{\alpha} ) = \mathfrak{n}_{\alpha}  \otimes \mathbb{C}[t,t^{-1}].
\end{equation*}
We denote with $U_{B_{l}^{(1)}}$ the vector space 
\begin{equation*}
U_{B_{l}^{(1)}} = U(\mathcal{L}\left(\mathfrak{n}_{\alpha_{l}}\right))\cdots U(\mathcal{L}\left(\mathfrak{n}_{\alpha_{1}}\right)).
\end{equation*}
Using the same argument as Georgiev in \cite{G} we can prove
\begin{lem}\label{prvalemaB} Let $k \geq 1$. We have 
\begin{align}\nonumber
W_{L_{B_{l}^{(1)}}(k\Lambda_{0})}& = U_{B_{l}^{(1)}}v_{L_{B_{l}^{(1)}}(k\Lambda_{0})},\\
\nonumber
W_{N_{B_{l}^{(1)}}(k\Lambda_{0})}& = U_{B_{l}^{(1)}}v_{N_{B_{l}^{(1)}}(k\Lambda_{0})}.
\end{align}
\begin{flushright}
$\square$
\end{flushright}
\end{lem} 
By extending the construction of bases of principal subspaces in the case of affine Lie algebra of type $B_2^{(1)}$, we shall construct bases for the principal subspaces $W_{L_{B_{l}^{(1)}}(k\Lambda_{0})}$ and $W_{N_{B_{l}^{(1)}}(k\Lambda_{0})}$, which will be generated by quasi-particles acting on the highest weight vectors. We start with the principal subspaces $W_{L_{B_{l}^{(1)}}(k\Lambda_{0})}$.

\subsection{The spanning set of \texorpdfstring{$W_{L_{B_{l}^{(1)}}(k\Lambda_{0})}$}{W{L{Bl{(1)}}(kLambda{0})}}}\label{Bl32}
In order to find a set of quasi-particle monomials which generate a basis of $W_{L_{B_{l}^{(1)}}(k\Lambda_{0})}$, first we complete the list of relations among quasi-particles. Here we find the expressions for the products of the form  $x_{n_i\alpha_i}(m_i)x_{n'_j\alpha_j}(m'_j)$, where $i =1,\ldots , l-1$, $j=i+ 1$ $n_i,n_j' \in \mathbb{N}$ and $m_i,m'_j \in \mathbb{Z}$.
\par First, notice that as in the case of $B_2^{(1)}$, we have:
\begin{lem}\label{lem:S321B}
Let $n_{l-1},n_l \in \mathbb{N}$ be fixed. One has
\begin{align}\label{al:S3211B} 
\left(1-\frac{z_{l-1}}{z_{l}}\right)^{\emph{\text{min}}\left\{ 
n_{l},2n_{l-1}\right\}}  x_{n_{l}\alpha_{l}}(z_{l}) x_{n_{l-1}\alpha_{l-1}}(z_{l-1})v_{N_{B_l^{(1)}}(k\Lambda_{0})}\\
\nonumber
\in z_{l}^{-\emph{\text{min}}\left\{ 
n_{l},2n_{l-1}\right\}}W_{N_{B_l^{(1)}}(k\Lambda_{0})}\left[\left[z_{l},z_{l-1}\right]\right]. 
\end{align}
\end{lem}
\begin{flushright}
$\square$
\end{flushright}
Now, fix color $i$, $1 \leq i \leq l-2$. 
\begin{lem}\label{lem:S322B}
Let  $n_{i+1},n_i \in \mathbb{N}$ be fixed. One has
\begin{itemize}
	\item [a)] 
$(z_{1}-z_{2})^{n_i}x_{n_i\alpha_{i}}(z_{1})x_{n_{i+1}\alpha_{i+1}}(z_{2})=(z_{1}-z_{2})^{n_i}x_{n_{i+1}\alpha_{i+1}}(z_{2})x_{n_i\alpha_{i}}(z_{1});$
\item [b)]
$(z_{1}-z_{2})^{n_{i+1}}x_{n_i\alpha_{i}}(z_{1})x_{n_{i+1}\alpha_{i+1}}(z_{2})=(z_{1}-z_{2})^{n_{i+1}}x_{n_{i+1}\alpha_{i+1}}(z_{2})x_{n_i\alpha_{i}}(z_{1}).$
\end{itemize}
\end{lem}
\begin{proof}
Follows by direct computation employing the commutator formula \ref{al:S14} for vertex operators.
\end{proof}
From  Lemma \ref{lem:S322B} follows:
\begin{lem}\label{lem:S323B}
Let $n_{i+1},n_i \in \mathbb{N}$ be fixed. One has
\begin{align}\label{al:S3233B} 
\left(1-\frac{z_{i}}{z_{i+1}}\right)^{\emph{\text{min}}\left\{ 
n_{ i+1},n_{ i}\right\}} x_{n_{i+1}\alpha_{i+1}}(z_{i+1}) x_{n_{i}\alpha_{i}}(z_{i})v_{N_{B_l^{(1)}}(k\Lambda_{0})}\\
\nonumber
\in z_{i+1}^{-\emph{\text{min}}\left\{ 
n_{ i+1},n_{ i}\right\}} W_{N_{B_l^{(1)}}(k\Lambda_{0})}\left[\left[z_{i+1},z_{i}\right]\right]. 
\end{align}
\end{lem}
\begin{proof}
(\ref{al:S3233B}) is immediate from creation property of vertex operators (cf. \cite{LL}) and Lemma \ref{lem:S322B}, i.e.
\begin{equation*}(z_{i+1}-z_{i})^{\emph{\text{min}}\left\{ 
n_{ i+1},n_{ i}\right\}}x_{n_{i+1}\alpha_{i+1}}(z_{i+1}) x_{n_{i}\alpha_{i}}(z_{i})
=(z_{i+1}-z_{i})^{\emph{\text{min}}\left\{ 
n_{ i+1},n_{ i}\right\}}x_{n_{i}\alpha_{i}}(z_{i})x_{n_{i+1}\alpha_{i+1}}(z_{i+1}).\end{equation*}
\end{proof}
\begin{ax}
The obtained relation in Lemma \ref{lem:S323B} is generalization of relations obtained in \cite{G} for the case of $A_{l-1}^{(1)}$. Later, we will show similar relation for quasi-particles corresponding to the long roots in the $C_l^{(1)}$ case (see Lemma \ref{lem:S235Cl}).
\end{ax}
\par Using the above considerations and relations among quasi-particles of the same color, induction on charge-type and total energie of quasi-particle monomials (\cite{Bu}, \cite{G}), follows the proof of the following proposition: 
\begin{prop}\label{prop:S22B} The set $\mathfrak{B}_{W_{L_{B_l^{(1)}}(k\Lambda_{0})}}=\left\{bv_{L_{B_l^{(1)}}(k\Lambda_{0})}:b \in B_{W_{L_{B_l^{(1)}}(k\Lambda_{0})}}\right\}$, where
\begin{equation}\label{SkupLB}
B_{W_{L_{B_l^{(1)}}(k\Lambda_{0})}}= \bigcup_{\substack{n_{r_{1}^{(1)},1}\leq \ldots \leq n_{1,1}\leq 
k\\\substack{ \ldots \\n_{r_{l-1}^{(1)},l-1}\leq \ldots \leq n_{1,l-1}\leq k}\\n_{r_{l}^{(1)},l}\leq \ldots \leq n_{1,l}\leq 2k}}\left(\text{or, equivalently,} \ \ \ 
\bigcup_{\substack{r_{1}^{(1)}\geq \cdots\geq r_{1}^{(k)}\geq 0\\\substack{\cdots \\r_{l-1}^{(1)}\geq \cdots\geq r_{l-1}^{(k)}\geq 
0}\\r_{l}^{(1)}\geq \cdots\geq r_{l}^{(2k)}\geq 
0}}\right)
\end{equation}
\begin{equation*}\left\{b\right.= b(\alpha_{l})\cdots b(\alpha_{1})
=x_{n_{r_{l}^{(1)},l}\alpha_{l}}(m_{r_{l}^{(1)},l})\cdots x_{n_{1,l}\alpha_{l}}(m_{1,l})\cdots x_{n_{r_{1}^{(1)},1}\alpha_{1}}(m_{r_{1}^{(1)},1})\cdots  x_{n_{1,1}\alpha_{1}}(m_{1,1}):\end{equation*}
\begin{align}\nonumber
\left|
\begin{array}{l}
m_{p,i}\leq  -n_{p,i}+ \sum_{q=1}^{r_{i-1}^{(1)}}\text{min}\left\{n_{q,i-1},n_{p,i}\right\}- \sum_{p>p'>0} 
2 \ \text{min}\{n_{p,i}, n_{p',i}\}, \\
\ \ \ \ \ \ \ \ \ \ \ \ \ \ \ \ \ \ \ \ \ \ \ \ \ \ \ \ \ \ \ \ \ \ \ \ \ \ \ \ \ \ \ \ \ \ \ \ \ \ \ \ \ \ \ \ \ \ \ \ \ \ \ \ \ \ 1\leq  p\leq r_{i}^{(1)}, \ 1 \leq i \leq l-1;\\
m_{p+1,i} \leq   m_{p,i}-2n_{p,i} \  \text{if} \ n_{p+1,i}=n_{p,i}, \ 1\leq  p\leq r_{i}^{(1)}-1, \ 1 \leq i \leq l-1;\\
m_{p,l}\leq  -n_{p,l} + \sum_{q=1}^{r_{l-1}^{(1)}}\text{min}\left\{ 2n_{q,l-1},n_{p,l }\right\} - \sum_{p>p'>0} 2 \ \text{min}\{n_{p,l}, n_{p',l}\}, \  1\leq  p\leq r_{l}^{(1)};\\
m_{p+1,l}\leq  m_{p,l}-2n_{p,l} \  \text{if} \ n_{p,l}=n_{p+1,l}, \  1\leq  p\leq r_{l}^{(1)}-1 
\end{array}\right\},
\end{align}
and where $r_0^{(1)}=0$, spans the principal subspace $W_{L_{B_l^{(1)}}(k\Lambda_{0})}$.
\end{prop}
\begin{flushright}
$\square$
\end{flushright}

\subsection{Proof of linear independence} 
Here we introduce operators which we use in our proof of linear independence of the set $\mathfrak{B}_{W_{L_{B_l^{(1)}}(k\Lambda_{0})}}$. 
\subsubsection{\textbf{Projection \texorpdfstring{$\pi_{\mathfrak{R}}$}{pi{R}}}}\label{ss:projB}
We start with a projection $\pi_{\mathfrak{R}}$, which is a generalisation of projection introduced in \cite{Bu}. If we restrict the action of the Cartan subalgebra $\mathfrak{h}=\mathfrak{h} \otimes 1$ to the principal subspace $W_{L_{B_l^{(1)}}(\Lambda_{0})}$ of level $1$ standard  modules $L_{B_l^{(1)}}(\Lambda_{0})$, we get the direct sum of vector spaces: 
 \begin{equation*}
 W_{L_{B_l^{(1)}}(\Lambda_{0})}= \bigoplus_{u_l,\ldots, u_1\geq 0} {W_{L_{B_l^{(1)}}(\Lambda_{0})}}_{(u_l, \ldots, u_1)},
\end{equation*}
where 
\begin{equation*}
{W_{L_{B_l^{(1)}}(\Lambda_{0})}}_{(u_l, \ldots, u_1)}={W^B_{L(\Lambda_{0})}}_{u_l\alpha_l + \cdots +u_1\alpha_1}
\end{equation*}
is  the weight subspace of weight 
\begin{equation*}
u_l\alpha_l + \cdots +u_1\alpha_1 \in Q.
\end{equation*}
\par Fix a level $k> 1$. The principal subspace $W_{L_{B_l^{(1)}}(k\Lambda_{0})}$ has a realization as a subspace of the tensor product
of $k$ principal subspaces $W_{L_{B_l^{(1)}}(\Lambda_{0})}$ of level 1
\begin{align*}
W_{L_{B_l^{(1)}}(k\Lambda_{0})}
\subset  W_{L_{B_l^{(1)}}(\Lambda_{0})}\otimes \cdots \otimes  W_{L_{B_l^{(1)}}(\Lambda_{0})}
\subset  L_{B_l^{(1)}}(\Lambda_{0})^{\otimes k},
\end{align*}
where
\begin{equation*}
v_{L_{B_l^{(1)}}(k\Lambda_0)}=\underbrace{v_{L_{B_l^{(1)}}(\Lambda_{0})} \otimes \cdots \otimes v_{L_{B_l^{(1)}}(\Lambda_{0})}}_{k \ \text{factors}}\end{equation*}
is the highest weight vector of weight $k\Lambda_0$. 
\par For a chosen dual-charge-type 
\begin{equation*}
\mathfrak{R}=\left( r_{l}^{(1)}, \ldots ,r_{l}^{(2k)}; r_{l-1}^{(1)}, \ldots , r_{l-1}^{(k)};\ldots; r_{1}^{(1)}, \ldots , r_{1}^{(k)}\right)
\end{equation*}
and the corresponding charge-type $\mathfrak{R}'$
\begin{equation*}\mathfrak{R}'=\left(n_{r_{l}^{(1)},l}, \ldots ,n_{1,l};\ldots ;n_{r_{1}^{(1)},1}, \ldots ,n_{1,1}\right),\end{equation*} 
denote with $\pi_{\mathfrak{R}}$ the projection of principal subspace $W_{L_{B_l^{(1)}}(k\Lambda_{0})}$ 
to the subspace
\begin{equation*}
 {W_{L_{B_l^{(1)}}(\Lambda_0)}}_{(\mu^{(k)}_{l};r_{l-1}^{(k)};\ldots; r_{1}^{(k)})}\otimes \cdots \otimes  {W_{L_{B_l^{(1)}}(\Lambda_0)}}_{(\mu^{(1)}_{l};r_{l-1}^{(1)};\ldots;r_{1}^{(1)})},
\end{equation*}
where 
\begin{equation*}  
\mu^{(t)}_{l}=r^{(2t)}_{l}+ r^{(2t-1)}_{l},
\end{equation*}
for every $1 \leq  t \leq k$ (cf. Figure \ref{slika1B} and \ref{slika2B}). The projection can be in an obvious way generalized to the space of formal Laurent series
with coefficients in $W_{L_{B_l^{(1)}}(\Lambda_0)} \otimes \cdots \otimes  W_{L_{B_l^{(1)}}(\Lambda_0)}$. Let
\begin{equation}\label{eq:S331B}
x_{n_{r_{l}^{(1)},l}\alpha_{l}}(z_{r_{l}^{(1)},l}) \cdots     x_{n_{1,l}\alpha_{l}}(z_{1,l})x_{n_{r_{l-1}^{(1)},l-1}\alpha_{l-1}}(z_{r_{l-1}^{(1)},l-1}) \cdots     x_{n_{1,l-1}\alpha_{l-1}}(z_{1,l-1})\cdots 
\end{equation}
\begin{equation*}
\cdots x_{n_{r_{1}^{(1)},1}\alpha_{1}}(z_{r_{1}^{(1)},1})\cdots  x_{n_{1,1}\alpha_{1}}(z_{1,1})
\end{equation*}
be a generating function of the chosen dual-charge-type $\mathfrak{R}$ and the corresponding charge-type $\mathfrak{R}'$. 
\par From the relation $x_{2\alpha_{i}}(z)=0$, $1\leq i \leq l-1$, on the principal subspace $W_{L_{B_l^{(1)}}(\Lambda_0)}$ and the definition of the action of Lie algebra on the modules, follows that $n_{p,i}$ generating functions $x_{\alpha_{i}}(z_{p,i})$ ($1\leq p \leq r^{(1)}_{i}$), whose product generates a quasi-particle of charge $n_{p,i}$, 
 \enquote{are placed at} the first (from right to left) $n_{p,i}$ tensor factors: 
\begin{equation*}
 x_{n^{(k)}_{p,i}\alpha_{i}}(z_{p,i})\otimes  x_{n^{(k-1)}_{p,i}\alpha_{i}}(z_{p,i}) \otimes \cdots \otimes      x_{n^{(2)}_{p,i}\alpha_{i}}(z_{p,i})\otimes  x_{n^{(1)}_{p,i}\alpha_{i}}(z_{p,i}),
\end{equation*}
where
\begin{equation*}
0 \leq  n^{(t)}_{p,i} \leq 1, 1 \leq t \leq k, \ n^{(1)}_{p,i}\geq n^{(2)}_{p,i}\geq \ldots \geq  n^{(k-1)}_{p,i}\geq  n^{(k)}_{p,i}, \ 
n_{p,i}=\sum_{t=1}^k n^{(t)}_{p,i},
\end{equation*}
for every every $p$, $1 \leq p \leq r_{i}^{(1)}$, so that, in the $t$-tensor factor from the right ($1 \leq t \leq k$), we have: 
\begin{equation*}
 \cdots x_{n_{r^{(t)}_{i},i}^{(t)}\alpha_{i}}(z_{r_{i}^{(t)},i})
  \cdots  x_{n{_{1,i}^{(t)}\alpha_{i}}}(z_{1,i})\cdots v_{L_{B_l^{(1)}}(\Lambda_{0})} \otimes \cdots, \ \ 
\end{equation*}
as in the in the example in Figure \ref{slika1B}, where each box represents $n^{(t)}_{p,i}$.
\begin{figure}[h!tb]
\centering
\setlength{\unitlength}{5mm}
\begin{picture}(10,10)
\linethickness{0.075mm}
\multiput(0,0)(1,0){1}
{\line(0,1){1}}
\multiput(1,0)(1,0){1}
{\line(0,1){2}}
\multiput(2,0)(1,0){5}
{\line(0,1){3}}
\multiput(2,1)(1,0){1}
{\line(1,0){6}}
\multiput(2,3)(1,0){1}
{\line(1,0){7}}
\multiput(1,0)(1,0){3}
{\line(0,1){2}}
\multiput(1,2)(0,1){1}
{\line(13,0){8}}
\multiput(4,0)(1,0){1}
{\line(0,1){2}}
\multiput(5,0)(1,0){2}
{\line(0,1){2}}
\multiput(8,0)(1,0){1}
{\line(0,1){2}}
\multiput(9,0)(1,0){1}
{\line(0,1){3}}
\multiput(8,5)(1,0){2}
{\line(0,1){3}}
\multiput(7,0)(1,0){1}
{\line(0,1){2}}
\multiput(7,5)(1,0){2}
{\line(0,1){3}}
\multiput(5,5)(1,0){2}
{\line(0,1){3}}
\multiput(5,7)(1,0){2}
{\line(1,0){3}}
\multiput(5,8)(1,0){2}
{\line(1,0){3}}
\multiput(4,7)(1,0){2}
{\line(1,0){3}}
\multiput(4,6)(1,0){2}
{\line(1,0){4}}
\multiput(4,5)(1,0){2}
{\line(1,0){4}}
\multiput(4,5)(1,0){1}
{\line(0,1){2}}
\multiput(0,0)(0,1){2}
{\line(1,0){9}}
\multiput(5,0)(1,0){2}
{\line(0,1){2}}
\multiput(7,0)(1,0){1}
{\line(0,1){3}}
\multiput(8,0)(1,0){2}
{\line(0,1){3}}
\multiput(6,0)(1,0){2}
{\line(0,1){2}}
\put(0,-0.5){\scriptsize{\footnotesize{$n_{r_i^{(1)},i}$}}}
\put(5.1,-0.5){\scriptsize{\footnotesize{$n_{r_i^{(k)},i}$}}}
\put(8.2,-0.5){\scriptsize{\footnotesize{$n_{1,i}$}}}
\put(4.5,3.7){$\textbf{\vdots}$}
\put(5.5,3.7){$\textbf{\vdots}$}
\put(6.5,3.7){$\textbf{\vdots}$}
\put(7.5,3.7){$\textbf{\vdots}$}
\put(8.5,3.7){$\textbf{\vdots}$}
\put(9.3,3.7){$\textbf{\vdots}$}
\put(-0.9,3.7){$\textbf{\vdots}$}
\put(-1.1,0.3){\footnotesize{$r_i^{(1)}$}}
\put(-1.1,1.3){\footnotesize{$r_i^{(2)}$}}
\put(-1.1,2.3){\footnotesize{$r_i^{(3)}$}}
\put(9.3,0.3){\footnotesize{$v_{L_{B_l^{(1)}}(\Lambda_{0})}$}}
\put(9.3,1.3){\footnotesize{$v_{L_{B_l^{(1)}}(\Lambda_{0})}$}}
\put(9.3,2.3){\footnotesize{$v_{L_{B_l^{(1)}}(\Lambda_{0})}$}}
\put(9.3,5.3){\footnotesize{$v_{L_{B_l^{(1)}}(\Lambda_{0})}$}}
\put(9.3,6.3){\footnotesize{$v_{L_{B_l^{(1)}}(\Lambda_{0})}$}}
\put(9.3,7.3){\footnotesize{$v_{L_{B_l^{(1)}}(\Lambda_{0})}$}}
\put(-1.1,5.3){\footnotesize{$r_i^{(k-2)}$}}
\put(-1.1,6.3){\footnotesize{$r_i^{(k-1)}$}}
\put(-1.1,7.3){\footnotesize{$r_i^{(k)}$}}
\end{picture}
\bigskip
\caption{Sketch of projection $\pi_{\mathfrak{R}}$ for color $i$, $1 \leq i\leq l-1$}\label{slika1B}
\end{figure}
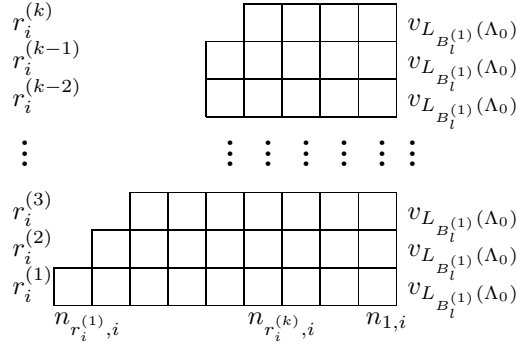
\par From the relation $x_{3\alpha_{l}}(z)=0$ on the principal subspace $W_{L_{B_l^{(1)}}(\Lambda_0)}$, follows that at most, two generating functions of color $i=l$ \enquote{can be placed at} every tensor factor. If $n_{p,l}$ ($1\leq p \leq r_l^{(1)}$) is an even number, then two generating functions $x_{\alpha_{l}}(z_{p,l})$ \enquote{are placed at} the  first $\frac{n_{p,l}}{2}$ tensor factors (from right to left) and if $n_{p,l}$ is an odd number, then two generating functions  $x_{\alpha_{l}}(z_{p,l})$ \enquote{are placed at} the first $\frac{n_{p,l}-1}{2}$ tensor factors (from right to left), and the last generating  function $x_{\alpha_{l}}(z_{p,l})$ \enquote{is placed at} $\frac{n_{p,l}-1}{2}+1$ tensor factor:  
\begin{equation*}
 x_{n^{(k)}_{p,l}\alpha_{1}}(z_{p,l}) \otimes  x_{n^{(k-1)}_{p,l}\alpha_{1}}(z_{p,l}) \otimes  \cdots \otimes   x_{n^{(2)}_{p,l}\alpha_{1}}(z_{p,l})\otimes  x_{n^{(1)}_{p,l}\alpha_{1}}(z_{p,l}),
\end{equation*} 
where
\begin{equation*}
0 \leq  n^{(t)}_{p,l}\leq  2,  \ n^{(1)}_{p,l}\geq  n^{(2)}_{p,l} \geq  \ldots  \geq  n^{(k-1)}_{p,l} \geq  n^{(k)}_{p,l}, \ n_{p,l}=\sum_{t=1}^k n^{(t)}_{p,l},
\end{equation*}
for every every $p$, $1 \leq p \leq r_{l}^{(1)}$, so that at most one $n^{(t)}_{p,l}$ ($1 \leq t \leq k$) can be $1$ 
and so that, in every $t$-tensor factor from the right ($1 \leq t \leq k$), we have: 
\begin{equation*}
\cdots \otimes x_{n_{r^{(2t-1)}_{l},l}^{(t)}\alpha_{l}}(z_{r_{1}^{(2t-1)},l}) \cdots        x_{n_{r^{(2t)}_{l},l}^{(t)}\alpha_{l}}(z_{r_{l}^{(2t)},l})
\cdots  x_{n{_{1,l}^{(t)}\alpha_{l}}}(z_{1,l})\cdots  v_{L_{B_l^{(1)}}(\Lambda_{0})}\otimes \cdots \ \ . 
\end{equation*}
This situation is shown in the example in Figure \ref{slika2B}.
\begin{figure}[h!tb]
\centering
\setlength{\unitlength}{5mm}
\begin{picture}(10,10)
\linethickness{0.075mm}
\multiput(0,0)(1,0){1}
{\line(0,1){1}}
\multiput(1,0)(1,0){1}
{\line(0,1){2}}
\multiput(2,0)(1,0){8}
{\line(0,1){4}}
\multiput(2,1)(1,0){1}
{\line(1,0){6}}
\multiput(2,3)(1,0){1}
{\line(1,0){7}}
\multiput(2,4)(1,0){1}
{\line(1,0){7}}
\multiput(1,0)(1,0){3}
{\line(0,1){2}}
\multiput(1,2)(0,1){1}
{\line(13,0){8}}
\multiput(4,0)(1,0){1}
{\line(0,1){2}}
\multiput(5,0)(1,0){2}
{\line(0,1){2}}
\multiput(8,0)(1,0){1}
{\line(0,1){2}}
\multiput(9,0)(1,0){1}
{\line(0,1){3}}
\multiput(8,6)(1,0){2}
{\line(0,1){3}}
\multiput(7,0)(1,0){1}
{\line(0,1){2}}
\multiput(7,6)(1,0){2}
{\line(0,1){3}}
\multiput(5,6)(1,0){2}
{\line(0,1){2}}
\multiput(5,7)(1,0){2}
{\line(1,0){3}}
\multiput(5,8)(1,0){2}
{\line(1,0){3}}
\multiput(6,9)(1,0){2}
{\line(1,0){2}}
\multiput(6,8)(1,0){1}
{\line(0,1){1}}
\multiput(7,8)(1,0){3}
{\line(0,1){2}}
\multiput(7,10)(1,0){2}
{\line(1,0){1}}
\multiput(5,8)(1,0){2}
{\line(1,0){2}}
\multiput(4,7)(1,0){2}
{\line(1,0){4}}
\multiput(4,6)(1,0){2}
{\line(1,0){4}}
\multiput(4,6)(1,0){1}
{\line(0,1){1}}
\multiput(0,0)(0,1){2}
{\line(1,0){9}}
\multiput(5,0)(1,0){2}
{\line(0,1){2}}
\multiput(7,0)(1,0){1}
{\line(0,1){3}}
\multiput(8,0)(1,0){2}
{\line(0,1){3}}
\multiput(6,0)(1,0){2}
{\line(0,1){2}}
\put(0,-0.5){\footnotesize{$n_{r_l^{(1)},i}$}}
\put(8.5,-0.5){\scriptsize{\footnotesize{$n_{1,l}$}}}
\put(4.5,4.7){$\textbf{\vdots}$}
\put(5.5,4.7){$\textbf{\vdots}$}
\put(6.5,4.7){$\textbf{\vdots}$}
\put(7.5,4.7){$\textbf{\vdots}$}
\put(8.5,4.7){$\textbf{\vdots}$}
\put(9.3,4.7){$\textbf{\vdots}$}
\put(-0.9,4.7){$\textbf{\vdots}$}
\put(-1.1,0.3){\footnotesize{$r_l^{(1)}$}}
\put(-1.1,1.3){\footnotesize{$r_l^{(2)}$}}
\put(-1.1,2.3){\footnotesize{$r_l^{(3)}$}}
\put(-1.1,3.3){\footnotesize{$r_l^{(4)}$}}
\put(9.3,0.9){\footnotesize{$v_{L_{B_l^{(1)}}(\Lambda_{0})}$}}
\put(9.3,2.9){\footnotesize{$v_{L_{B_l^{(1)}}(\Lambda_{0})}$}}
\put(9.3,6.9){\footnotesize{$v_{L_{B_l^{(1)}}(\Lambda_{0})}$}}
\put(9.3,8.9){\footnotesize{$v_{L_{B_l^{(1)}}(\Lambda_{0})}$}}
\put(-1.1,6.3){\footnotesize{$r_l^{(2k-3)}$}}
\put(-1.1,7.3){\footnotesize{$r_l^{(2k-2)}$}}
\put(-1.1,8.3){\footnotesize{$r_l^{(2k-1)}$}}
\put(-1.1,9.3){\footnotesize{$r_l^{(2k)}$}}
\end{picture}
\bigskip
\caption{Sketch of projection $\pi_{\mathfrak{R}}$ for color $i=l$}\label{slika2B}
\end{figure}
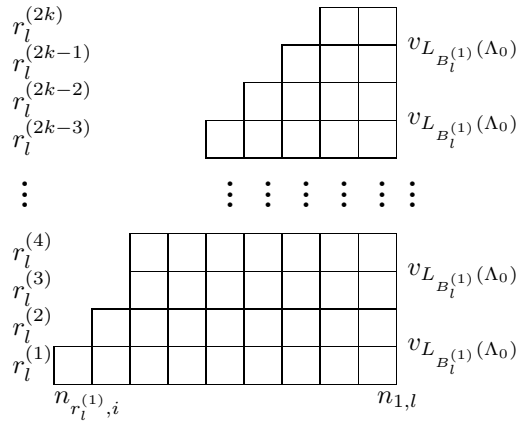

\par Now, we have the projection of the generating function (\ref{eq:S331B}) 
\begin{align}
\label{projekcijaB}
\pi_{\mathfrak{R}}& x_{n_{r_{l}^{(1)},l}\alpha_{l}}(z_{r_{l}^{(1)},l})\cdots  x_{n_{1,1}\alpha_{1}}(z_{1,1}) \ v_{L_{B_l^{(1)}}(k\Lambda_{0})}&\\
\nonumber
=&\text{C} \ x_{n^{(k)}_{r^{(2k-1)}_{l},l}\alpha_{l}}(z_{r_{l}^{(2k-1)},l})\cdots                x_{n^{(k)}_{r^{(2k)}_{l},l}\alpha_{l}}(z_{r^{(2k)}_{l},l})\cdots  x_{n^{(k)}_{1,l}\alpha_{l}}(z_{1,l})x_{n^{(k)}_{r_{l-1}^{(1)},l-1}\alpha_{l-1}}(z_{r_{l-1}^{(1)},l-1}) \cdots     &\\
\nonumber
& \ \ \ \ \ \ \ \ \ \ \ \ \ \ \ \ \ \ \ \ \ \ \ \ \ \ \ \cdots x_{n^{(k)}_{1,l-1}\alpha_{l-1}}(z_{1,l-1})\cdots x_{n^{(k)}_{r^{(k)}_{1},1}\alpha_{1}}(z_{r_{1}^{(k)},1})\cdots   x_{n^{(k)}_{1,1}\alpha_{1}}(z_{1,1}) \ v_{L_{B_l^{(1)}}(\Lambda_{0})}&\\
\nonumber
& \ \ \ \ \ \ \ \ \ \ \ \ \ \ \ \ \otimes \ldots \otimes&\\
\nonumber
\otimes &x_{n_{r^{(1)}_{l},l}^{(1)}\alpha_{l}}(z_{r_{l}^{(1)},l})\cdots  x_{n_{r^{(2)}_{l},l}^{(1)}\alpha_{l}}(z_{r_{l}^{(2)},l})\cdots 
  x_{n_{1,l}^{(1)}\alpha_{l}}(z_{1,l})x_{n^{(1)}_{r_{l-1}^{(1)},l-1}\alpha_{l-1}}(z_{r_{l-1}^{(1)},l-1}) \cdots     & \\ 
\nonumber
&   \ \ \ \ \ \ \ \ \ \ \ \ \ \ \ \ \ \ \ \ \ \ \ \ \ \cdots x_{n^{(1)}_{1,l-1}\alpha_{l-1}}(z_{1,l-1})\cdots x_{n_{r^{(1)}_{1},1}^{(1)}\alpha_{1}}(z_{r_{1}^{(1)},1})\cdots     x_{n{_{1,1}^{(1)}\alpha_{1}}}(z_{1,1}) \ v_{L_{B_l^{(1)}}(\Lambda_{0})},&
\end{align}
where $\text{C} \in \mathbb{C}^{*}$. 
\par From the above considerations follows that the projection of the monomial vector $bv_{L_{B_l^{(1)}}(k\Lambda_{0})}$, where $b\in B_{W_{L_{B_l^{(1)}}(k\Lambda_0)}}$ is a monomial 
\begin{equation}\label{eq:polB}
b=x_{n_{r_{l}^{(1)},l}\alpha_{l}}(m_{r_{l}^{(1)},l})\cdots      x_{n_{1,l}\alpha_{l}}(m_{1,l})\cdots x_{n_{r_{1}^{(1)},1}\alpha_{1}}(m_{r_{1}^{(1)},1})\cdots  x_{n_{1,1}\alpha_{1}}(m_{1,1})
\end{equation}
colored with color-type $(r_{l},\ldots, r_{1}),$ charge-type $\mathfrak{R}'$ and dual-charge-type $\mathfrak{R}$, 
 is a coefficient of the projection of the generating function  (\ref{projekcijaB}) which we denote as $$\pi_{\mathfrak{R}}bv_{L_{B_l^{(1)}}(k\Lambda_0)}.$$ 
\begin{ax}\label{S62LB}
Here we note, that if $\bar{b}\in B_{W_{L_{B_l^{(1)}}(k\Lambda_0)}}$ is monomial such that it  
is of charge-type $$(\bar{n}_{{\bar{r}}_{l}^{(1)},l}, \ldots ,  \bar{n}_{1,l};\ldots; \bar{n}_{{\bar{r}}_{1}^{(1)},1}, \ldots , \bar{n}_{1,1}),$$ dual-charge-type 
$ \overline{\mathfrak{R}} =\left( {\bar{r}}_{l}^{(1)}, \ldots ,{\bar{r}}_{l}^{(2k)};\ldots; {\bar{r}}_{1}^{(1)}, \ldots , {\bar{r}}_{1}^{(k)}\right)$ and such that 
\begin{equation*}\mathfrak{R}'< \overline{\mathfrak{R}'} ,
\end{equation*} then, from the definition of projection, follows that 
\begin{equation*}
\pi_{\mathfrak{R}}\bar{b}v_{L_{B_l^{(1)}}(k\Lambda_0)}=0.
\end{equation*}
This argument we will use in our proof of linear independance.
\end{ax}

\subsubsection{\textbf{A coefficient of an intertwining operator}}\label{ss:intertB} 
Denote by $Y (\cdot, z)$ the vertex operator which determines the structure of $L_{B_l^{(1)}}(\Lambda_0)$-module $L_{B_l^{(1)}}(\Lambda_1)$. We shall
use the coefficient of intertwining operator $I(\cdot , z)$ of type \begin{equation*} 
\binom{L_{B_l^{(1)}}(\Lambda_1)}{L_{B_l^{(1)}}(\Lambda_1) \ L_{B_l^{(1)}}(\Lambda_0)},
\end{equation*}
defined by  
\begin{align}\label{al:S33B}
I(w,z)v=\exp(zL(-1))Y(v,-z)w, \ \  w \in L_{B_l^{(1)}}(\Lambda_1), \ v \in L_{B_l^{(1)}}(\Lambda_0)
\end{align}
(cf. \cite{FHL}). 
If we use the commutator formula  
\begin{equation*} 
\left[x(m),I(v_{L_{B_l^{(1)}}(\Lambda_1)},z)\right]=\sum_{j\geq 0} \binom {m}{j}z^{m-j}I(x(j)v_{L_{B_l^{(1)}}(\Lambda_1)},z)
\end{equation*}
(cf. (2.13) in \cite{Li2}), where $x_{\alpha_i}(m) \in \widehat{\mathfrak{g}}$ for $\alpha_i \in \Pi$, we have:
\begin{equation*}
\left[x_{\alpha_i}(m), I(v_{L_{B_l^{(1)}}(\Lambda_1)},z)\right]=0.
\end{equation*}
\par We define the following coefficient of an intertwining operator 
\begin{equation*} 
A_{\omega_1}=\text{Res}_z \ z^{-1}I(v_{L_{B_l^{(1)}}(\Lambda_1)}, z)
\end{equation*}
and by (\ref{al:S33B}), we have
\begin{equation}\label{eq:S34B}
A_{\omega_1}v_{L_{B_l^{(1)}}(\Lambda_0)}=v_{L_{B_l^{(1)}}(\Lambda_1)}.
\end{equation}
Let $s \leq k$. We consider the operator on $L_{B_l^{(1)}}(\Lambda_0) \otimes \cdots \otimes L_{B_l^{(1)}}(\Lambda_0)$ defined as
\begin{equation}\label{eq:S35B}
A_s=1\otimes\cdots \otimes  A_{\omega_{1}} \otimes \underbrace{1 \otimes \cdots \otimes 1}_{s-1 \ \text{factors}}.
\end{equation}
If we act with this operator on the vector $v_{L_{B_l^{(1)}}(k\Lambda_0)}=v_{ L_{B_l^{(1)}}(\Lambda_0)}\otimes \cdots \otimes v_{L_{B_l^{(1)}}(\Lambda_0)}$, it follows from (\ref{eq:S34B}):
\begin{equation}\label{eq:S36B}
A_s(v_{L_{B_l^{(1)}}(k\Lambda_0)})=v_{ L_{B_l^{(1)}}(\Lambda_0)}\otimes
\cdots \otimes v_{ L_{B_l^{(1)}}(\Lambda_0)}\otimes v_{ L_{B_l^{(1)}}(\Lambda_1)}\otimes \underbrace{v_{ L_{B_l^{(1)}}(\Lambda_0)}\otimes \cdots\otimes v_{ L_{B_l^{(1)}}(\Lambda_0)}}_{s-1 \ \text{factors}}.
\end{equation}
\par Set $b\in B_{W_{L_{B_l^{(1)}}(k\Lambda_0)}}$ as in (\ref{eq:polB}). It follows that \begin{equation*}A_s\pi_{\mathfrak{R}}bv_{L_{B_l^{(1)}}(k\Lambda_0)}
\end{equation*}
is the coefficient of  
\begin{equation*}
A_s\pi_{\textsl{\emph{R}}}x_{n_{r_{2}^{(1)},2}\alpha_{2}}(z_{r_{2}^{(1)},2})\cdots x_{s\alpha_{1}}(z_{1,1})v_{L_{B_l^{(1)}}(k\Lambda_0)}.
\end{equation*}
From (\ref{eq:S36B}), it follows that operator $A_{\omega_{1}}$ acts only on the $s$-th tensor factor from the right: 
\begin{equation*}\otimes  x_{n^{(s)}_{r^{(2s-1)}_{l},l}\alpha_{l}}(z_{r_{l}^{(2s-1)},l})\cdots   x_{n^{(s)}_{r^{(2s)}_{l},l}\alpha_{l}}(z_{r^{(2s)}_{l},l})\cdots  x_{n^{(s)}_{1,l}\alpha_{l}}(z_{1,l})\end{equation*}
\begin{equation*}x_{n^{(s)}_{r^{(s)}_{1},1}\alpha_{1}}(z_{r_{1}^{(s)},1})\cdots  x_{\alpha_{1}}(z_{1,1})v_{L_{B_l^{(1)}}(\Lambda_{0})}\otimes,\end{equation*}
where $ 0 \leq  n^{(s)}_{p,i} \leq 1$, for $1\leq  p \leq  r^{(s)}_{i}$ and $ 0 \leq  n^{(s)}_{p,l} \leq 2$, for $1\leq p \leq    r^{(2s-1)}_{l}$ (see (\ref{projekcijaB})). Since $A_{\omega_{1}}$ commutes with the generating functions, in the $s$-th tensor facor from the right, we have
\begin{equation*} \cdots \otimes x_{n^{(s)}_{r^{(2s-1)}_{l},l}\alpha_{l}}(z_{r_{l}^{(2s-1)},l})\cdots   x_{n^{(s)}_{r^{(2s)}_{l},l}\alpha_{l}}(z_{r^{(2s)}_{l},l})\cdots  x_{n^{(s)}_{1,l}\alpha_{l}}(z_{1,l})\end{equation*}
\begin{equation}\label{TanjaB}x_{n^{(s)}_{r^{(s)}_{1},1}\alpha_{1}}(z_{r_{1}^{(s)},1})\cdots  x_{\alpha_{1}}(z_{1,1})v_{L_{B_l^{(1)}}(\Lambda_{1})}\otimes \cdots.
 \end{equation}

\subsubsection{\textbf{Simple current operator \texorpdfstring{$e_{\omega_{1}}$}{eomega1}}}\label{ss:currB}
In the same way as in \cite{Bu} in the proof of linear independence, we use simple current operators $e_{\omega_{1}}$ on level $1$ standard modules for $B_l^{(1)}$, $l\geq 2$:
\begin{equation*}
e_{\omega_{1}}:L_{B_l^{(1)}}(\Lambda_0)\rightarrow L_{B_l^{(1)}}(\Lambda_1),
\end{equation*}
associated to $\omega_1 \in \mathfrak{h}$,  
which are uniquely characterized by
their action on the highest weight vector 
\begin{equation}\label{nekakoB}
e_{\omega_{1}}{v}_{L_{B_l^{(1)}}(\Lambda_{0})}=v_{L_{B_l^{(1)}}(\Lambda_{1})}
\end{equation}
and by their commutation relations 
\begin{equation}\label{S341B}
x_{\alpha}(z)e_{\omega_{1}}=e_{\omega_{1}}z^{\alpha(\omega_{1})}x_{\alpha}(z),
\end{equation}
for all $\alpha \in R$, or, written by components,
\begin{equation}\label{S342B}
x_{\alpha}(m)e_{\omega_{1}}=e_{\omega_{1}}x_{\alpha}(m+\alpha(\omega_{1})),
\end{equation}
for all $\alpha \in R$ and $m \in \mathbb{Z}$ (cf. \cite{DLM}, \cite{Li2}).
\par Let $s \leq k$. We define the linear bijection
\begin{equation}\label{S657B}
B_s=1\otimes\ldots \otimes 1\otimes e_{\omega_{1}} \otimes \underbrace{1 \otimes \ldots \otimes 1}_{s-1 \ \text{factors}}.
\end{equation}
If we act with this operator (\ref{S657B}) on the vector $v_{L_{B_l^{(1)}}(k\Lambda_0)}=v_{ L_{B_l^{(1)}}(\Lambda_0)}\otimes \cdots \otimes v_{L_{B_l^{(1)}}(\Lambda_0)}$, we get 
\begin{align*} 
B_s(v_{L_{B_l^{(1)}}(k\Lambda_0)})
=v_{ L_{B_l^{(1)}}(\Lambda_0)}\otimes
\cdots \otimes v_{ L_{B_l^{(1)}}(\Lambda_0)}\otimes v_{ L_{B_l^{(1)}}(\Lambda_1)}\otimes \underbrace{v_{ L_{B_l^{(1)}}(\Lambda_0)}\otimes \cdots\otimes v_{ L_{B_l^{(1)}}(\Lambda_0)}}_{s-1 \ \text{factors}}.\end{align*}
Now it follows that in (\ref{TanjaB}) we can commute $B_s$ to the left and obtain
\begin{equation*}
\cdots \otimes x_{n^{(s)}_{r^{(2s)}_{l},l}\alpha_{l}}(z_{r_{l}^{(2s-1)},l})\cdots           x_{n^{(s)}_{r^{(2s)}_{l},l}\alpha_{l}}(z_{r^{(2s)}_{l},l})\cdots  x_{n^{(s)}_{1,l}\alpha_{l}}(z_{1,l})\end{equation*}
\begin{equation*}
\cdots x_{n^{(s)}_{r_{1}^{(k)},1}\alpha_{1}}(z_{r_{1}^{(k)},1})z_{r_{1}^{(k)},1}\cdots x_{\alpha_{1}}(z_{1,1})z_{1,1}v_{ L_{B_l^{(1)}}(\Lambda_0)}\otimes \cdots.\end{equation*}
By taking the corresponding coefficients, we have
\begin{equation*}
A_s\pi_{\mathfrak{R}}bv_{L_{B_l^{(1)}}(k\Lambda_0)}=B_s\pi_{\mathfrak{R}}b^{+}v_{L_{B_l^{(1)}}(k\Lambda_0)}
\end{equation*}
where the monomial $b^{+}$:
\begin{align*}
b^{+}=b^+(\alpha_{l})\cdots b^{+}(\alpha_{1}),
\end{align*}
is such that 
\begin{align*}
b^+(\alpha_{i})&=b(\alpha_{i}), \ \ 2 \leq i \leq l&\\
b^+(\alpha_{1})&=x_{n_{r_{1}^{(1)},1}\alpha_{1}}(m_{r_{1}^{(1)},1}+1)\cdots x_{s\alpha_{1}}(m_{1,1}+1)&\\
&=x_{n_{r_{1}^{(1)},1}\alpha_{1}}(m^{+}_{r_{1}^{(1)},1})\cdots x_{s\alpha_{1}}(m^{+}_{1,1}).&
\end{align*}

\subsubsection{\textbf{Operator \texorpdfstring{$e_{\alpha_{i}}$}{ealphai}}}\label{S66B}
For every simple root $\alpha_i \in \Pi$, $1 \leq i \leq l$, we define on the level $1$ standard module $L_{B_l^{(1)}}(\Lambda_0)$, the ``Weyl group translation'' operator  $e_{\alpha_i}$ by
\begin{equation}\label{weyl}
 e_{\alpha_i}=\exp\  x_{-\alpha_i}(1)\exp\  (- x_{\alpha_i}(-1))\exp\  x_{-\alpha_i}(1) \exp\ x_{\alpha_i}(0)\exp\   (-x_{-\alpha_i}(0))\exp\ x_{\alpha_i}(0),\end{equation}
(cf. \cite{K}). Using (\ref{weyl}) we see that
\begin{lem}\label{S662B} Let $i$, ($1\leq i \leq l-1$) be fixed. For every $i' \neq i, i+1$, we have:
\begin{itemize}
	\item [a)]  $ e_{\alpha_1}v_{_{B_l^{(1)}}(\Lambda_0)}=-x_{\alpha_{i}}(-1)v_{_{B_l^{(1)}}(\Lambda_{0})}$;
\item [b)]  
$x_{\alpha_i}(z)e_{\alpha_i}=z^2e_{\alpha_i}x_{\alpha_i}(z)$;
	\item [c)]  
$x_{\alpha_{i+1}}(z)e_{\alpha_i}=z^{-1}e_{\alpha_i}x_{\alpha_{i+1}}(z);$
\item [d)] 
$ x_{\alpha_{i'}}(z)e_{\alpha_i}=e_{\alpha_i}x_{\alpha_{i'}}(z)$.
\end{itemize}
\end{lem}
\begin{flushright}
$\square$
\end{flushright}
Set 
\begin{equation*}
1\otimes\cdots  \otimes 1\otimes \underbrace{e_{\alpha_{i}} \otimes e_{\alpha_{i}} \otimes \cdots \otimes e_{\alpha_{i}}}_{s \ \text{factors}}
\end{equation*}
where $s \leq k$. 
From Lemma \ref{S662B} a), it now follows
\begin{align*} 
&\left(1\otimes\cdots  \otimes 1\otimes e_{\alpha_{i}} \otimes e_{\alpha_{i}}\otimes  \cdots \otimes e_{\alpha_{i}}\right)v_{L_{B_l^{(1)}}(k\Lambda_0)}&\\
=&(-1)^sv_{L_{B_l^{(1)}}(\Lambda_{0})}\otimes
 \cdots\otimes v_{L_{B_l^{(1)}}(\Lambda_{0})}\otimes&\\
 & \ \ \ \ \ \ \ \ \ \ \ \ \underbrace{x_{\alpha_{i}}(-1)v_{L_{B_l^{(1)}}(\Lambda_{0})}\otimes      x_{\alpha_{i}}(-1)v_{L_{B_l^{(1)}}(\Lambda_{0})}\otimes \cdots \otimes  x_{\alpha_{i}}(-1)v_{L_{B_l^{(1)}}(\Lambda_{0})}}_{s \ \text{factors}}.&
\end{align*}
Let $1\leq i \leq l-1$ be fixed and let $b$ be a monomial   
\begin{align}\label{S6610B}
b&=b(\alpha_{i+1})b(\alpha_{i})x_{s\alpha_{i}}(-s)&\\
\nonumber
&=x_{n_{r^{(1)}_{{i+1}},{i+1}}\alpha_{{i+1}}}(m_{r^{(1)}_{{i+1}},{i+1}})\cdots     x_{n_{1,{i+1}}\alpha_{{i+1}}}(m_{1,{i+1}})&\\
\nonumber
& \ \ \ \ \ \ \ \ \ \ \ \ \ x_{n_{r^{(1)}_{i},i}\alpha_{i}}(m_{r^{(1)}_{i},i})\cdots  x_{n_{2,i}\alpha_{i}}(m_{2,i})x_{s\alpha_{i}}(-s),&
\end{align}
of dual-charge-type 
\begin{equation*}
\mathfrak{R}=\left(r^{(1)}_{i+1},\ldots, r^{(p)}_{i+1} ;r^{(1)}_{i},\ldots, r_i^{(s)},0 \ldots, 0\right),
\end{equation*}
where $p=k$ if $i+1 <l$ or $p=2k$ if $i+1=l$.
\par Assume that $i < l-1$. The situation of $i=l-1$ is similar to the case which is considered in \cite{Bu}. 
Let $\pi_{\mathfrak{R}}$ be the projection of principal subspace $W_{L_{B_l^{(1)}}(\Lambda_{0})}\otimes \cdots\otimes W_{L_{B_l^{(1)}}(\Lambda_{0})}$ on the vector space  
\begin{equation*}
{W_{L_{B_l^{(1)}}(\Lambda_{0})}}_{(r_{i+1}^{(k)};0)}\otimes \cdots \otimes {W_{L_{B_l^{(1)}}(\Lambda_{0})}}_{( r_{i+1}^{(s)}; r_{i}^{(s)})}\otimes \cdots\otimes {W_{L_{B_l^{(1)}}(\Lambda_{0})}}_{( r_{i+1}^{(1)};r_{i}^{(1)})}. \end{equation*}
The projection \begin{equation*}
\pi_{\mathfrak{R}}b\left(v_{L_{B_l^{(1)}}(\Lambda_{0})}\otimes \cdots\otimes v_{L_{B_l^{(1)}}(\Lambda_{0})}\right)\end{equation*}
of the monomial vector $b\left(v_{L_{B_l^{(1)}}(\Lambda_{0})}\otimes \cdots\otimes v_{L_{B_l^{(1)}}(\Lambda_{0})}\right)$ is a coefficient of the generating function
\begin{equation*}\pi_{\mathfrak{R}}x_{n_{r_{{i+1}}^{(1)},i+1}\alpha_{i+1}}(z_{r_{i+1}^{(1)},i+1})\cdots
 x_{n_{1,i+1}\alpha_{i+1}}(z_{1,i+1}) x_{n_{r_{i}^{(1)},i}\alpha_{i}}(z_{r_{i}^{(1)},i}) 
\cdots  x_{n_{2,i}\alpha_{i}}(z_{2,i})\end{equation*}
\begin{equation*}\left(v_{L_{B_l^{(1)}}(\Lambda_{0})}\otimes \cdots\otimes v_{L_{B_l^{(1)}}(\Lambda_{0})}\otimes x_{\alpha_{i}}(-1)v_{L_{B_l^{(1)}}(\Lambda_{0})}\otimes \cdots \otimes  x_{\alpha_{i}}(-1)v_{L_{B_l^{(1)}}(\Lambda_{0})}\right)\end{equation*}
\begin{equation*}=C x_{n^{(k)}_{r^{(k)}_{i+1},i+1}\alpha_{i+1}}(z_{r_{i+1}^{(k)},i+1})\cdots    x_{n^{(k)}_{1,i+1}\alpha_{i+1}}(z_{1,i+1}) v_{L_{B_l^{(1)}}(\Lambda_{0})}\end{equation*}
\begin{equation*}\otimes \cdots \otimes \end{equation*}
\begin{equation*}\otimes x_{n^{(s)}_{r^{(2s-1)}_{i+1},i+1}\alpha_{i+1}}(z_{r_{i+1}^{(2s-1)},i+1}) \cdots 
 x_{n^{(s)}_{r^{(2s)}_{i+1},i+1}\alpha_{i+1}}(z_{r^{(2s)}_{i+1},i+1})\cdots  x_{n^{(s)}_{1,i+1}\alpha_{i+1}}(z_{1,i+1})\end{equation*}
\begin{equation*} x_{n^{(s)}_{r^{(s)}_{i},i}\alpha_{i}}(z_{r_{i}^{(s)},i})\cdots    x_{n^{(s)}_{2,i}\alpha_{i}}(z_{2,i})e_{\alpha_{i}}v_{L_{B_l^{(1)}}(\Lambda_{0})}\end{equation*}
\begin{equation*}\otimes \cdots \otimes \end{equation*}
\begin{equation*}\otimes x_{n_{r^{(1)}_{i+1},i+1}^{(1)}\alpha_{i+1}}(z_{r_{2}^{(1)},i+1})\cdots 
 x_{n_{r^{(2)}_{i+1},i+1}^{(1)}\alpha_{i+1}}(z_{r_{i+1}^{(2)},i+1})\cdots x_{n_{2,i+1}^{(1)}\alpha_{2}}(z_{2,i+1})
x_{n_{1,i+1}^{(1)}\alpha_{i+1}}(z_{1,i+1})\end{equation*}
\begin{equation*}x_{n_{r^{(1)}_{i},i}^{(1)}\alpha_{i}}(z_{r_{i}^{(1)},i})\cdots   x_{n{_{2,i}^{(1)}\alpha_{i}}}(z_{2,i})e_{\alpha_{i}}v_{L_{B_l^{(1)}}(\Lambda_{0})},\end{equation*}
where $C \in \mathbb{C}^*$ (see (\ref{projekcijaB})).  We shift operator $1\otimes\ldots \otimes e_{\alpha_{i}} \otimes e_{\alpha_{i}} \otimes \ldots \otimes        e_{\alpha_{i}}$ all the way to the left using commutation relations b), c) in Lemma \ref{S662B} \begin{equation*}
(1 \otimes \cdots \otimes 1 \otimes e_{\alpha_{i}} \otimes e_{\alpha_{i}} \otimes \cdots \otimes  e_{\alpha_{i}})\pi_{\mathfrak{R}'} b'\left(v_{L_{B_l^{(1)}}(\Lambda_{0})}\otimes \cdots\otimes v_{L_{B_l^{(1)}}(\Lambda_{0})}\right), 
\end{equation*} 
where 
\begin{equation*}
\mathfrak{R}'=\left(r^{(1)}_{i+1},\ldots, r^{(s)}_{i+1};r^{(1)}_{i}-1,\ldots, r_i^{(s)}-1\right)
\end{equation*}
and
\begin{align*}
b'&= b'(\alpha_{i+1})b'(\alpha_{i})&\\
&= \cdots x_{n_{1,i+1}\alpha_{i+1}}(m_{1,i+1}-n^{(1)}_{1,i+1}-\cdots-n^{(s)}_{1,i+1})&\\
& \ \ \ \ \ \ \ \ \ \ \ \ \ \ \ \ \ \ \ \ \ \ \ \ \ \ \  x_{n_{r^{(1)}_{i}\alpha_{i}}}(m_{r^{(1)}_{i},i}+2n_{r^{(1)}_{i}})\cdots   x_{n_{2,i}\alpha_{i}}(m_{2,i}+2n_{2,i})&\\
&= x_{n_{r^{(1)}_{i+1},i+1}\alpha_{i+1}}(m'_{r^{(1)}_{i+1},i+1})\cdots 
 x_{n_{1,i+1}\alpha_{i+1}}(m'_{1,i+1})&\\
 &\ \ \ \ \ \ \ \ \ \ \ \ \ \ \ \ \ \ \ \ \ \ \ \ \ \ \  x_{n_{r^{(1)}_{i}\alpha_{i}}}(m'_{r^{(1)}_{1},i})\cdots x_{n_{2,i}\alpha_{i}}(m'_{2,i}).&
\end{align*}
In the proof of linear independence, we use the following proposition: 
\begin{prop}\label{S66PB}
Let $b$ (\ref{S6610B}) be an element of the set $B_{W_{L_{B_l^{(1)}}(k\Lambda_{0})}}$. Then the monomial $b'$ is an element of the set $B_{W_{L_{B_l^{(1)}}(k\Lambda_{0})}}$. 
\end{prop}
\begin{proof}
The proposition follows by considering the possible situation for $n_{p,i}$, $2\leq p \leq r^{(1)}_{i}$, and $n_{p,i+1}$, $ 1\leq p \leq r^{(1)}_{i+1}$, from which it follows that $m_{p,i}$ comply the defining conditions of the set $B_{W_{L_{B_l^{(1)}}(k\Lambda_{0})}}$. As before we will assume that $i<l-1$, since for the $i=l-1$ the argument is similar as in the case of affine Lie algebra $B_2^{(1)}$ (see \cite{Bu}).
\begin{enumerate}
\item For $n_{p,i}=\bar{s}\leq s$, we have
\begin{align*} 
m'_{p,i}&=m_{p,i}+2\bar{s}\\
&\leq -\bar{s}- 2 (p-1)\bar{s} +2\bar{s}\\
&=-\bar{s}- 2 (p-2)\bar{s}
\end{align*}
and 
\begin{align*}
m'_{p+1,i}&=m_{p+1,i}+2\bar{s}\\
& \leq  -2\bar{s}+m_{p,i}+2\bar{s}\\
&= m'_{p,i} -2\bar{s} 
\ \ \ \ \ \ \ \text{for} \ \ \ n_{p+1,i}=n_{p,i}.
\end{align*}

\item For $n_{p,i+1}\leq s$, we have
\begin{align*} 
m'_{p,i+1}&=m_{p,i+1}-n_{p,i+1}\\
&\leq -n_{p,i+1} - \sum_{p>p'>0} 2 \ \text{min} \left\{n_{p,i+1},n_{p',i+1}\right\}+  
\sum^{r^{(1)}_{i}}_{q=1}\text{min} \left\{n_{p,i+1},n_{q,i}\right\}-n_{p,i+1}\\
 \\
&=-n_{p,i+1} - \sum_{p>p'>0} 2 \ \text{min} \left\{n_{p,i+1},n_{p',i+1}\right\}+  
\sum^{r^{(1)}_{i}}_{q=2}\text{min} \left\{n_{p,i+1},n_{q,i}\right\} 
\end{align*}
and 
\begin{align*}
m'_{p+1,i+1}&=m_{p+1,i+1}-n_{p,i+1}\\
& \leq  m_{p,i+1}-2n_{p,i+1}-n_{p,i+1}\\
&= m'_{p,i+1} -2n_{p,i+1} 
\ \ \ \ \ \ \ \text{for} \ \ \ n_{p+1,i+1}=n_{p,i+1}.
\end{align*}

\item For $n_{p,i+1}>s$, we have:
\begin{align*}
m'_{p,i+1}&= m_{p,i+1}-s\\
 &\leq  -n_{p,i+1}  
 - \sum_{p>p'>0} 2 \ \text{min} \left\{n_{p,i+1},n_{p',i+1}\right\}+ 
 \sum^{r^{(1)}_{i}}_{q=1}\text{ min} \left\{n_{p,i+1},n_{q,i}\right\}-s\\
 &=  -n_{p,i+1}
- \sum_{p>p'>0} 2 \ \text{min} \left\{n_{p,i+1},n_{p',i+1}\right\}+ 
\sum^{r^{(1)}_{ i}}_{q=2}\text{ min} \left\{n_{p,i+1},n_{q,i}\right\}
\end{align*}
and
\begin{align*}
m'_{p+1,i+1} &= m_{p+1,i+1} -s \\
&\leq  m_{p,i+1} -2n_{p,i+1}-s\\
&= m'_{p,i+1} -2n_{p,i+1}
 \ \ \ \ \ \text{for} \ \ \ n_{p+1,i+1}=n_{p,i+1}.
\end{align*}
\end{enumerate}
\end{proof}

\subsubsection{\textbf{The proof of linear independence}}\label{S67B}
By Proposition \ref{prop:S22B} the set $\mathfrak{B}_{W_{L_{B_l^{(1)}}(k\Lambda_{0})}}$ of monomial vectors $bv_{L_{B_l^{(1)}}(k\Lambda_{0})}$ spans $W_{L_{B_l^{(1)}}(k\Lambda_{0})}$. We prove linear independence of this set by induction on $l$ and charge-type $\mathfrak{R}'$ of monomials $b \in B_{W_{L_{B_l^{(1)}}(k\Lambda_{0})}}$. Linear independence for the case $l=2$ is proved in \cite{Bu}.
\begin{ax} The idea of proof is that for the \enquote{minimum} quasi-monomial vector of dual-charge-type $\mathfrak{R}$ in a given subset of  $\mathfrak{B}_{W_{L_{B_l^{(1)}}(k\Lambda_{0})}}$, we define the projection $\pi_{\mathfrak{R}}$, which \enquote{kills} all monomial vectors higher in the linear lexicographic ordering \enquote{$<$} (see Remark \ref{S62LB}).
\end{ax} 
\par We fix $1< i \leq l$ and the dual-charge-type 
\begin{equation}\label{malidualtype}
\mathfrak{R}=\left(r_{l}^{(1)}, \ldots, r_{l}^{(2k)}; r_{l-1}^{(1)}, \ldots, r_{l-1}^{(k)}; \ldots ;r_{i}^{(1)}, \ldots, r_{i}^{(k)}\right) ,
\end{equation}
\begin{equation*}
r_{l}^{(1)}\geq \ldots \geq r_{l}^{(2k)},
\end{equation*}
\begin{equation*}
r_{l-1}^{(1)}\geq \ldots \geq r_{l-1}^{(k)},
\end{equation*}
\begin{equation*} 
\cdots
\end{equation*}
\begin{equation*}
r_{i}^{(1)}\geq \ldots \geq r_{i}^{(k)}.
\end{equation*} 
Denote by $\mathfrak{A}_{\mathfrak{R}} \subset \mathfrak{B}_{W_{L_{B_{l}^{(1)}}(k\Lambda_{0})}}$ the set of monomial vectors $bv_{L_{B_{l}^{(1)}}(k\Lambda_0)}$, where monomials $b$ are of dual-charge-type (\ref{malidualtype}) and the corresponding charge-type
\begin{equation*}
\mathfrak{R}'=\left(n_{r_{l}^{(1)},l}, \ldots, n_{1,l};n_{r_{l-1}^{(1)},l-1}, \ldots, n_{1,l-1};\ldots ; n_{r_{i}^{(1)},i}, \ldots, n_{1,i}\right) ,
\end{equation*}
\begin{equation*} 
n_{r_{i}^{(1)},l}\leq \ldots \leq n_{1,l}\leq 2k, 
\end{equation*}
\begin{equation*} 
n_{r_{l-1}^{(1)},l-1}\leq \ldots \leq n_{1,l-1}\leq k,
\end{equation*}
\begin{equation*} 
\cdots
\end{equation*}
\begin{equation*} 
n_{r_{i}^{(1)},i}\leq \ldots \leq n_{1,i}\leq k.
\end{equation*}
Note, that monomials $b \in B_{W_{L_{B_{l}^{(1)}}(k\Lambda_{0})}}$   
\begin{align}\label{novoBldanas}
&b= b(\alpha_l)b(\alpha_{l-1})\cdots b(\alpha_{i})=&\\
\nonumber
&=x_{n_{r_{l}^{(1)},l}\alpha_{l}}(m_{r_{l}^{(1)},l})\cdots  x_{n_{1,l}\alpha_{l}}(m_{1,l})&\\
\nonumber
&\ \ \ \ \ \ \ \ \ \ \ \ \ \ \ \ \ \ \ \ \ \ \ \ x_{n_{r_{l-1}^{(1)},l-1}\alpha_{l-1}}(m_{r_{l-1}^{(1)},l-1})\cdots  x_{n_{1,l-1}\alpha_{l-1}}(m_{1,l-1})\cdots &\\
\nonumber
&\ \ \ \ \ \ \ \ \ \ \ \ \ \ \ \ \ \ \ \ \ \ \ \ \ \ \ \ \ \ \ \ \ \ \ \ \ \cdots x_{n_{r_{i}^{(1)},i}\alpha_{i}}(m_{r_{i}^{(1)},i})\cdots  x_{n_{1,i}\alpha_{i}}(m_{1,i}),&
\end{align} 
of the charge-type $\mathfrak{R}'$ and dual-charge type $\mathfrak{R}$ 
can be realised as elements of the principal subspace in the case of the affine Lie algebra of type $B_{l-i+1}^{(1)}$. 
\par Under consideration at the subsection \ref{ss:projB}, the default dual-charge-type $\mathfrak{R}$ determines the projection $\pi_{\mathfrak{R}}$ on the vector space
\begin{equation*}
{W_{L_{B_l^{(1)}}(\Lambda_{0})}}_{(\mu^{(k)}_{l};r_{l-1}^{(k)};\ldots; r_{i}^{(k)} )}\otimes \cdots \otimes {W_{L_{B_l^{(1)}}(\Lambda_{0})}}_{(\mu^{(1)}_{l};r_{l-1}^{(1)}; \ldots; r_{i}^{(1)} )} \subset W_{L_{B_l^{(1)}}(\Lambda_{0})}\otimes \cdots \otimes  W_{L_{B_l^{(1)}}(\Lambda_{0})}.
\end{equation*}
Since the restriction of $B_l^{(1)}$-module $L(\Lambda_0)$ on the subalgebra $B_{l-i+1}^{(1)}$ is a direct sum of the level one $B_{l-i+1}^{(1)}$-modules $L_{B_{l-i+1}^{(1)}}(\Lambda_0)$, with a highest weight vector $v_{L_{B_{l}^{(1)}}(\Lambda_0)}=v_{L_{B_{l-i+1}^{(1)}}(\Lambda_0)}$, it follows that
\begin{equation}\label{nekakoBl}
\pi_{\mathfrak{R}}bv_{L_{B_{l}^{(1)}}(k\Lambda_0)} \in {W_{L_{B_{l-i+1}^{(1)}}(\Lambda_{0})}}\otimes \cdots \otimes {W_{L_{B_{l-i+1}^{(1)}}(\Lambda_{0})}} \subset W_{L_{B_l^{(1)}}(\Lambda_{0})}\otimes \cdots \otimes  W_{L_{B_l^{(1)}}(\Lambda_{0})},
\end{equation}
where $W_{L_{B_{l-i+1}^{(1)}}(\Lambda_{0})}={W_{L(\Lambda_{0})}}_{0\alpha_1+\cdots +0\alpha_{i-1}}$ is a principal subspace of standard $B_{l-i+1}^{(1)}$-module $L_{B_{l-i+1}^{(1)}}(\Lambda_0)\subset L_{B_l^{(1)}}(\Lambda_{0})$. 

\par On (\ref{nekakoBl}) we can act with operators 
$A_{n_{1,i}}$, $B_{n_{1,i}}$ and $e_{\alpha_{i}}$ defined for vertex operator algebra $L_{B_{l-i+1}^{(1)}}(\Lambda_0)$, whose properties are described in subsections \ref{ss:intertB}, \ref{ss:currB} and \ref{S66B}. 
With these operators we \enquote{move} monomial vectors $\pi_{\mathfrak{R}}bv_{L_{B_{l}^{(1)}}(k\Lambda_0)}$ from one space to another until we get vectors of the form $\pi_{\mathfrak{R}}b(\alpha_l)b(\alpha_{l-1})v_{L_{B_{l}^{(1)}}(k\Lambda_0)} \in \pi_{\mathfrak{R}}\mathfrak{A}$. In \cite{Bu} has been proven that the set $\pi_{\mathfrak{R}}\mathfrak{A}$ of vectors $\pi_{\mathfrak{R}}b(\alpha_l)b(\alpha_{l-1})v_{L_{B_{l}^{(1)}}(k\Lambda_0)}$ is a linearly independent set.  
\par By using the previous observations, we can prove:
\begin{thm}\label{THMB}
The set $\mathfrak{B}_{W_{L_{B_{l}^{(1)}}(k\Lambda_{0})}}$ forms a basis for the principal subspace $W_{L_{B_{l}^{(1)}}(k\Lambda_{0})}\subset L_{B_{l}^{(1)}}(k\Lambda_0)$.
\end{thm}
\begin{proof}
Assume that we have
\begin{equation}\label{S6752Bl}
\sum_{a \in A}
c_{a}b_av_{L_{B_{l}^{(1)}}(k\Lambda_{0})}=0, 
\end{equation}
where $A$ is a finite non-empty set and 
\begin{align*}
b_a\in  B_{W_{L_{B_{l}^{(1)}}(k\Lambda_{0})}}.
\end{align*}
Assume that all $b_a$ are the same color-type $(r_{l}, \ldots,r_{1})$. Let $b$ be the smallest monomial in the linear lexicographic ordering \enquote{$<$}
\begin{align*} 
b&=b(\alpha_{l})\cdots b(\alpha_{2})b(\alpha_{1})&\\
&=x_{n_{r_{l}^{(1)},l}\alpha_{l}}(m_{r_{l}^{(1)},l})\cdots x_{n_{1,l}\alpha_{l}}(m_{1,l})\cdots x_{n_{r_{2}^{(1)},2}\alpha_{2}}(m_{r_{2}^{(1)},2})\cdots x_{n_{1,2}\alpha_{2}}(-m_{1,2})&\\
& \ \ \ \ \ \ \ \ \ \ \ \ \ \ \ \ \ x_{n_{r_{1}^{(1)},1}\alpha_{1}}(m_{r_{1}^{(1)},1})\cdots x_{n_{1,1}\alpha_{1}}(-j),
\end{align*}
of dual-charge-type 
\begin{equation*}
\mathfrak{R}=\left( r_{l}^{(1)},\ldots , r_{l}^{(2k)};\ldots;  r_{2}^{(1)},\ldots , r_{2}^{(k)};  r_{1}^{(1)},\ldots , r_{1}^{(n_{1,1})}\right),
\end{equation*}
and charge-type  
\begin{equation}\label{S675444Bl}
\mathfrak{R}'=\left(n_{r_{l}^{(1)},l}, \ldots , 
 n_{1,l}; \ldots ; n_{r_{2}^{(1)},2}, \ldots ,n_{1,2}; n_{r_{1}^{(1)},1}, \ldots ,n_{1,1}\right),
\end{equation}
such that $c_{a}\neq 0$. Then for every other monomial in (\ref{S6752Bl}) we have
\begin{equation*}m_{1,1}\geq -j.
\end{equation*}
Dual-charge-type $\mathfrak{R}$ determines projection
$\pi_{\mathfrak{R}}$ of $\underbrace{W_{L_{B_{l}^{(1)}}(\Lambda_{0})}\otimes \ldots\otimes W_{L_{B_{l}^{(1)}}(\Lambda_{0})}}_{k \ \text{factors}}$ on the vector space 
\begin{align*}
&{W_{L_{B_{l}^{(1)}}(\Lambda_{0})}}_{(\mu^{(k)}_{l};\ldots; r_{2}^{(k)}; 0)}\otimes \ldots  \otimes {W_{L_{B_{l}^{(1)}}(\Lambda_{0})}}_{(\mu^{(n_{1,1}+1)}_{l};\ldots; r_{2}^{(n_{1,1}+1)}; 0)}\otimes&\\ 
& \otimes {W_{L_{B_{l}^{(1)}}(\Lambda_{0})}}_{(\mu^{(n_{1,1})}_{l};\ldots; r_{2}^{(n_{1,1})};r_{1}^{(n_{1,1})})}\otimes  \cdots {W_{L_{B_{l}^{(1)}}(\Lambda_{0})}}_{(\mu^{(1)}_{l};\ldots;r_{2}^{(1)};r_{1}^{(1)})},&
\end{align*}
where 
\begin{equation*}
\mu^{(t)}_{l}=r^{(2t)}_{l}+ r^{(2t-1)}_{l}.
\end{equation*}
By Remark \ref{S62LB}, $\pi_{\mathfrak{R}}$ maps to zero all monomial vectors $b_av_{L_{B_{l}^{(1)}}(k\Lambda_{0})}$ such that $b_a$ has a larger charge-type in the linear lexicographic ordering \enquote{$<$} than (\ref{S675444Bl}). So, in (\ref{S6733Bl}) 
\begin{equation}\label{S6733Bl}
\sum_{a} c_a\pi_{\mathfrak{R}}b_{a}v_{L_{B_{l}^{(1)}}(k\Lambda_{0})}=0, 
\end{equation}
we have a projection of $b_{a}v_{L_{B_{l}^{(1)}}(k\Lambda_{0})}$, where $b_{a}$ are of charge-type (\ref{S675444Bl}). 
On (\ref{S6733Bl}), we act with
\begin{equation*}
A_{n_{1,1}}=1\otimes\ldots \otimes A_{\omega_{1}} \otimes \underbrace{1 \otimes \ldots \otimes 1}_{n_{1,1}-1 \ \text{factors}},
\end{equation*}
then, from \ref{ss:intertB} and \ref{ss:currB} follows
\begin{align*}
A_{n_{1,1}}\left(\sum_{a\in A}
c_{a}\pi_{\mathfrak{R}}b_av_{L_{B_{l}^{(1)}}(k\Lambda_{0})}\right)
=e_{n_{1,1}}\left(\sum_{a\in A}
c_{a}\pi_{\mathfrak{R}}b_a^{+}v_{L_{B_{l}^{(1)}}(k\Lambda_{0})}\right),
\end{align*} 
where 
\begin{equation*}
e_{n_{1,1}}=1\otimes\ldots \otimes
e_{\omega_{1}} \otimes \underbrace{1 \otimes \ldots \otimes 1}_{n_{1,1}-1 \ \text{factors}}. 
\end{equation*}
After leaving out the invertible operator $e_{n_{1,1}}$, we get 
\begin{equation*} 
\sum_{a} c_{a}\pi_{\mathfrak{R}}b_{a}^{+}v_{L_{B_{l}^{(1)}}(k\Lambda_{0})}=0,
\end{equation*}
where $b_{{a}}^{+} \in \mathfrak{A}_{\mathfrak{R}}\subset B_{W_{L_{B_{l}^{(1)}}(k\Lambda_{0})}}$ are the same charge-type as $b_{a}$ in (\ref{S6752Bl}). We act with $A_{n_{1,1}}$ and $e_{n_{1,1}}$ until $j$ becomes $-n_{1,1}$. Assume that after $n_{1,1,}-j$ steps we got
\begin{equation*}
\sum_{a}
c_{a}\pi_{\mathfrak{R}}b_{a}(\alpha_{l})\cdots b_{a}(\alpha_{l})\cdots b_{a}^{+}(\alpha_{1})x_{n_{1,1}\alpha_{1}}(-n_{1,1})v_{L_{B_{l}^{(1)}}(k\Lambda_{0})}=0,
\end{equation*}
where monomial $b_{a}^{+}(\alpha_{1})x_{n_{1,1}\alpha_{1}}(-n_{1,1})$ is of color $i=1$ and  
\begin{equation*}b_{\mathfrak{R}}(\alpha_{l})\cdots b_{\mathfrak{R}}^{+}(\alpha_{1})x_{n_{1,1}\alpha_{1}}(-n_{1,1})v_{L_{B_{l}^{(1)}}(k\Lambda_{0})} \in \mathfrak{A}_{\mathfrak{R}}.\end{equation*} 
Now, from the subsection \ref{S66B} follows 
\begin{align*}
&\pi_{\mathfrak{R}}b(\alpha_{l})\cdots b(\alpha_{2})b^{+}(\alpha_{1})x_{n_{1,1}\alpha_{1}}(-n_{1,1})v_{L_{B_{l}^{(1)}}(k\Lambda_{0})}&\\
&=(1\otimes\cdots 1\otimes  e_{\alpha_{1}}\otimes e_{\alpha_{1}} \cdots \otimes   e_{\alpha_{1}})b'(\alpha_{2})b'(\alpha_{1})v_{L_{B_{l}^{(1)}}(k\Lambda_{0})},&
\end{align*} where $b(\alpha_{l})\cdots b'(\alpha_{2})b'(\alpha_{1})$ does not have a quasi-particle of charge $n_{1,1}$.  
Monomial $b(\alpha_{l})\cdots b'(\alpha_{2})b'(\alpha_{1})$ is of dual-charge-type
\begin{equation*}
\mathfrak{R}^{-}=\left( r_{l}^{(1)}, \ldots ,r_{l}^{(2k)};\ldots; r_{2}^{(1)}, \ldots ,r_{2}^{(k)}; r_{1}^{(1)}-1, \ldots , r_{1}^{(n_{1,1})}-1\right),
 \end{equation*}
and charge-type
\begin{equation*}
\left(n_{r_{l}^{(1)},l}, \ldots , n_{1,l};\ldots; n_{r_{2}^{(1)},2}, \ldots , n_{1,2}; n_{r_{1}^{(1)},1}, \ldots 
,n_{2,1}\right),\end{equation*}
such that\begin{align*}
&\left(n_{r_{l}^{(1)},l}, \ldots , n_{1,l};\ldots; n_{r_{2}^{(1)},2}, \ldots , n_{1,2}; n_{r_{1}^{(1)},1}, \ldots 
,n_{2,1}\right) <&\\
&< \left(n_{r_{l}^{(1)},l}, \ldots , n_{1,l};\ldots; n_{r_{2}^{(1)},2}, \ldots , n_{1,2}; n_{r_{1}^{(1)},1}, \ldots 
,n_{2,1},n_{1,1}\right).&
\end{align*} 
From Proposition \ref{S66PB}, it follows that with the described process, we get elements from the set $\mathfrak{B}_{W_{L_{B_{l}^{(1)}}(k\Lambda_0)}}$. We continue with the described algorithm, until we get monomial \enquote{colored} only with colors $i=l$ and $i=l-1$. Thus, under the consideration at the beginning of this subsection, it follows $c_{a}=0$.
\end{proof}

\subsection{Characters of the principal subspace \texorpdfstring{$W_{L_{B_{l}^{(1)}}(k\Lambda_{0})}$}{WLBl(1)(kLambda0)}}
We use the following expressions (\ref{S710B}),(\ref{S713B}), (\ref{S711B}), and (\ref{S712B}) to determine the character of $W_{L_{B_{l}^{(1)}}(k\Lambda_{0})}$. These expressions can be easy proved by using induction on the level $k \in \mathbb{N}$ of the standard module $L_{B_{l}^{(1)}}(k\Lambda_0)$.
\begin{lem}\label{S7L1B}
For the given color-type $(r_{l},r_{l-1}, \ldots , r_{2}, r_{1})$, charge-type      
\begin{equation*}\left(n_{r_{l}^{(1)},l}, \ldots, n_{1,l};n_{r_{l-1}^{(1)},l-1}, \ldots, n_{1,l-1};\ldots;n_{r_{2}^{(1)},2}, \ldots, n_{1,2}; n_{r_{1}^{(1)},1}, \ldots , n_{1,1}\right)\end{equation*} and dual-charge-type   
\begin{equation*}
\left(r_{l}^{(1)},r_{l}^{(2)},\ldots ,r_{l}^{(2k)};r_{l-1}^{(1)},r_{l-1}^{(2)},\ldots ,r_{l-1}^{(k)}; \ldots;r_{2}^{(1)}, r_{2}^{(2)}, \ldots ,r_{2}^{(k)}; r_{1}^{(1)}, r_{1}^{(2)}, \ldots ,r_{1}^{(k)}\right), \end{equation*}
we have:
\begin{align}\label{S710B}
\sum_{p=1}^{r^{(1)}_l}\sum_{q=1}^{r^{(1)}_{l-1}}\mathrm{min}\{n_{p,l},2n_{q,l-1}\}&=\sum_{s=1}^{k}r^{(s)}_{l-1}(r_l^{(2s-1)}+r_l^{(2s)}),&\\
\label{S713B}
\sum_{p=1}^{r^{(1)}_i}\sum_{q=1}^{r^{(1)}_{i-1}}\mathrm{min}\{n_{p,i},n_{q,i-1}\}&=\sum_{s=1}^{k}r^{(s)}_{i}r_{i-1}^{(s)}, \ \ 2\leq i \leq l-1,&\\
\label{S711B}\sum_{p=1}^{r_{i}^{(1)}} (\sum_{p>p'>0}2\mathrm{min} \{ n_{p,i},
n_{p',i}\}+n_{p,i})&= \sum_{s=1}^{k}r^{(s)^{2}}_{i}, \ \ 2\leq i \leq l-1,&\\
\label{S712B}\sum_{p=1}^{r_{l}^{(1)}} (\sum_{p>p'>0}2\mathrm{min} \{ n_{p,l},
n_{p',l}\}+n_{p,l})&= \sum_{s=1}^{2k}r^{(s)^{2}}_{l}.&
\end{align}
\end{lem}
\begin{flushright}
$\square$
\end{flushright}
We also need the combinatorial identity 
\begin{align}\label{S7KB}
\frac{1}{(q)_r}=\sum_{j\geq 0}p_r(j)q^j,
\end{align}
where \begin{equation*}\frac{1}{(q)_r}=\frac{1}{(1-q)(1-q^2)\cdots (1-q^r)},\end{equation*}
$r>0$ and $p_r(j)$ is the number of partition of $j$ with most $r$ parts (cf. \cite{A}).
\par Now, from the definition of the set $\mathfrak{B}_{W_{L_{B_{l}^{(1)}}(k\Lambda_0)}}$ and (\ref{S710B}), (\ref{S713B}), (\ref{S711B}), (\ref{S712B}), (\ref{S7KB}) follows the character formula:
\begin{thm}\label{uvodBl1}
\begin{align}\nonumber 
&\mathrm{ch} \  W_{L_{B_{l}^{(1)}}(k\Lambda_{0})}&\\
\nonumber
= \sum_{\substack{r^{(1)}_{1}\geq \ldots \geq r^{(k)}_{1}\geq 0}}
&\frac{q^{r^{(1)^{2}}_{1}+\cdots +r^{(k)^{2}}_{1}}}{(q)_{r^{(1)}_{1}-r^{(2)}_{1}}\cdots (q)_{r^{(k)}_{1}}}y^{r_1}_{1}&\\
\nonumber
 \sum_{\substack{r^{(1)}_{2}\geq \ldots \geq r^{(k)}_{2}\geq 0}}
&\frac{q^{r^{(1)^{2}}_{2}+\cdots +r^{(k)^{2}}_{2}-r_1^{(1)}r_2^{(1)}-\cdots -r_1^{(k)}r_2^{(k)}}}{(q)_{r^{(1)}_{2}-r^{(2)}_{2}}\cdots (q)_{r^{(k)}_{2}}}y^{r_2}_{2}&\\
\nonumber
\cdots\\
\nonumber
 \sum_{\substack{r^{(1)}_{l-1}\geq \ldots \geq r^{(k)}_{l-1}\geq 0}}
&\frac{q^{r^{(1)^{2}}_{l-1}+\cdots +r^{(k)^{2}}_{l-1}-r_{l-2}^{(1)}r_{l-1}^{(1)}-\cdots -r_{l-2}^{(k)}r_{l-1}^{(k)}}}{(q)_{r^{(1)}_{l-1}-r^{(2)}_{l-1}}\cdots (q)_{r^{(k)}_{l-1}}}y^{r_{l-1}}_{l-1}&\\
\nonumber
\sum_{r^{(1)}_{l}\geq \ldots \geq r^{(2k)}_{l}\geq  0}&\frac{q^{r^{(1)^{2}}_{l}+\cdots +r^{(2k)^{2}}_{l}-r_{l-1}^{(1)}(r_{l}^{(1)}+r_{l}^{(2)})
-\cdots -r_{l-1}^{(k)}(r_{l}^{(2k-1)}+r_{l}^{(2k)})}}{(q)_{r^{(1)}_{l}-r^{(2)}_{l}}\cdots (q)_{r^{(2k)}_{l}}}
y^{r_l}_{l}.&\\
\nonumber
&& \ \ \ \ \ \ \ \ \ \ \ \ \ \ \ \ \ \ \ \ \ \ \ \ \ \ \square
\end{align}
\end{thm}

\subsection{The basis of the \texorpdfstring{$W_{N_{B_l^{(1)}}(k\Lambda_{0})}$}{BWNBl(1)(kLambda0)}}
Using the relations among quasi-particles of the same and different colors, using the proof of the theorem \ref{THMB} with the same arguments as in \cite{Bu}, we can prove:
\begin{thm}\label{prop:S22BN} The set $\mathfrak{B}_{W_{N_{B_l^{(1)}}(k{\Lambda}_{0})}}=\left\{bv_{N_{B_l^{(1)}}(k\Lambda_{0})}:b \in B_{W_{N_{B_l^{(1)}}\left(k\Lambda_{0}\right)}}\right\}$, where
\begin{equation}\label{SkupNB}
B_{W_{N_{B_l^{(1)}}\left(k\Lambda_{0}\right)}}= \bigcup_{\substack{n_{r_{1}^{(1)},1}\leq \ldots \leq n_{1,1}\\\substack{ \ldots \\n_{r_{l-1}^{(1)},l-1}\leq \ldots \leq n_{1,l-1} }\\n_{r_{l}^{(1)},l}\leq \ldots \leq n_{1,l} }}\left(\text{or, equivalently,} \ \ \ 
\bigcup_{\substack{r_{1}^{(1)}\geq \cdots \geq 0\\\substack{\cdots \\r_{l-1}^{(1)}\geq \cdots \geq 
0}\\r_{l}^{(1)}\geq \cdots \geq 
0}}\right)
\end{equation}
\begin{equation*}\left\{b\right.= b(\alpha_{l})\cdots b(\alpha_{1})
=x_{n_{r_{l}^{(1)},l}\alpha_{l}}(m_{r_{l}^{(1)},l})\cdots x_{n_{1,l}\alpha_{l}}(m_{1,l})\cdots x_{n_{r_{1}^{(1)},1}\alpha_{1}}(m_{r_{1}^{(1)},1})\cdots  x_{n_{1,1}\alpha_{1}}(m_{1,1}):\end{equation*}
\begin{align}\nonumber
\left|
\begin{array}{l}
m_{p,i}\leq  -n_{p,i}+ \sum_{q=1}^{r_{i-1}^{(1)}}\text{min}\left\{n_{q,i-1},n_{p,i}\right\}- \sum_{p>p'>0} 
2 \ \text{min}\{n_{p,i}, n_{p',i}\},\\
\ \ \ \ \ \ \ \ \ \ \ \ \ \ \ \ \ \ \ \ \ \ \ \ \ \ \ \ \ \ \ \ \ \ \ \ \ \ \ \ \ \ \ \ \ \ \ \ \ \ \ \ \ \ \ \ \ \ \ \ \ \ \ \ \ \ 1\leq p \leq r_i^{(1)}, \ 1 \leq i \leq l-1; \\
m_{p+1,i} \leq   m_{p,i}-2n_{p,i} \  \text{if} \ n_{p+1,i}=n_{p,i}, \ 1\leq  p\leq r_{i}^{(1)}-1, \ 1 \leq i \leq l-1;\\
m_{p,l}\leq  -n_{p,l} + \sum_{q=1}^{r_{l-1}^{(1)}}\text{min}\left\{ 2n_{q,l-1},n_{p,l }\right\} - \sum_{p>p'>0} 2 \ \text{min}\{n_{p,l}, n_{p',l}\}, \  1\leq  p\leq r_{l}^{(1)};\\
m_{p+1,l}\leq  m_{p,l}-2n_{p,l} \  \text{if} \ n_{p,l}=n_{p+1,l}, \  1\leq  p\leq r_{l}^{(1)}-1
\end{array}\right\},
\end{align}
where $r_0^{(1)}=0$, is the base of the principal subspace $W_{N_{B_l^{(1)}}\left(k\Lambda_{0}\right)}$.
\end{thm}
\begin{flushright}
$\square$
\end{flushright}

\subsection{Characters of the principal subspace \texorpdfstring{$W_{N_{B_l^{(1)}}(k\Lambda_{0})}$}{WNBl(1)(kLambda0)}}
From the above theorem and (\ref{S710B})-(\ref{S7KB}) we can write the character formulas of principal subspace $W_{N_{B_l^{(1)}}(k\Lambda_{0})}$:
\begin{thm}\label{uvodBl2}
\begin{align}\label{characterB1} 
&\mathrm{ch} \  W_{N_{B_l^{(1)}}(k\Lambda_{0})}&\\
\nonumber
&= \sum_{\substack{r^{(1)}_{1}\geq \ldots \geq r^{(u_1)}_{1}\geq 0\\ u_1\geq0 }}
\frac{q^{r^{(1)^{2}}_{1}+\cdots +r^{(u_1)^{2}}_{1}}}{(q)_{r^{(1)}_{1}-r^{(2)}_{1}}\cdots (q)_{r^{(u_1)}_{1}}}y^{r_1}_{1}& 
\end{align}
\begin{equation*} 
 \sum_{\substack{r^{(1)}_{2}\geq \ldots \geq r^{(u_2)}_{2}\geq 0\\ u_2\geq 0}}
\frac{q^{r^{(1)^{2}}_{2}+\cdots +r^{(u_2)^{2}}_{2}-r_1^{(1)}r_2^{(1)}-\cdots -r_1^{(u_2)}r_2^{(u_2)}}}{(q)_{r^{(1)}_{2}-r^{(2)}_{2}}\cdots (q)_{r^{(u_2)}_{2}}}y^{r_2}_{2} 
\end{equation*} 
\begin{equation*}
\ \ \ \ \ \ \ \ \cdots
\end{equation*}
\begin{equation*}
\sum_{\substack{r^{(1)}_{l-1}\geq \ldots \geq r^{(u_{l-1})}_{l-1}\geq 0\\ u_{l-1}\geq0}}
\frac{q^{r^{(1)^{2}}_{l-1}+\cdots +r^{(u_{l-1})^{2}}_{l-1}-r_{l-2}^{(1)}r_{l-1}^{(1)}-\cdots -r_{l-2}^{(u_{l-1})}r_{l-1}^{(u_{l-1})}}}{(q)_{r^{(1)}_{l-1}-r^{(2)}_{l-1}}\cdots (q)_{r^{(u_{l-1})}_{l-1}}}y^{r_{l-1}}_{l-1} 
\end{equation*} 
\begin{equation*} 
\sum_{\substack{r^{(1)}_{l}\geq \ldots \geq r^{(2u_{l})}_{l}\geq  0\\ u_{l}\geq 0}}
\frac{q^{r^{(1)^{2}}_{l}+\cdots +r^{(2u_{l})^{2}}_{l}-r_{l-1}^{(1)}(r_{l}^{(1)}+r_{l}^{(2)})
-\cdots -r_{l-1}^{(u_{l})}(r_{l}^{(2u_{l}-1)}+r_{l}^{(2u_{l})})}}{(q)_{r^{(1)}_{l}-r^{(2)}_{l}}\cdots (q)_{r^{(2u_{l})}_{l}}}
y^{r_l}_{l}.
\end{equation*} 
\end{thm}
\begin{flushright}
$\square$
\end{flushright}
We can determine the character of principal subspace $W_{N_{B_l^{(1)}}(k{\Lambda}_{0})}$ using the Po\-incar\'{e}-Birk\-hoff-Witt theorem, since we have
\begin{equation*}
W_{N_{B_l^{(1)}}(k{\Lambda}_{0})}\cong U(\mathcal{L}(\mathfrak{n}_+)_{<0}).
\end{equation*}
Set $\left\{x_{\alpha}(m): \alpha \in R_+, m <0\right\}$ a basis of the Lie algebra $\mathcal{L}(\mathfrak{n}_+)_{<0}$ with a total order on this set: 
\begin{equation*}x(m)\leq y(m')\Leftrightarrow  x< y \quad \text{or} \quad x=y \quad \text{and} \quad m< m'.
\end{equation*}
Now, we can write a basis of $U(\mathcal{L}(\mathfrak{n}_+)_{<0})$:
\begin{align}\label{102B}
&x_{\alpha_1}(m^1_1)\cdots  x_{\alpha_1}(m^{s_1}_1)x_{\alpha_1+\alpha_2}(m^1_2)\cdots x_{\alpha_1+\alpha_2}(m^{s_2}_2)\cdots x_{\alpha_1+\alpha_2+\cdots + \alpha_l}(m^1_l)\cdots&\\
\nonumber
& \cdots x_{\alpha_1+\alpha_2+\cdots + \alpha_l}(m^{s_l}_l)x_{\alpha_1+2\alpha_2+\cdots +2\alpha_l}(m^1_{l+1})\cdots   x_{\alpha_1+2\alpha_2+\cdots +2\alpha_l}(m^{s_{l+1}}_{l+1})\cdots &\\
\nonumber
& \cdots x_{\alpha_1+\alpha_2+\cdots +2\alpha_l}(m^1_{2l-1})\cdots   x_{\alpha_1+\alpha_2+\cdots +2\alpha_l}(m^{s_{2l-1}}_{2l-1})\cdots &\\
\nonumber
&\cdots x_{ \alpha_{l-1}}(m^1_{l^2-3})\cdots x_{ \alpha_{l-1}}(m^{s_{l^2-3}}_{l^2-3})\cdots x_{ \alpha_{l-1}+2\alpha_l}(m^1_{l^2-1})\cdots x_{ \alpha_{l-1}+2\alpha_l}(m^{s_{l^2-1}}_{l^2-1})&\\
\nonumber
& x_{ \alpha_l}(m^1_{l^2})\cdots x_{ \alpha_l}(m^{s_{l^2}}_{l^2}),&
\end{align} with $m_i^1\leq \cdots \leq m_i^{s_i}$, $s_i \in \mathbb{N}$ for $i=1,\ldots,l^2$. \\ 
It follows that the subspace $U(\mathcal{L}(\mathfrak{n}_+)_{<0})_{(m,r_1,\ldots, r_l)}$ has basis (\ref{102B}), where  
\begin{equation*}(m,r_1,\ldots, r_{l-1}, r_l)=(\sum_{i=1}^{l^2}\sum_{j=1}^{s_i}m^j_i, s_1+s_2+\cdots + s_{2l-1},\ldots,  s_{l-1}+\cdots +s_{l^2-1},s_l+2s_{l+1}+\cdots +s_{l^2}).\end{equation*}
The  bijection map 
\begin{align}\nonumber
U(\mathcal{L}(\mathfrak{n}_+)_{<0})&\rightarrow  W_{N_{B_l^{(1)}}(k{\Lambda}_{0})}\\
\nonumber
b&\mapsto  bv_{N_{B_l^{(1)}}(k{\Lambda}_{0})}
\end{align}
maps weighted subspace $U(\mathcal{L}(\mathfrak{n}_+)_{<0})_{(m,r_1,\ldots, r_n)}$ on ${W_{N_{B_l^{(1)}}(k{\Lambda}_{0})}}_{(m,r_1,\ldots,r_l)}$. Thus, we also have
\begin{equation}\label{KN3B} 
\textrm{ch} \ W^B_{N(k\Lambda_{0})}=\prod_{m > 0}\frac{1}{(1-q^my_1)}\frac{1}{(1-q^my_1y_2)}\cdots \frac{1}{(1-q^my_1\cdots y_l)}
\end{equation} 
\begin{equation*} 
\frac{1}{(1-q^my_1y_2^2\cdots y_l^2)}\cdots \frac{1}{(1-q^my_1y_2\cdots y_l^2)}
\end{equation*} 
\begin{equation*} 
\frac{1}{(1-q^my_2)}\frac{1}{(1-q^my_2y_3)}\cdots \frac{1}{(1-q^my_2\cdots y_l)}\frac{1}{(1-q^my_2y_3^2\cdots y_l^2)}\cdots \frac{1}{(1-q^my_2y_3\cdots y_l^2)} 
\end{equation*} 
\begin{equation*}
\ \ \ \ \ \ \ \ \cdots  
\end{equation*} 
\begin{equation*}
\frac{1}{(1-q^{l-1})}\frac{1}{(1-q^my_{l-1}y_l)} \frac{1}{(1-q^my_{l-1}y_l^2)} \frac{1}{(1-q^my_l)} 
\end{equation*} 
Now from (\ref{characterB1}) and (\ref{KN3B}) follows a new identity of Rogers-Ramanujan's type:
\begin{thm}
\begin{equation*} 
\prod_{m > 0}\frac{1}{(1-q^my_1)}\frac{1}{(1-q^my_1y_2)}\cdots \frac{1}{(1-q^my_1\cdots y_l)}\frac{1}{(1-q^my_1y_2^2\cdots y_l^2)}\cdots \frac{1}{(1-q^my_1y_2\cdots y_l^2)} 
\end{equation*} 
\begin{equation*} 
\frac{1}{(1-q^my_2)}\frac{1}{(1-q^my_2y_3)}\cdots \frac{1}{(1-q^my_2\cdots y_l)}\frac{1}{(1-q^my_2y_3^2\cdots y_l^2)}\cdots \frac{1}{(1-q^my_2y_3\cdots y_l^2)} 
\end{equation*} 
\begin{equation*}
\ \ \ \ \ \ \ \ \cdots  
\end{equation*} 
\begin{equation*}
\frac{1}{(1-q^{l-1})}\frac{1}{(1-q^my_{l-1}y_l)} \frac{1}{(1-q^my_{l-1}y_l^2)} \frac{1}{(1-q^my_l)} 
\end{equation*} 
\begin{equation*} 
= \sum_{\substack{r^{(1)}_{1}\geq \ldots \geq r^{(u_1)}_{1}\geq 0\\ u_1\geq0 }}
\frac{q^{r^{(1)^{2}}_{1}+\cdots +r^{(u_1)^{2}}_{1}}}{(q)_{r^{(1)}_{1}-r^{(2)}_{1}}\cdots (q)_{r^{(u_1)}_{1}}}y^{r_1}_{1} 
\end{equation*} 
\begin{equation*} 
 \sum_{\substack{r^{(1)}_{2}\geq \ldots \geq r^{(u_2)}_{2}\geq 0\\ u_2\geq 0}}
\frac{q^{r^{(1)^{2}}_{2}+\cdots +r^{(u_2)^{2}}_{2}-r_1^{(1)}r_2^{(1)}-\cdots -r_1^{(u_2)}r_2^{(u_2)}}}{(q)_{r^{(1)}_{2}-r^{(2)}_{2}}\cdots (q)_{r^{(u_2)}_{2}}}y^{r_2}_{2} 
\end{equation*} 
\begin{equation*}
\ \ \ \ \ \ \ \ \cdots 
\end{equation*} 
\begin{equation*}
\sum_{\substack{r^{(1)}_{l-1}\geq \ldots \geq r^{(u_{l-1})}_{l-1}\geq 0\\ u_{l-1}\geq0}}
\frac{q^{r^{(1)^{2}}_{l-1}+\cdots +r^{(u_{l-1})^{2}}_{l-1}-r_{l-2}^{(1)}r_{l-1}^{(1)}-\cdots -r_{l-2}^{(u_{l-1})}r_{l-1}^{(u_{l-1})}}}{(q)_{r^{(1)}_{l-1}-r^{(2)}_{l-1}}\cdots (q)_{r^{(u_{l-1})}_{l-1}}}y^{r_{l-1}}_{l-1}
\end{equation*} 
\begin{equation*}
\sum_{\substack{r^{(1)}_{l}\geq \ldots \geq r^{(2u_{l})}_{l}\geq  0\\ u_{l}\geq 0}}
\frac{q^{r^{(1)^{2}}_{l}+\cdots +r^{(2u_{l})^{2}}_{l}-r_{l-1}^{(1)}(r_{l}^{(1)}+r_{l}^{(2)})
-\cdots -r_{l-1}^{(u_{l})}(r_{l}^{(2u_{l}-1)}+r_{l}^{(2u_{l})})}}{(q)_{r^{(1)}_{l}-r^{(2)}_{l}}\cdots (q)_{r^{(2u_{l})}_{l}}}
y^{r_l}_{l}.\end{equation*} 
\end{thm}\begin{flushright}
$\square$
\end{flushright}

\section{The case  \texorpdfstring{$C_\MakeLowercase{l}^{(1)}$}{Cl{(1)}}}\label{Cl4}
\subsection{Principal subspaces for affine Lie algebra of type \texorpdfstring{$C_\MakeLowercase{l}^{(1)}$}{Cl{(1)}}}
Let $\mathfrak{g}$ be of the type $C_l$, $l \geq 3$. We have the following base of the root system $R$: 
$$\Pi =\left\{\alpha_1=\frac{1}{\sqrt{2}}\left(\epsilon_1-\epsilon_2\right), \ldots,\alpha_{l-1}=\frac{1}{\sqrt{2}}\left(\epsilon_{l-1}-\epsilon_l\right),\alpha_l=\sqrt{2}\epsilon_l\right\},$$ (where $\left\{\epsilon_1,\ldots, \epsilon_l\right\}$ is as before orthonormal basis of the $\mathbb{R}^l$),
the set of positive roots: $$R_{+}=\left\{\frac{1}{\sqrt{2}}\left(\epsilon_{i}- \epsilon_{j}\right):i < j\right\}\cup\left\{\frac{1}{\sqrt{2}}\left(\epsilon_{i}+ \epsilon_{j}\right):i \neq j\right\}\cup \left\{ \sqrt{2}\epsilon_{i}: 1 \leq i \leq l\right\}$$ and the highest root $$\theta=\sqrt{2}\epsilon_1=2\alpha_1 + \cdots +2\alpha_{l-1}+\alpha_l$$
(cf: \cite{K}).
We denote the vector space 
\begin{equation*}
U_{C_l^{(1)}} = U(\mathcal{L}\left(\mathfrak{n}_{\alpha_{1}}\right))\cdots U(\mathcal{L}\left(\mathfrak{n}_{\alpha_{l}}\right)),
\end{equation*}
where 
\begin{equation*}
\mathcal{L}(\mathfrak{n}_{\alpha} ) = \mathfrak{n}_{\alpha}  \otimes \mathbb{C}[t,t^{-1}],
 \ \ \ \mathfrak{n}_{\alpha} = \mathbb{C}x_{\alpha}, \  \  \alpha \in R_+.
\end{equation*}
Now, we have:
\begin{lem}\label{prvalemaC} Let $k\geq 1$. We have
\begin{align}\nonumber
W_{L_{C_l^{(1)}}(k\Lambda_{0})}& = U_{C_l^{(1)}}v_{L_{C_l^{(1)}}(k\Lambda_{0})},\\
\nonumber
W_{N_{C_l^{(1)}}(k\Lambda_{0})}&=U_{C_l^{(1)}}v_{N_{C_l^{(1)}}(k\Lambda_{0})}.
\end{align}
\end{lem} 
\begin{flushright}
$\square$
\end{flushright}

\subsection{The spanning set of \texorpdfstring{$W_{L_{C_l^{(1)}}(k\Lambda_{0})}$}{W{LCl(1)(kLambda{0})}}}\label{Cl42}
Here we establish relations among quasi-particles of colors $i=l-1$ and $i=l$ and relations among quasi-particles of colors $i=1,\ldots , l$ and $j=i\pm 1$.
\par Note, that as in the case of $C_2^{(1)}$, we have:
\begin{lem}\label{lem:S235}
Let $n_{l-1},n_l \in \mathbb{N}$ be fixed. One has
\begin{align}\nonumber
\left(1-\frac{z_{l}}{z_{l-1}}\right)^{\emph{\text{min}}\left\{ 
n_{l-1},2n_{l}\right\}} x_{n_{l-1}\alpha_{l-1}}(z_{l-1})x_{n_{l }\alpha_{l }}(z_{l})v_{N_{C_l^{(1)}}(k\Lambda_{0})}\\
\label{al:S28} 
\in z_{l-1}^{-\emph{\text{min}}\left\{ 
n_{l-1},2n_{l}\right\}}W_{N_{C_l^{(1)}}(k\Lambda_{0})}\left[\left[z_{l-1},z_{l}\right]\right]. 
\end{align}
\end{lem}
\begin{flushright}
$\square$
\end{flushright}
Using the commutator formula for vertex operators we can prove:
\begin{lem}\label{lem:S232Cl}
Let $1 \leq i \leq l-2$, $n_{i+1},n_i \in \mathbb{N}$ be fixed. One has
\begin{itemize}
	\item [a)]  
$(z_{1}-z_{2})^{n_i}x_{n_i\alpha_{i}}(z_{l})x_{n_{i+1}\alpha_{i+1}}(z_{2})=(z_{1}-z_{2})^{n_i}x_{n_{i+1}\alpha_{i+1}}(z_{2})x_{n_i\alpha_{i}}(z_{1}).$
\item [b)] 
$(z_{1}-z_{2})^{n_{i+1}}x_{n_i\alpha_{i}}(z_{1})x_{n_{i+1}\alpha_{i+1}}(z_{2})=(z_{1}-z_{2})^{n_{i+1}}x_{n_{i+1}\alpha_{i+1}}(z_{i+1})x_{n_i\alpha_{i}}(z_{1}).$
\end{itemize}
\end{lem}
\begin{flushright}
$\square$
\end{flushright}
Using the same arguments as in proof of Lemma \ref{lem:S323B} follows:
\begin{lem}\label{lem:S235Cl}
\begin{align}\nonumber
\left(1-\frac{z_{i}}{z_{i+1}}\right)^{\emph{\text{min}}\left\{ 
n_{i+1},n_{i}\right\}}  x_{n_{i}\alpha_{i}}(z_{i}) x_{n_{i +1}\alpha_{i+1 }}(z_{i+1 }) v_{N_{C_l^{(1)}}(k\Lambda_{0})}\\
\label{al:S28Cl} 
\in 
z_{p,i+1}^{- \emph{\text{min}}\left\{ 
n_{i+1},n_{i}\right\}} W_{N_{C_l^{(1)}}(k\Lambda_{0})}\left[\left[z_{i},z_{i+1}\right]\right]. 
\end{align}
\end{lem}
\begin{flushright}
$\square$
\end{flushright}
Now, as in Subsection \ref{Bl32} follows 
\begin{prop}\label{prop:S22Cl} The set $\mathfrak{B}_{W_{L_{C_l^{(1)}}(k{\Lambda}_{0})}}=\left\{bv_{L_{C_l^{(1)}}(k\Lambda_{0})}:b \in B_{W_{L_{C_l^{(1)}}\left(k\Lambda_{0}\right)}}\right\}$, where
\begin{equation}\label{SkupLCl}
B_{W_{L_{C_l^{(1)}}\left(k\Lambda_{0}\right)}}= \bigcup_{\substack{n_{r_{1}^{(1)},1}\leq \ldots \leq n_{1,1}\leq 
2k\\\substack{ \ldots \\n_{r_{l-1}^{(1)},l-1}\leq \ldots \leq n_{1,l-1}\leq 2k}\\n_{r_{l}^{(1)},l}\leq \ldots \leq n_{1,l}\leq k}}\left(\text{or, equivalently,} \ \ \ 
\bigcup_{\substack{r_{1}^{(1)}\geq \cdots\geq r_{1}^{(2k)}\geq 0\\\substack{\cdots \\r_{l-1}^{(1)}\geq \cdots\geq r_{l-1}^{(2k)}\geq 
0}\\r_{l}^{(1)}\geq \cdots\geq r_{l}^{(k)}\geq 
0}}\right)
\end{equation}
\begin{equation*}\left\{b\right.=b(\alpha_{1})\cdots b(\alpha_{l})
=x_{n_{r_{1}^{(1)},1}\alpha_{1}}(m_{r_{1}^{(1)},1})\cdots  x_{n_{1,l}\alpha_{l}}(m_{1,l}):\end{equation*}
\begin{align}\nonumber
\left|
\begin{array}{l}
m_{p,l}\leq  -n_{p,l} - \sum_{p>p'>0} 2 \ \text{min}\{n_{p,l}, n_{p',l}\}, \  1\leq  p\leq r_{l}^{(1)};\\
m_{p+1,l}\leq  m_{p,l}-2n_{p,l} \  \text{if} \ n_{p,l}=n_{p+1,l}, \  1\leq  p\leq r_{l}^{(1)}-1;\\
m_{p,l-1}\leq  -n_{p,l-1}+ \sum_{q=1}^{r_{l}^{(1)}}\text{min}\left\{2n_{q,l},n_{p,l-1}\right\}- \sum_{p>p'>0} 
2 \ \text{min}\{n_{p,l}, n_{p',l}\},\\
\ \ \ \ \ \ \ \ \ \ \ \ \ \ \ \ \ \ \ \ \ \ \ \ \ \ \ \ \ \ \ \ \ \ \ \ \ \ \ \ \ \ \ \ \ \ \ \ \ \ \ \ \ \ \ \ \ \ \ \ \ \ \ \ \ \ \ \ \ \ \ \ \ \ \ \ \ \ \ \ \  1\leq  p\leq r_{l}^{(1)};\\
m_{p+1,l-1} \leq   m_{p,l-1}-2n_{p,l-1} \  \text{if} \ n_{p+1,l-1}=n_{p,l-1}, \ 1\leq  p\leq r_{l-1}^{(1)}-1;\\
m_{p,i}\leq  -n_{p,i}+ \sum_{q=1}^{r_{i+1}^{(1)}}\text{min}\left\{n_{q,i+1},n_{p,i}\right\}- \sum_{p>p'>0} 
2 \ \text{min}\{n_{p,i}, n_{p',i}\},\\
\ \ \ \ \ \ \ \ \ \ \ \ \ \ \ \ \ \ \ \ \ \ \ \ \ \ \ \ \ \ \ \ \ \ \ \ \ \ \ \ \ \ \ \ \ \ \ \ \ \ \ \ \ \ \ \ \ \ \ \ \ \ \ \ 1\leq  p\leq r_{i}^{(1)}, \ 1 \leq i \leq l-2;\\
m_{p+1,i} \leq   m_{p,i}-2n_{p,i} \  \text{if} \ n_{p+1,i}=n_{p,i}, \ 1\leq  p\leq r_{i}^{(1)}-1, \ 1 \leq i \leq l-2\\
\end{array}
\right\},
\end{align} spans the principal subspace $W_{L_{C_l^{(1)}}\left(k\Lambda_{0}\right)}$.
\end{prop}
\begin{flushright}
$\square$
\end{flushright}

\subsection{Proof of linear independence}
As we mentioned in the Introduction, we prove the linear independence of the monomial vectors from Proposition \ref{prop:S22Cl} using the coefficient of an intertwining operator, the simple current operator and the \enquote{Weyl group translation} operator. So, first we describe their properties, which we use in the proof of linear independence of quasi-particle bases. 

\subsubsection{\textbf{Projection \texorpdfstring{$\pi_{\mathfrak{R}}$}{piR}}}\label{ss:projCl}
Fix a level $k> 1$. Consider the direct sum decomposition of tensor product of $k$ principal subspaces $W_{L_{C_l^{(1)}}( \Lambda _{0})}$ of level 1 standard modules $L_{C_l^{(1)}}( \Lambda _{0})$
\begin{equation*}
W_{L_{C_l^{(1)}}( \Lambda _{0})}\otimes \cdots\otimes W_{L_{C_l^{(1)}}( \Lambda _{0})}=\bigcup_{\substack{u_{1}^{(1)}\geq \cdots\geq u_{1}^{(k)}\geq 0\\\substack{\cdots \\u_{l-1}^{(1)}\geq \cdots\geq u_{l-1}^{(k)}\geq 
0}\\u_{l}^{(1)}\geq \cdots\geq u_{l}^{(k)}\geq 
0}} {W_{L_{C_l^{(1)}}( \Lambda _{0})}}_{(u^{(k)}_{1};\ldots ;u^{(k)}_{l-1}; u_{l}^{(k)})}\otimes \cdots \otimes  {W_{L_{C_l^{(1)}}( \Lambda _{0})}}_{(u^{(1)}_{1};\ldots ;u^{(1)}_{l-1}; u_{l}^{(1)})},
\end{equation*}
where $ {W_{L_{C_l^{(1)}}( \Lambda _{0})}}_{(u^{(j)}_{1};\ldots ;u^{(j)}_{l-1}; u_{l}^{(j)})}$ is a $\mathfrak{h}$-weight subspace of weight $\sum_{i=1}^l u^{(j)}_i\alpha_i$, $1 \leq j \leq k$, and
where
\begin{equation*}
v_{L_{C_l^{(1)}}(k\Lambda_0)}=\underbrace{v_{L_{C_l^{(1)}}(\Lambda_{0})} \otimes \cdots \otimes v_{L_{C_l^{(1)}}(\Lambda_{0})}}_{k \ \text{factors}}\end{equation*}
is the highest weight vector of weight $k\Lambda_0$. 
\par For a chosen dual-charge-type 
\begin{equation*}
\mathfrak{R}=\left( r_{1}^{(1)}, \ldots ,r_{1}^{(2k)}; \ldots ; r_{l-1}^{(1)}, \ldots ,r_{l-1}^{(2k)};r_{l}^{(1)}, \ldots , r_{l}^{(k)}\right),
\end{equation*}
set the projection $\pi_{\mathfrak{R}}$ of principal subspace $W_{L_{C_l^{(1)}}(k\Lambda_{0})}$ 
to the subspace
\begin{equation*}
 {W_{L_{C_l^{(1)}}( \Lambda _{0})}}_{(\mu^{(k)}_{1};\ldots ;\mu^{(k)}_{l-1}; r_{l}^{(k)})}\otimes \cdots \otimes  {W_{L_{C_l^{(1)}}( \Lambda _{0})}}_{(\mu^{(1)}_{1};\ldots ;\mu^{(1)}_{l-1}; r_{l}^{(1)})},
\end{equation*}
where 
\begin{equation*}  
\mu^{(t)}_{i}=r^{(2t)}_{i}+ r^{(2t-1)}_{i},
\end{equation*}
for every $1 \leq  t \leq k$ and $1 \leq i \leq l-1$. 
If we denote by the same symbol $\pi_{\mathfrak{R}}$ the generalization of this projection to the space of formal series with coefficients in $W_{L_{C_l^{(1)}}( \Lambda _{0})} \otimes \cdots \otimes  W_{L_{C_l^{(1)}}( \Lambda _{0})}$, then for a generating function
\begin{equation}\label{eq:S31Cl}
x_{n_{r_{1}^{(1)},1}\alpha_{1}}(z_{r_{1}^{(1)},1}) \cdots     x_{n_{1,1}\alpha_{1}}(z_{1,1})\cdots x_{n_{r_{l-1}^{(1)},l-1}\alpha_{l-1}}(z_{r_{l-1}^{(1)},l-1}) \cdots     x_{n_{1,l-1}\alpha_{l-1}}(z_{1,l-1})
\end{equation}
\begin{equation*}
x_{n_{r_{l}^{(1)},l}\alpha_{l}}(z_{r_{l}^{(1)},l})\cdots  x_{n_{1,l}\alpha_{l}}(z_{1,l}),
\end{equation*}
we have 
\begin{align}
\label{projekcijaCl}
\pi_{\mathfrak{R}}& x_{n_{r_{1}^{(1)},1}\alpha_{1}}(z_{r_{1}^{(1)},1})\cdots  x_{n_{1,l}\alpha_{l}}(z_{1,l}) \ v_{L_{C_l^{(1)}}(k\Lambda_{0})}&\\
\nonumber
=&\text{C}  
 \ x_{n^{(k)}_{r^{(2k-1)}_{1},1}\alpha_{1}}(z_{r_{1}^{(2k-1)},1})\cdots                x_{n^{(k)}_{r^{(2k)}_{1},1}\alpha_{1}}(z_{r^{(2k)}_{1},1})\cdots  x_{n^{(k)}_{1,1}\alpha_{1}}(z_{1,1})\cdots &\\
\nonumber
&\ \cdots x_{n^{(k)}_{r^{(2k-1)}_{l-1},l-1}\alpha_{l-1}}(z_{r_{1}^{(2k-1)},l-1})\cdots                x_{n^{(k)}_{r^{(2k)}_{l-1},l-1}\alpha_{l-1}}(z_{r^{(2k)}_{l-1},l-1})\cdots  x_{n^{(k)}_{1,l-1}\alpha_{l-1}}(z_{1,l-1})&\\
\nonumber
&  \ \ \ \ \ \ \ \ \ \ \ \ \ \ \ \ \ \ \ \ \ \ \ \ \ \ \ \ \ \ \ \  x_{n^{(k)}_{r^{(k)}_{l},l}\alpha_{l}}(z_{r_{l}^{(k)},l})\cdots   \cdots x_{n^{(k)}_{1,l}\alpha_{l}}(z_{1,l}) \ v_{L_{C_l^{(1)}}(\Lambda_{0})}\otimes &\\
\nonumber
& \ \ \ \ \ \ \ \ \ \ \ \ \ \ \ \ \otimes \ldots \otimes&\\
\nonumber
\otimes &
 \ x_{n^{(1)}_{r^{( 1)}_{1},1}\alpha_{1}}(z_{r_{1}^{( 1)},1})\cdots                x_{n^{(1)}_{r^{(2)}_{1},1}\alpha_{1}}(z_{r^{(2)}_{1},1})\cdots  x_{n^{(1)}_{1,1}\alpha_{1}}(z_{1,1})\cdots &\\
\nonumber
&\ x_{n^{(1)}_{r^{( 1)}_{l-1},l-1}\alpha_{l-1}}(z_{r_{1}^{( 1)},l-1})\cdots                x_{n^{(1)}_{r^{(2 )}_{l-1},l-1}\alpha_{l-1}}(z_{r^{(2 )}_{l-1},l-1})\cdots  x_{n^{(1)}_{1,l-1}\alpha_{l-1}}(z_{1,l-1})&\\
\nonumber
&  \ \ \ \ \ \ \ \ \ \ \ \ \ \ \ \ \ \ \ \ \ \ \ \ \ \ \ \ \ \ \ \  x_{n^{(1)}_{r^{(1)}_{l},l}\alpha_{l}}(z_{r_{l}^{(1)},l})\cdots   x_{n^{(1)}_{1,l}\alpha_{l}}(z_{1,l}) \ v_{L_{C_l^{(1)}}(\Lambda_{0})},&
\end{align}
where $\text{C} \in \mathbb{C}^{*}$,  
\begin{equation*}
0 \leq  n^{(t)}_{p,l} \leq 1, \  n^{(1)}_{p,l}\geq n^{(2)}_{p,l}\geq \ldots \geq  n^{(k-1)}_{p,l}\geq  n^{(k)}_{p,l}, \  n_{p,l}=n^{(1)}_{p,l}+n^{(2)}_{p,l}+ \cdots +n^{(k-1)}_{p,l}+n^{(k)}_{p,l}
\end{equation*}
and 
\begin{equation*}
0 \leq  n^{(t)}_{p,i}\leq  2, \ n^{(1)}_{p,i}\geq  n^{(2)}_{p,i} \geq  \ldots  \geq  n^{(k-1)}_{p,i} \geq  n^{(k)}_{p,i}, \  n_{p,i}=n^{(1)}_{p,i}+n^{(2)}_{p,i}+ \cdots +n^{(k-1)}_{p,i}+n^{(k)}_{p,i},
\end{equation*}
for every $t$, $1\leq t \leq k$, and every $p$, $1 \leq p \leq r_{i}^{(1)}$, $1 \leq i \leq l$.
\par In the projection (\ref{projekcijaCl}), $n_{p,l}$ generating functions $x_{\alpha_{l}}(z_{p,l})$ ($1\leq p \leq r^{(1)}_{l}$), whose product generates a quasi-particle of charge $n_{p,l}$, 
 \enquote{are placed at} the first (from right to left) $n_{p,l}$ tensor factors.  
This can be shown as in the example in Figure \ref{slika3C}, where each box represents $ n_ {p,l}^{(t)}$. 
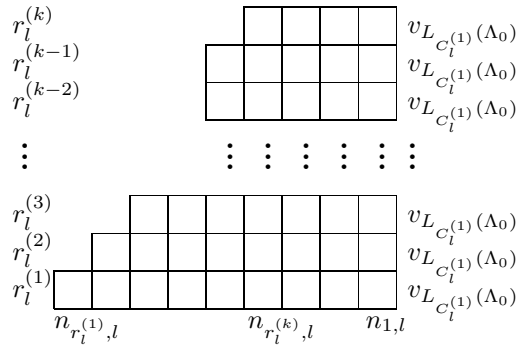
\begin{figure}[h!tb]
\centering
\setlength{\unitlength}{5mm}
\begin{picture}(10,10)
\linethickness{0.075mm}
\multiput(0,0)(1,0){1}
{\line(0,1){1}}
\multiput(1,0)(1,0){1}
{\line(0,1){2}}
\multiput(2,0)(1,0){5}
{\line(0,1){3}}
\multiput(2,1)(1,0){1}
{\line(1,0){6}}
\multiput(2,3)(1,0){1}
{\line(1,0){7}}
\multiput(1,0)(1,0){3}
{\line(0,1){2}}
\multiput(1,2)(0,1){1}
{\line(13,0){8}}
\multiput(4,0)(1,0){1}
{\line(0,1){2}}
\multiput(5,0)(1,0){2}
{\line(0,1){2}}
\multiput(8,0)(1,0){1}
{\line(0,1){2}}
\multiput(9,0)(1,0){1}
{\line(0,1){3}}
\multiput(8,5)(1,0){2}
{\line(0,1){3}}
\multiput(7,0)(1,0){1}
{\line(0,1){2}}
\multiput(7,5)(1,0){2}
{\line(0,1){3}}
\multiput(5,5)(1,0){2}
{\line(0,1){3}}
\multiput(5,7)(1,0){2}
{\line(1,0){3}}
\multiput(5,8)(1,0){2}
{\line(1,0){3}}
\multiput(4,7)(1,0){2}
{\line(1,0){3}}
\multiput(4,6)(1,0){2}
{\line(1,0){4}}
\multiput(4,5)(1,0){2}
{\line(1,0){4}}
\multiput(4,5)(1,0){1}
{\line(0,1){2}}
\multiput(0,0)(0,1){2}
{\line(1,0){9}}
\multiput(5,0)(1,0){2}
{\line(0,1){2}}
\multiput(7,0)(1,0){1}
{\line(0,1){3}}
\multiput(8,0)(1,0){2}
{\line(0,1){3}}
\multiput(6,0)(1,0){2}
{\line(0,1){2}}
\put(0,-0.5){\scriptsize{\footnotesize{$n_{r_l^{(1)},l}$}}}
\put(5.1,-0.5){\scriptsize{\footnotesize{$n_{r_l^{(k)},l}$}}}
\put(8.2,-0.5){\scriptsize{\footnotesize{$n_{1,l}$}}}
\put(4.5,3.7){$\textbf{\vdots}$}
\put(5.5,3.7){$\textbf{\vdots}$}
\put(6.5,3.7){$\textbf{\vdots}$}
\put(7.5,3.7){$\textbf{\vdots}$}
\put(8.5,3.7){$\textbf{\vdots}$}
\put(9.3,3.7){$\textbf{\vdots}$}
\put(-0.9,3.7){$\textbf{\vdots}$}
\put(-1.1,0.3){\footnotesize{$r_l^{(1)}$}}
\put(-1.1,1.3){\footnotesize{$r_l^{(2)}$}}
\put(-1.1,2.3){\footnotesize{$r_l^{(3)}$}}
\put(9.3,0.3){\footnotesize{$v_{L_{C_l^{(1)}}(\Lambda_{0})}$}}
\put(9.3,1.3){\footnotesize{$v_{L_{C_l^{(1)}}(\Lambda_{0})}$}}
\put(9.3,2.3){\footnotesize{$v_{L_{C_l^{(1)}}(\Lambda_{0})}$}}
\put(9.3,5.3){\footnotesize{$v_{L_{C_l^{(1)}}(\Lambda_{0})}$}}
\put(9.3,6.3){\footnotesize{$v_{L_{C_l^{(1)}}(\Lambda_{0})}$}}
\put(9.3,7.3){\footnotesize{$v_{L_{C_l^{(1)}}(\Lambda_{0})}$}}
\put(-1.1,5.3){\footnotesize{$r_l^{(k-2)}$}}
\put(-1.1,6.3){\footnotesize{$r_l^{(k-1)}$}}
\put(-1.1,7.3){\footnotesize{$r_l^{(k)}$}}
\end{picture}
\bigskip
\caption{Sketch of projection $\pi_{\mathfrak{R}}$ for color $i=l$}\label{slika3C}
\end{figure}
The situation for the genarationg functions of colors $1 \leq i\leq l-1$ can be shown as in the example in Figure \ref{slika4C},
\begin{figure}[h!tb]
\centering
\setlength{\unitlength}{5mm}
\begin{picture}(10,10)
\linethickness{0.075mm}
\multiput(0,0)(1,0){1}
{\line(0,1){1}}
\multiput(1,0)(1,0){1}
{\line(0,1){2}}
\multiput(2,0)(1,0){8}
{\line(0,1){4}}
\multiput(2,1)(1,0){1}
{\line(1,0){6}}
\multiput(2,3)(1,0){1}
{\line(1,0){7}}
\multiput(2,4)(1,0){1}
{\line(1,0){7}}
\multiput(1,0)(1,0){3}
{\line(0,1){2}}
\multiput(1,2)(0,1){1}
{\line(13,0){8}}
\multiput(4,0)(1,0){1}
{\line(0,1){2}}
\multiput(5,0)(1,0){2}
{\line(0,1){2}}
\multiput(8,0)(1,0){1}
{\line(0,1){2}}
\multiput(9,0)(1,0){1}
{\line(0,1){3}}
\multiput(8,6)(1,0){2}
{\line(0,1){3}}
\multiput(7,0)(1,0){1}
{\line(0,1){2}}
\multiput(7,6)(1,0){2}
{\line(0,1){3}}
\multiput(5,6)(1,0){2}
{\line(0,1){2}}
\multiput(5,7)(1,0){2}
{\line(1,0){3}}
\multiput(5,8)(1,0){2}
{\line(1,0){3}}
\multiput(6,9)(1,0){2}
{\line(1,0){2}}
\multiput(6,8)(1,0){1}
{\line(0,1){1}}
\multiput(7,8)(1,0){3}
{\line(0,1){2}}
\multiput(7,10)(1,0){2}
{\line(1,0){1}}
\multiput(5,8)(1,0){2}
{\line(1,0){2}}
\multiput(4,7)(1,0){2}
{\line(1,0){4}}
\multiput(4,6)(1,0){2}
{\line(1,0){4}}
\multiput(4,6)(1,0){1}
{\line(0,1){1}}
\multiput(0,0)(0,1){2}
{\line(1,0){9}}
\multiput(5,0)(1,0){2}
{\line(0,1){2}}
\multiput(7,0)(1,0){1}
{\line(0,1){3}}
\multiput(8,0)(1,0){2}
{\line(0,1){3}}
\multiput(6,0)(1,0){2}
{\line(0,1){2}}
\put(0,-0.5){\footnotesize{$n_{r_i^{(1)},i}$}}
\put(8.5,-0.5){\scriptsize{\footnotesize{$n_{1,i}$}}}
\put(4.5,4.7){$\textbf{\vdots}$}
\put(5.5,4.7){$\textbf{\vdots}$}
\put(6.5,4.7){$\textbf{\vdots}$}
\put(7.5,4.7){$\textbf{\vdots}$}
\put(8.5,4.7){$\textbf{\vdots}$}
\put(9.3,4.7){$\textbf{\vdots}$}
\put(-0.9,4.7){$\textbf{\vdots}$}
\put(-1.1,0.3){\footnotesize{$r_i^{(1)}$}}
\put(-1.1,1.3){\footnotesize{$r_i^{(2)}$}}
\put(-1.1,2.3){\footnotesize{$r_i^{(3)}$}}
\put(-1.1,3.3){\footnotesize{$r_i^{(4)}$}}
\put(9.3,0.9){\footnotesize{$v_{L_{C_l^{(1)}}(\Lambda_{0})}$}}
\put(9.3,2.9){\footnotesize{$v_{L_{C_l^{(1)}}(\Lambda_{0})}$}}
\put(9.3,6.9){\footnotesize{$v_{L_{C_l^{(1)}}(\Lambda_{0})}$}}
\put(9.3,8.9){\footnotesize{$v_{L_{C_l^{(1)}}(\Lambda_{0})}$}}
\put(-1.1,6.3){\footnotesize{$r_i^{(2k-3)}$}}
\put(-1.1,7.3){\footnotesize{$r_i^{(2k-2)}$}}
\put(-1.1,8.3){\footnotesize{$r_i^{(2k-1)}$}}
\put(-1.1,9.3){\footnotesize{$r_i^{(2k)}$}}
\end{picture}
\bigskip
\caption{Sketch of projection $\pi_{\mathfrak{R}}$ for color $1 \leq i\leq l-1$}\label{slika4C}
\end{figure}
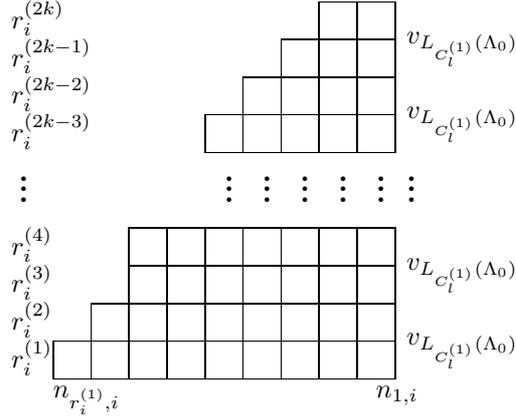
where two generating functions $x_{\alpha_{i}}(z_{p,i})$ ($1\leq p \leq r_i^{(1)}$) \enquote{are placed at} the  first $\frac{n_{p,i}}{2}$ tensor factors (from right to left) if $n_{p,i}$ is an even number and if $n_{p,i}$ is an odd number, then two generating functions  $x_{\alpha_{i}}(z_{p,i})$ \enquote{are placed at} the first $\frac{n_{p,i}-1}{2}$ tensor factors (from right to left), and the last generating  function $x_{\alpha_{i}}(z_{p,i})$ \enquote{is placed at} $\frac{n_{p,i}-1}{2}+1$ tensor factor. 
Therefore, for a given monomial $b\in B_{W_{L_{C_l^{(1)}}(k{\Lambda}_{0})}}$ 
\begin{equation}\label{eq:polCl}
b=x_{n_{r_{1}^{(1)},1}\alpha_{1}}(m_{r_{1}^{(1)},1})\cdots      x_{n_{1,1}\alpha_{1}}(m_{1,1})\cdots x_{n_{r_{l}^{(1)},l}\alpha_{l}}(m_{r_{l}^{(1)},l})\cdots  x_{n_{1,l}\alpha_{1}}(m_{1,l})
\end{equation}
colored with color-type $(r_{1}, \ldots ,r_{l}),$ charge-type $\mathfrak{R}'$ and dual-charge-type $\mathfrak{R}$, the projection is a coefficient of the projection of the generating function  (\ref{projekcijaCl}) which we denote as $$\pi_{\mathfrak{R}}bv_{L_{C_l^{(1)}}(k\Lambda_{0})}.$$

\subsubsection{\textbf{A coefficient of an intertwining operator}}\label{ss:intertCl}
With $A_{\omega_l}$, 
\begin{equation*} 
A_{\omega_l}=\text{Res}_z \ z^{-1}I(v_{L_{C_l^{(1)}}(\Lambda_l)}, z),
\end{equation*}
we denote the coefficient of an intertwining operator $I(\cdot , z)$ of type \begin{equation*} 
\binom{L_{C_l^{(1)}}(\Lambda_l)}{L_{C_l^{(1)}}(\Lambda_l) \ L_{C_l^{(1)}}( \Lambda _{0})},
\end{equation*}
defined by  
\begin{align}\label{al:S33Cl}
I(w,z)v=\exp(zL(-1))Y(v,-z)w, \ \  w \in L_{C_l^{(1)}}(\Lambda_l), v \in L_{C_l^{(1)}}( \Lambda _{0}),
\end{align}
which commutes with the quasi-particles (see \ref{ss:intertB}).
From definition (\ref{al:S33Cl}), we have
\begin{equation}\label{eq:S34Cl}
A_{\omega_l}v_{L_{C_l^{(1)}}(\Lambda_0)}=v_{L_{C_l^{(1)}}(\Lambda_l)}.
\end{equation}
\par Let $s\leq k$. As in the case $B_l^{(1)}$ we consider the operator on $L_{C_l^{(1)}}({\Lambda}_{0})\otimes \cdots \otimes L_{C_l^{(1)}}({\Lambda}_{0})$, which we denote by the same symbol as in \ref{ss:intertB}
\begin{equation*}
A_s=1\otimes\cdots \otimes  A_{\omega_{l}} \otimes \underbrace{1 \otimes \cdots \otimes 1}_{s-1 \ \text{factors}}.
\end{equation*}
Set $b\in B_{W_{L_{C_l^{(1)}}(k{\Lambda}_{0})}}$ as in (\ref{eq:polCl}). From the consideration in Subsection \ref{ss:projCl}, it follows that 
\begin{equation*}A_s\pi_{\mathfrak{R}}bv_{L_{C_l^{(1)}}(k{\Lambda}_{0})}
\end{equation*}
is the coefficient of  
\begin{equation*}
A_s\pi_{\textsl{\emph{R}}}x_{n_{r_{1}^{(1)},1}\alpha_{1}}(z_{r_{1}^{(1)},1})\cdots x_{s\alpha_{l}}(z_{1,l})v_{L_{C_l^{(1)}}(k{\Lambda}_{0})},
\end{equation*}
where operator $A_{\omega_{l}}$ acts only on the $s$-th tensor factor from the right  
\begin{equation*}\cdots \otimes  x_{n^{(s)}_{r^{(2s-1)}_{1},1}\alpha_{1}}(z_{r_{1}^{(2s-1)},1})\cdots   x_{n^{(s)}_{r^{(2s)}_{1},1}\alpha_{1}}(z_{r^{(2s)}_{1},1})\cdots  x_{n^{(s)}_{1,1}\alpha_{2}}(z_{1,1})\end{equation*}
\begin{equation*}\cdots x_{n^{(s)}_{r^{(s)}_{l},l}\alpha_{l}}(z_{r_{l}^{(s)},l})\cdots  x_{\alpha_{1}}(z_{1,l})v_{L_{C_l^{(1)}}(\Lambda_{0})}\otimes \cdots ,\end{equation*}
where $ 0 \leq  n^{(s)}_{p,l} \leq 1$, for $1\leq  p \leq  r^{(s)}_{l}$ and $ 0 \leq  n^{(s)}_{p,i} \leq 2$, for $1\leq p \leq    r^{(2s-1)}_{i}$, $1 \leq i \leq l-1$ (see (\ref{projekcijaCl})).  
Therefore, in the $s$-th tensor factor from the right, we have
\begin{equation*} \cdots \otimes x_{n^{(s)}_{r^{(2s-1)}_{1},1}\alpha_{1}}(z_{r_{1}^{(2s-1)},1})\cdots     x_{n^{(s)}_{r^{(2s)}_{1},1}\alpha_{1}}(z_{r^{(2s)}_{1},1})\cdots  x_{n^{(s)}_{1,1}\alpha_{1}}(z_{1,1})\end{equation*}
\begin{equation}\cdots \label{TanjaCl}x_{n^{(s)}_{r^{(s)}_{l},l}\alpha_{l}}(z_{r_{l}^{(s)},l})\cdots  x_{\alpha_{l}}(z_{1,l})v_{L_{C_l^{(1)}}(\Lambda_{0})}\otimes \cdots .
 \end{equation}

\subsubsection{\textbf{Simple current operator \texorpdfstring{$e_{\omega_{l}}$}{eomegal}}}\label{ss:currCl}
For $\omega_l \in \mathfrak{h}$ we denoted by $e_{\omega_{l}}$ a simple current operator (cf. \cite{Li2}) between the level $1$ standard modules
\begin{equation*}
e_{\omega_{l}}:L_{C_l^{(1)}}( \Lambda _{0})\rightarrow L_{C_l^{(1)}}(\Lambda_l),
\end{equation*}
such that
\begin{equation*}
e_{\omega_{l}}{v}_{L_{C_l^{(1)}}(\Lambda_{0})}=v_{L_{C_l^{(1)}}(\Lambda_{l})}
\end{equation*}
and 
\begin{equation}\label{S341Cl}
x_{\alpha}(z)e_{\omega_{l}}=e_{\omega_{l}}z^{\alpha(\omega_{l})}x_{\alpha}(z),
\end{equation}
for all $\alpha \in R$. 
\par We can rewrite (\ref{TanjaCl}) as 
\begin{align*} 
\cdots \otimes x_{n^{(s)}_{r^{(2s-1)}_{1},1}\alpha_{1}}(z_{r_{1}^{(2s-1)},1})\cdots     x_{n^{(s)}_{r^{(2s)}_{1},1}\alpha_{1}}(z_{r^{(2s)}_{1},1})\cdots  x_{n^{(s)}_{1,1}\alpha_{1}}(z_{1,1})\\
 \nonumber
x_{n^{(s)}_{r^{(2s-1)}_{l-1},l-1}\alpha_{l-1}}(z_{r_{l-1}^{(2s-1)},l-1})\cdots     x_{n^{(s)}_{r^{(2s)}_{l-1},l-1}\alpha_{l-1}}(z_{r^{(2s)}_{l-1},l-1})\cdots  x_{n^{(s)}_{1,l-1}\alpha_{1}}(z_{1,l-1})\\
\end{align*}
\begin{equation*}
x_{n^{(s)}_{r_{l}^{(k)},l}\alpha_{l}}(z_{r_{l}^{(k)},l})z_{r_{l}^{(k)},l}\cdots x_{\alpha_{l}}(z_{1,l})z_{1,l} e_{\omega_{l}}{v}_{L_{C_l^{(1)}}(\Lambda_{0})}\otimes \cdots.\end{equation*}
By taking the corresponding coefficients, we have
\begin{equation*}
A_s\pi_{\mathfrak{R}}bv_{L_{C_l^{(1)}}(k{\Lambda}_{0})}=B_s\pi_{\mathfrak{R}}b^{+}{v}_{L_{C_l^{(1)}}(\Lambda_{0})}
\end{equation*}
where
\begin{equation*}
B_s=1\otimes\cdots \otimes 1 \otimes      e_{\omega_{l}} \otimes \underbrace{1 \otimes \cdots \otimes 1}_{s-1\ \text{factors}}
\end{equation*}
and where the monomial $b^{+}$:
\begin{align*}
b^{+}=b^+(\alpha_{1})\cdots b^{+}(\alpha_{l}),
\end{align*}
is such that 
\begin{align*}
b^+(\alpha_{i})&=b(\alpha_{i}),\ \ 1 \leq i \leq l-1&\\
b^+(\alpha_{l})&=x_{n_{r_{l}^{(1)},l}\alpha_{l}}(m_{r_{l}^{(1)},l}+1)\cdots x_{s\alpha_{l}}(m_{1,1}+1)&\\
&=x_{n_{r_{l}^{(1)},l}\alpha_{l}}(m^{+}_{r_{l}^{(1)},l})\cdots x_{s\alpha_{l}}(m^{+}_{1,l}).&
\end{align*}

\subsubsection{\textbf{Operator \texorpdfstring{$e_{\alpha_{l}}$}{ealphal}}}\label{S66Cl}
On the level $1$ standard module $L_{C_l^{(1)}}( \Lambda _{0})$ we define the \enquote{Weyl group translation} operator  $e_{\alpha_l}$ 
\begin{equation*}
 e_{\alpha_l}=\exp\  x_{-\alpha_l}(1)\exp\  (- x_{\alpha_l}(-1))\exp\  x_{-\alpha_l}(1) \exp\ x_{\alpha_l}(0)\exp\   (-x_{-\alpha_l}(0))\exp\ x_{\alpha_l}(0).\end{equation*}
The properties of $e_{\alpha_{l}}$, which we use in the proof of linear independance are described in the following lemma.
  \begin{lem}
\begin{itemize}
	\item [a)]  $ e_{\alpha_l}v_{L_{C_l^{(1)}}(\Lambda_0)}=-x_{\alpha_{l}}(-1)v_{L_{C_l^{(1)}}(\Lambda_0)}$
\item [b)] 
$x_{\alpha_l}(z)e_{\alpha_l}=z^2e_{\alpha_l}x_{\alpha_l}(z)$
\item [c)] 
$x_{\alpha_{l-1}}(z)e_{\alpha_l}=z^{-1}e_{\alpha_l}x_{\alpha_{l-1}}(z)$
\item [d)]
$x_{\alpha_i}(z)e_{\alpha_l}= e_{\alpha_l}x_{\alpha_i}(z)$, $1 \leq i \leq l-2$.
\end{itemize}
\end{lem}
\begin{flushright}
$\square$
\end{flushright}
Let $b$ be a monomial   
\begin{align}\label{istoCl}
b&=b(\alpha_{1})\cdots b(\alpha_{l-1})b(\alpha_{l})x_{s\alpha_{l}}(-s)&\\
\nonumber
&=x_{n_{r^{(1)}_{1},1}\alpha_{1}}(m_{r^{(1)}_{1},1})\cdots     x_{n_{1,1}\alpha_{1}}(m_{1,1})\cdots &\\
\nonumber 
& x_{n_{r^{(1)}_{l-1},l-1}\alpha_{l-1}}(m_{r^{(1)}_{l-1},l-1})\cdots     x_{n_{1,l-1}\alpha_{l-1}}(m_{1,l-1}) x_{n_{r^{(1)}_{l},l}\alpha_{l}}(m_{r^{(1)}_{l},l})\cdots  x_{n_{2,l}\alpha_{l}}(m_{2,l})x_{s\alpha_{l}}(-s),&
\end{align}
of dual-charge-type 
\begin{equation*}
\mathfrak{R}=\left(r^{(1)}_{1},\ldots, r^{(2k)}_{1};\ldots; r^{(1)}_{l-1},\ldots, r_{l-1}^{(2k)};r^{(1)}_{l},\ldots, r_{l}^{(s)},0 \ldots, 0\right).
\end{equation*}
As in Subsection \ref{ss:projCl}, let $\pi_{\mathfrak{R}}$ be the projection of principal subspace $W_{L_{C_l^{(1)}}(\Lambda_{0})}\otimes \cdots\otimes W_{L_{C_l^{(1)}}(\Lambda_{0})}$ on the vector space  
\begin{align*}
{W_{L_{C_l^{(1)}}(\Lambda_{0})}}_{(\mu^{(k)}_{1};\ldots ;\mu^{(k)}_{l-1} ;0)}\otimes \cdots \otimes {W_{L_{C_l^{(1)}}(\Lambda_{0})}}_{(\mu^{(s+1)}_{1};\ldots ;\mu^{(s+1)}_{l-1} ; 0)}\otimes\\ \otimes {W_{L_{C_l^{(1)}}(\Lambda_{0})}}_{(\mu^{(s)}_{1};\ldots ;\mu^{(s)}_{l-1} ; r_{1}^{(s)})}\otimes \cdots\otimes {W_{L_{C_l^{(1)}}(\Lambda_{0})}}_{(\mu^{(1)}_{1};\ldots ;\mu^{(1)}_{1} ; r_{1}^{(1)})}. \end{align*}
The projection \begin{equation*}
\pi_{\mathfrak{R}}b\left(v_{L_{C_l^{(1)}}(\Lambda_{0})}\otimes \cdots\otimes v_{L_{C_l^{(1)}}(\Lambda_{0})}\right)\end{equation*}
of the monomial vector $b\left(v_{L_{C_l^{(1)}}(\Lambda_{0})}\otimes
 \cdots \otimes  v_{L_{C_l^{(1)}}(\Lambda_{0})}\right)$ is a coefficient of the generating function
\begin{equation*}\pi_{\mathfrak{R}}x_{n_{r_{1}^{(1)},1}\alpha_{1}}(z_{r_{1}^{(1)},1})\cdots
 x_{n_{1,1}\alpha_{1}}(z_{1,1})\cdots x_{n_{r_{l}^{(1)},l}\alpha_{l}}(z_{r_{l}^{(1)},l}) 
\cdots  x_{n_{2,l}\alpha_{l}}(z_{2,l})\end{equation*}
\begin{equation*}\left(v_{L_{C_l^{(1)}}(\Lambda_{0})}\otimes
 \cdots \otimes  v_{L_{C_l^{(1)}}(\Lambda_{0})}\otimes x_{\alpha_{l}}(-1)v_{L_{C_l^{(1)}}(\Lambda_{0})}\otimes \cdots \otimes  x_{\alpha_{l}}(-1)v_{L_{C_l^{(1)}}(\Lambda_{0})}\right)\end{equation*}
\begin{equation*}=C\cdots x_{n^{(k)}_{r^{(2k-1)}_{1},1}\alpha_{1}}(z_{r_{1}^{(2k-1)},1}) \cdots 
 x_{n^{(k)}_{r^{(2k)}_{1},1}\alpha_{1}}(z_{r^{(2k)}_{1},1})\cdots
 x_{n^{(k)}_{1,1}\alpha_{1}}(z_{1,1})\cdots \end{equation*}
\begin{equation*}x_{n^{(k)}_{r^{(2k-1)}_{l-1},l-1}\alpha_{l-1}}(z_{r_{l-1}^{(2k-1)},l-1}) \cdots 
 x_{n^{(k)}_{r^{(2k)}_{l-1},l-1}\alpha_{l-1}}(z_{r^{(2k)}_{l-1},l-1})\cdots
 x_{n^{(k)}_{1,l}\alpha_{l}}(z_{1,l})v_{L_{C_l^{(1)}}(\Lambda_{0})}\end{equation*}
\begin{equation*}\otimes \cdots \otimes \end{equation*}
\begin{equation*}\otimes x_{n^{(s)}_{r^{(2s-1)}_{1},1}\alpha_{1}}(z_{r_{1}^{(2s-1)},1}) \cdots 
 x_{n^{(s)}_{r^{(2s)}_{1},1}\alpha_{1}}(z_{r^{(2s)}_{1},1})\cdots  x_{n^{(s)}_{1,1}\alpha_{1}}(z_{1,1})\end{equation*}
\begin{equation*}\cdots x_{n^{(s)}_{r^{(2s-1)}_{l-1},l-1}\alpha_{l-1}}(z_{r_{l-1}^{(2s-1)},l-1}) \cdots 
 x_{n^{(s)}_{r^{(2s)}_{l-1},l-1}\alpha_{l-1}}(z_{r^{(2s)}_{l-1},l-1})\cdots
 x_{n^{(s)}_{1,l-1}\alpha_{l-1}}(z_{1,l-1})\end{equation*}\begin{equation*}  x_{n^{(s)}_{r^{(s)}_{l},l}\alpha_{l}}(z_{r_{l}^{(s)},l})\cdots    x_{n^{(s)}_{2,l}\alpha_{l}}(z_{2,l})e_{\alpha_{l}}v_{L_{C_l^{(1)}}(\Lambda_{0})}\end{equation*}
\begin{equation*}\otimes \cdots \otimes \end{equation*}
\begin{equation*}\otimes x_{n_{r^{(1)}_{1},1}^{(1)}\alpha_{1}}(z_{r_{1}^{(1)},1})\cdots 
 x_{n_{r^{(2)}_{1},1}^{(1)}\alpha_{1}}(z_{r_{1}^{(2)},1})\cdots x_{n_{2,1}^{(1)}\alpha_{1}}(z_{2,1})
x_{n_{1,1}^{(1)}\alpha_{1}}(z_{1,1})\end{equation*}
\begin{equation*}\cdots x_{n^{(1)}_{r^{(1)}_{l-1},l-1}\alpha_{l-1}}(z_{r_{l-1}^{(1)},l-1}) \cdots 
 x_{n^{(1)}_{r^{(2)}_{l-1},l-1}\alpha_{l-1}}(z_{r^{(2)}_{l-1},l-1})\cdots
x_{n^{(1)}_{1,l-1}\alpha_{l-1}}(z_{1,l-1})\end{equation*}\begin{equation*}x_{n_{r^{(1)}_{l},l}^{(1)}\alpha_{l}}(z_{r_{l}^{(1)},l})\cdots   x_{n{_{2,l}^{(1)}\alpha_{l}}}(z_{2,l})e_{\alpha_{l}}v_{L_{C_l^{(1)}}(\Lambda_{0})},\end{equation*}
(see (\ref{projekcijaCl})).  
Now if we shift operator $1 \otimes \cdots \otimes  e_{\alpha_{l}} \otimes e_{\alpha_{l}} \otimes \cdots \otimes  e_{\alpha_{l}}$ all the way to the left we get
 \begin{equation*}
(1 \otimes \cdots \otimes  e_{\alpha_{l}} \otimes e_{\alpha_{l}} \otimes \cdots \otimes  e_{\alpha_{l}})\pi_{\mathfrak{R}'} b'\left(v_{L_{C_l^{(1)}}(\Lambda_{0})}\otimes \ldots\otimes v_{L_{C_l^{(1)}}(\Lambda_{0})}\right), 
\end{equation*} 
where 
\begin{equation*}
\mathfrak{R}'=\left(r^{(1)}_{1},\ldots, r^{(2s)}_{1};\cdots; r^{(1)}_{l}-1,\ldots, r_l^{(s)}-1\right)
\end{equation*}
and
\begin{align*}
b'&=b'(\alpha_{ 1})\cdots b'(\alpha_{l-1})b'(\alpha_{l})&\\
&=x_{n_{r^{(1)}_{1}, 1}\alpha_{ 1}}(m_{r^{(1)}_{1}, 1})\cdots x_{n_{1, 1}\alpha_{ 1}}(m_{1, 1})&\\
&=x_{n_{r^{(1)}_{l-1},l-1}\alpha_{l-1}}(m_{r^{(1)}_{l-1},l-1}-n^{(1)}_{r_{l-1}^{(1)},l-1}-\cdots-n^{(s)}_{r_{1}^{(1)},1})&\\
& \ \ \ \ \ \ \ \ \ \ \ \ \ \ \ \ \ \ \ \ \ \ \ \ \ \  \cdots x_{n_{1,l-1}\alpha_{l-1}}(m_{1,l-1}-n^{(1)}_{1,l-1}-\cdots-n^{(s)}_{1,l-1})&\\
& \ \ \ \ \ \ \ \ \ \ \ \ \ \ \ \ \ \ \ \ \ \ \ \ \ \ \  x_{n_{r^{(1)}_{l}\alpha_{l}}}(m_{r^{(1)}_{l},l}+2n_{r^{(1)}_{l}})\cdots   x_{n_{2,l}\alpha_{l}}(m_{2,l}+2n_{1,l})&\\
&=x_{n_{r^{(1)}_{1}, 1}\alpha_{ 1}}(m_{r^{(1)}_{1}, 1})\cdots x_{n_{1, 1}\alpha_{ 1}}(m_{1, 1})&\\
&x_{n_{r^{(1)}_{l-1},l-1}\alpha_{l-1}}(m'_{r^{(1)}_{l-1},l-1})\cdots 
 x_{n_{1,l-1}\alpha_{l-1}}(m'_{1,l-1})x_{n_{r^{(1)}_{l}\alpha_{l}}}(m'_{r^{(1)}_{l},l})\cdots x_{n_{2,l}\alpha_{l}}(m'_{2,l}).&
\end{align*}
\begin{lem}
\label{S66PCl}
If $b$ (\ref{istoCl}) is an element of the set $B_{W_{L_{C_l^{(1)}}(k\Lambda_{0})}}$, then the monomial $b'$, from the above consideration, is an element of the set $B_{W_{L_{C_l^{(1)}}(k\Lambda_{0})}}$. 
\end{lem}
\begin{proof}
This Lemma easy follows 
by considering the possible situations for $n_{p,l}$, $2\leq p \leq r^{(1)}_{l}$ and $n_{p,l-1}$, $ 1\leq p \leq r^{(1)}_{l-1}$  
from which follows that $m'_{p,i}$, $l-1 \leq i \leq l$ comply the defining conditions of the set $B_{W_{L_{C_l^{(1)}}(k\Lambda_{0})}}$.
\end{proof}

\subsubsection{\textbf{The proof of linear independence of the set \texorpdfstring{$\mathfrak{B}_{W_{L_{C_l^{(1)}}(k\Lambda_{0})}}$}{BWLCl(1)(kLambda0)}}}\label{S67Cl}
In this section, we prove the following theorem: 
\begin{thm}\label{S67T1Cl}
The set $\mathfrak{B}_{W_{L_{C_l^{(1)}}(k\Lambda_{0})}}$ forms a basis for the principal subspace $W_{L_{C_l^{(1)}}(k\Lambda_{0})}$ of $L_{C_l^{(1)}}\left(k\Lambda_{0}\right)$.
\end{thm}
\begin{proof}
Assume that we have
\begin{equation}\label{S6752Cl}
\sum_{a \in A}
c_{a}b_av_{L_{C_l^{(1)}}(k\Lambda_{0})}=0, 
\end{equation}
where $A$ is a finite non-empty set and 
\begin{align*}
b_a\in  B_{W_{L_{C_l^{(1)}}(k{\Lambda}_{0})}}.
\end{align*}
Assume that all $b_a$ are the same color-type $(r_{1}, \ldots,r_{l})$. 

Let $b$ be the smallest monomial in the linear lexicographic ordering \enquote{$<$}
\begin{align*} 
b&=b(\alpha_{1})\cdots b(\alpha_{l-1})b(\alpha_{l})\\
&=x_{n_{r_{1}^{(1)},1}\alpha_{1}}(m_{r_{1}^{(1)},1})\cdots
 x_{n_{1,1}\alpha_{1}}(m_{1,1})\cdots &\\
 &\cdots x_{n_{r_{l-1}^{(1)},l-1}\alpha_{l-1}}(m_{r_{l-1}^{(1)},l-1})\cdots x_{n_{1,l-1}\alpha_{l-1}}(m_{1,l-1})x_{n_{r_{l}^{(1)},l}\alpha_{l}}(m_{r_{l}^{(1)},l})\cdots x_{n_{1,l}\alpha_{l}}(-j),
\end{align*}
of charge-type 
\begin{equation}\label{S675444Cl}
\left(n_{r_{1}^{(1)},1}, \ldots , 
 n_{1,1};\cdots; n_{r_{l-1}^{(1)},l-1}, \ldots ,n_{1,l-1}; n_{r_{l}^{(1)},l}, \ldots ,n_{1,l}\right)
\end{equation}
such that $c_{a}\neq 0$ and such that, for every $b_{a}$ in (\ref{S6752Cl}), 
we have
\begin{equation*}m_{1,l}\geq -j.\end{equation*}
Denote by 
\begin{equation*}
\mathfrak{R}=\left( r_{1}^{(1)},\ldots , r_{1}^{(2k)};\ldots ; r_{l-1}^{(1)},\ldots , r_{l-1}^{(n_{1,l-1})}; r_{l}^{(1)},\ldots , r_{l}^{(n_{1,l})}\right),
\end{equation*}
the dual-charge-type of $b$.  
For every $1\leq t\leq k$ such that
\begin{equation*}
\mu^{(t)}_{i}=r^{(2t)}_{i}+ r^{(2t-1)}_{i}
\end{equation*}
where $1 \leq i \leq l-1$ let  $\pi_{\mathfrak{R}}$ be the projection of $\underbrace{W_{L_{C_l^{(1)}}({\Lambda}_{0})}\otimes \ldots\otimes W_{L_{C_l^{(1)}}({\Lambda}_{0})}}_{k \ \text{factors}}$ on the vector space 
\begin{equation*}
{W_{L_{C_l^{(1)}}(\Lambda_{0})}}_{(\mu^{(k)}_{1};\ldots ;0)}\otimes \ldots  \otimes {W^C_{L(\Lambda_{0})}}_{(\mu^{(n_{1,l}+1)}_{1};\ldots; 0)}\otimes {W_{L_{C_l^{(1)}}(\Lambda_{0})}}_{(\mu^{(n_{1,l})}_{1};\ldots ; r_{1}^{(n_{1,l})})}\otimes  \cdots {W_{L_{C_l^{(1)}}(\Lambda_{0})}}_{(\mu^{(1)}_{1};\ldots ;r_{l}^{(1)})}.
\end{equation*}
It is not hard to see that the projection $\pi_{\mathfrak{R}}$ maps to zero all monomial vectors $b_av_{L_{C_l^{(1)}}(k{\Lambda}_{0})}$ such that $b_a$ has a larger charge-type in the linear lexicographic ordering \enquote{$<$} than (\ref{S675444Cl}). So, in (\ref{S6752Cl}) we have a projection of $b_{a}v_{L_{C_l^{(1)}}(k{\Lambda}_{0})}$, where $b_{a}$ are of charge-type (\ref{S675444Cl})
\begin{equation}\label{S6733Cl}
\sum_{a\in A} c_a\pi_{\mathfrak{R}}b_{a}v_{L_{C_l^{(1)}}(k\Lambda_{0})}=0. 
\end{equation}
On (\ref{S6733Cl}), we act with operators
\begin{equation*}
A_{n_{1,l}}=1\otimes\cdots \otimes  A_{\omega_{l}} \otimes \underbrace{1 \otimes \cdots \otimes 1}_{n_{1,l}-1 \ \text{factors}}
\end{equation*}
 and
\begin{equation*}
B_{n_{1,l}}=1\otimes\cdots \otimes  e_{\omega_{l}} \otimes \underbrace{1 \otimes \cdots \otimes 1}_{n_{1,l}-1 \ \text{factors}}
\end{equation*}
 until $j$ becomes $- n_{1,l}$. 
In that case, we get 
\begin{equation*}
\sum_{a}
c_{a}\pi_{\mathfrak{R}}b_{a}(\alpha_{1})\cdots b_{a}(\alpha_{l-1})b_{a}^{+}(\alpha_{l})x_{n_{1,l}\alpha_{l}}(-n_{1,l})v_{L_{C_l^{(1)}}(k\Lambda_{0})}=0,
\end{equation*}
where $b_{a}^{+}(\alpha_{l})x_{n_{1,l}\alpha_{l}}(-n_{1,l})$ is of color $i=l$ and 
\begin{equation*}b_{a}(\alpha_{1})\cdots b_{a}(\alpha_{l-1})b_{a}^{+}(\alpha_{l})x_{n_{1,l}\alpha_{l}}(-n_{1,l})v_{L_{C_l^{(1)}}(k\Lambda_{0})} \in \mathfrak{B}_{W_{L_{C_l^{(1)}}(k{\Lambda}_{0})}}.\end{equation*} 
From the subsection \ref{S66Cl} follows 
\begin{align*}
&\pi_{\mathfrak{R}}b_a(\alpha_{1})\cdots b_{a}(\alpha_{l-1})b_a^{+}(\alpha_{l})x_{n_{1,l}\alpha_{l}}(-n_{1,l})v_{L_{C_l^{(1)}}(k\Lambda_{0})}&\\
&=(1\otimes\cdots 1\otimes  e_{\alpha_{l}}\otimes e_{\alpha_{l}} \cdots \otimes   e_{\alpha_{l}})b_a (\alpha_{1})\cdots b'(\alpha_{l-1})b'(\alpha_{l})v_{L_{C_l^{(1)}}(k\Lambda_{0})},&
\end{align*} where $b_a (\alpha_{1})\cdots b_a'(\alpha_{l})$ does not have a quasi-particle of charge $n_{1,l}$. 
 $b(\alpha_{1})\cdots b'(\alpha_{l-1})b'(\alpha_{l})$ is of dual-charge-type
\begin{equation*}
\left( r_{1}^{(1)}, \ldots ,r_{1}^{(2k)}; \ldots ;  r_{l-1}^{(1)}, \ldots ,r_{l-1}^{(2k)};r_{l}^{(1)}-1, \ldots , r_{l}^{(n_{1,l})}-1\right),
 \end{equation*}
and charge-type
\begin{equation*}
\left(n_{r_{1}^{(1)},1}, \ldots , n_{1,1};\ldots ; n_{r_{l-1}^{(1)},l-1}, \ldots ,n_{1,l-1};  n_{r_{l}^{(1)},l}, \ldots 
,n_{2,l}\right),
\end{equation*}
such that\begin{equation*}
\left(n_{r_{i}^{(1)},1}, \ldots , n_{1,1};\ldots ; n_{r_{l}^{(1)},l}, \ldots ,n_{2,l}\right) < 
\left(n_{r_{2}^{(1)},1}, \ldots , n_{1,1}; \ldots ; n_{r_{l}^{(1)},l}, \ldots ,n_{2,l},n_{1,l}\right).
\end{equation*} 
From Remark \ref{S66PCl}, it follows that we get elements from the set $\mathfrak{B}_{W_{L_{C_l^{(1)}}(k{\Lambda}_{0})}}$.
\par We apply the described processes on \ref{S6733Cl}, until we get monomial \enquote{colored} with colors $1 \leq i \leq l-1$. 
Assume that after a finite number of steps we get 
\begin{equation}\label{S6733Cl1}
\sum_{a\in A} c_a\pi_{\mathfrak{R}}b_a(\alpha_{1})\cdots b_a(\alpha_{l-1})v_{L_{C_l^{(1)}}(k{\Lambda}_{0})}=0, 
\end{equation}
where 
\begin{equation*}b_a(\alpha_{1})\cdots b_a(\alpha_{l-1}) \in \mathfrak{A}_{\mathfrak{R}^{-}}\subset \mathfrak{B}_{W_{L_{C_l^{(1)}}(k\Lambda_{0})}}.
\end{equation*}
By $\mathfrak{A}_{\mathfrak{R}^{-}}$ we denote the set of monomial vectors 
of dual-charge type   
\begin{equation*}
\mathfrak{R}^{-}=\left( r_{1}^{(1)}, \ldots ,r_{1}^{(2k)}; \ldots ;  r_{l-1}^{(1)}, \ldots ,r_{l-1}^{(2k)}\right).
\end{equation*}
From the condition $x_{3\alpha_i}(z)=0$, ($1 \leq i \leq l-1$), it follows that monomial vectors 
in \ref{S6733Cl1} are from vector space 
\begin{equation*}
\underbrace{W_{L_{A_{l-1}^{(1)}}(2\Lambda_{0})}\otimes  \ldots  \otimes  W_{L_{A_{l-1}^{(1)}}(2\Lambda_{0})}}_{k \ \text{factors}},
\end{equation*}
where $W_{L_{A_{l-1}^{(1)}} (2\Lambda_{0})}={W_{L_{C_l^{(1)}}(\Lambda_{0})}}_{0 \cdot  \alpha_l}$ is the principal subspace of the standard module $L_{A_{l-1}^{(1)}}(2\Lambda_{0})\subset L_{A_{l-1}^{(1)}}(\Lambda_{0})$ of the affine Lie algebra $A_{l-1}^{(1)}$, with the highest weight vector $v_{L_{A_{l-1}^{(1)}}(\Lambda_0)}=v_{L_{C_{l}^{(1)}}(\Lambda_0)}$. Denote by 
\begin{align*}
{W_{L_{A_{l-1}^{(1)}}\left(2\Lambda_{0}\right)}}_{(u_1, \ldots, u_{l-1})}= {W_{L_{C_{l}^{(1)}}(\Lambda_{0})}}_{u_1\alpha_1+\cdots + u_{l-1}\alpha_{l-1}},
\end{align*}
$\mathfrak{h}$-weighted subspace of $W_{L_{A_{l-1}^{(1)}}(2\Lambda_{0})}$.   
On every factor in the tensor product $W_{L_{A_{l-1}^{(1)}}(2\Lambda_{0})}\otimes \cdots \otimes W_{L_{A_{l-1}^{(1)}}(2\Lambda_{0})}$ of $k$ principal subspaces       $W_{L_{A_{l-1}^{(1)}}(2\Lambda_{0})}$, we have embedding 
\begin{equation*}
{W_{L_{A_{l-1}^{(1)}} (2\Lambda_{0})}}_{(\mu^{(p)}_{1};\ldots; \mu^{(p)}_{l-1})}\hookrightarrow \sum_{\stackrel{u_1,\ldots, u_{l-1},v_1,\ldots, v_{l-1} \in 
\mathbb{N}}{\mu^{(p)}_{i}=u_i+v_i, 1 \leq i \leq l-1}}{W_{L_{A_{l-1}^{(1)}} (\Lambda_{0})}}_{(u_{1};\ldots; u_{l-1})}\otimes  {W_{L_{A_{l-1}^{(1)}} (\Lambda_{0})}}_{((v_{1};\ldots; v_{l-1})},
\end{equation*}
 for $1\leq p \leq k$. 
\par Denote by $\pi'_{\mathfrak{R}^-}$ the projection of the vector space
\begin{equation*}
W_{L_{A_{l-1}^{(1)}}(2\Lambda_{0})}\otimes  \cdots  \otimes  W_{L_{A_{l-1}^{(1)}} (2\Lambda_{0})}\end{equation*}
 on subspace
\begin{equation*}
{W_{L_{A_{l-1}^{(1)}}(\Lambda_{0})}}_{(r_{1}^{(2k)}; \ldots ;r_{l-1}^{(2k)} )}\otimes  {W_{L_{A_{l-1}^{(1)}} (\Lambda_{0})}}_{(r_{1}^{(2k-1)}; \ldots ;r_{l-1}^{(2k-1)})}\otimes  \cdots  \otimes       {W_{L_{A_{l-1}^{(1)}} (\Lambda_{0})}}_{(r_{1}^{(2)}; \ldots ;r_{l-1}^{(2)})}\otimes  {W_{L_{A_{l-1}^{(1)}} (\Lambda_{0})}}_{(r_{1}^{(1)}; \ldots ;r_{l-1}^{( 1)})}.
\end{equation*}
In particular, extending the above projection on the space of formal series with coefficients in
\begin{equation*}
\underbrace{W_{L_{A_{l-1}^{(1)}} (\Lambda_{0})}\otimes \cdots \otimes  W_{L_{A_{l-1}^{(1)}} (\Lambda_{0})}}_{2k \ \text{factors}}\end{equation*}
from the condition $x_{2\alpha_i}(z)=0$ ($1 \leq i \leq l-1$) follows
\begin{align*}
&\pi'_{\mathfrak{R}^-}\left(\pi_{\mathfrak{R}}b_a(\alpha_{1})\cdots b_a(\alpha_{l-1})v_{L_{C_l^{(1)}}k\Lambda_{0}}\right)&\\
 \in & {W_{L_{A_{l-1}^{(1)}} (\Lambda_{0})}}_{(r_{1}^{(2k)}; \ldots ;r_{l-1}^{(2k)} )}\otimes  {W_{L_{A_{l-1}^{(1)}} (\Lambda_{0})}}_{(r_{1}^{(2k-1)}; \ldots ;r_{l-1}^{(2k-1)})}\otimes  \cdots  \otimes&\\
& \otimes        {W_{L_{A_{l-1}^{(1)}} (\Lambda_{0})}}_{(r_{1}^{(2)}; \ldots ;r_{l-1}^{(2)})}\otimes  {W_{L_{A_{l-1}^{(1)}} (\Lambda_{0})}}_{(r_{1}^{(1)}; \ldots ;r_{l-1}^{( 1)})}.&
\end{align*}
Georgiev showed that 
\begin{equation*}
\pi'_{\mathfrak{R}^-}\circ {\pi_{\mathfrak{R}}}_{\left|_{W_{L_{A_{l-1}^{(1)}}(2k\Lambda_{0})}}\right.}\mathfrak{A}_{\mathfrak{R}^-}\end{equation*} is a  linearly independent set. Thus, it follows that the set $c_a=0$ and the desired theorem follows.  
\end{proof}

\subsection{Characters of the principal subspace \texorpdfstring{$W_{L_{C_l^{(1)}}(k\Lambda_{0})}$}{WLCl(1)(kLambda0)}}
In determining the character formulas of $W_{L_{C_l^{(1)}}(k\Lambda_{0})}$ we will use the expressions in Lemma \ref{S7L1Cl}.
\begin{lem}\label{S7L1Cl}
For the given color-type $(r_{1},\ldots, r_{l})$, charge-type      
\begin{equation*}\left(n_{r_{1}^{(1)},1}, \ldots, n_{1,1};\ldots \right. \left. n_{r_{l}^{(1)},l}, \ldots , n_{1,l}\right)\end{equation*} and dual-charge-type   
\begin{equation*}
\left(r_{1}^{(1)},r_{1}^{(2)},\ldots ,r_{1}^{(2k)};\ldots; r_{l}^{(1)}, r_{l}^{(2)}, \ldots ,r_{l}^{(k)}\right), \end{equation*}
we have:
\begin{align}\label{S710Cl}
\sum_{p=1}^{r^{(1)}_{l-1}}\sum_{q=1}^{r^{(1)}_l}\mathrm{min}\{n_{p,l-1},2n_{q,l}\}&=\sum_{s=1}^{k}r^{(s)}_l(r_{l-1}^{(2s-1)}+r_{l-1}^{(2s)}),&\\
\label{S71novoCl}
\sum_{p=1}^{r^{(1)}_{i}}\sum_{q=1}^{r^{(1)}_{i+1}}\mathrm{min}\{n_{p,i},n_{q,i+1}\}&=\sum_{s=1}^{2k}r^{(s)}_ir_{i+1}^{(s)},&\\
\label{S711Cl}\sum_{p=1}^{r_{l}^{(1)}} (\sum_{p>p'>0}2\mathrm{min} \{ n_{p,l},
n_{p',l}\}+n_{p,l})&= \sum_{s=1}^{k}r^{(s)^{2}}_{l},\\
\label{S712Cl}\sum_{p=1}^{r_{i}^{(1)}} (\sum_{p>p'>0}2\mathrm{min} \{ n_{p,i},
n_{p',i}\}+n_{p,i})&= \sum_{s=1}^{2k}r^{(s)^{2}}_{i}, 1 \leq i \leq l-1.&
\end{align}
\end{lem}
\begin{flushright}
$\square$
\end{flushright}
\par Now, from the definition of the set $\mathfrak{B}_{W_{L_{C_l^{(1)}}(k{\Lambda}_{0})}}$ and (\ref{S710Cl}- \ref{S712Cl}), (\ref{S7KB}) follows the character formula of $W_{L_{C_l^{(1)}}(k{\Lambda}_{0})}$:
\begin{thm}\label{uvodCl1}
\begin{align}\nonumber 
&\mathrm{ch} \  W_{L_{C_l^{(1)}}(k\Lambda_{0})}&\\
\nonumber
= &\sum_{\substack{r^{(1)}_{1}\geq \ldots \geq r^{(2k)}_{1}\geq 0}}
\frac{q^{r^{(1)^{2}}_{1}+\cdots +r^{(2k)^{2}}_{1}}}{(q)_{r^{(1)}_{1}-r^{(2)}_{1}}\cdots (q)_{r^{(2k)}_{1}}}y^{r_1}_{1}&\\
\nonumber
 &\sum_{\substack{r^{(1)}_{2}\geq \ldots \geq r^{(2u_2)}_{2}\geq 0}}
\frac{q^{r^{(1)^{2}}_{2}+\cdots +r^{(2k)^{2}}_{2}-r_1^{(1)}r_2^{(1)}-\cdots -r_1^{(2k)}r_2^{(2k)}}}{(q)_{r^{(1)}_{2}-r^{(2)}_{2}}\cdots (q)_{r^{(2k)}_{2}}}y^{r_2}_{2}&\\
\nonumber
&\ \ \ \ \ \ \ \ \cdots& \\
\nonumber
& \sum_{\substack{r^{(1)}_{l-1}\geq \ldots \geq r^{(2u_{l-1})}_{l-1}\geq 0 }}
\frac{q^{r^{(1)^{2}}_{l-1}+\cdots +r^{(2k)^{2}}_{l-1}-r_{l-2}^{(1)}r_{l-1}^{(1)}-\cdots -r_{l-2}^{(2k)}r_{l-1}^{(2k)}}}{(q)_{r^{(1)}_{l-1}-r^{(2)}_{l-1}}\cdots (q)_{r^{(2k)}_{l-1}}}y^{r_{l-1}}_{l-1}&\\
\nonumber
&\sum_{\substack{r^{(1)}_{l}\geq \ldots \geq r^{(u_{l})}_{l}\geq  0 }}
\frac{q^{r^{(1)^{2}}_{l}+\cdots +r^{(k)^{2}}_{l}-r_{l}^{(1)}(r_{l-1}^{(1)}+r_{l-1}^{(2)})
-\cdots -r_{l}^{(k)}(r_{l-1}^{(2k}+r_{l-1}^{(2k)})}}{(q)_{r^{(1)}_{l}-r^{(2)}_{l}}\cdots (q)_{r^{(k)}_{l}}}
y^{r_l}_{l}.&\\
\nonumber&& \ \ \ \ \ \ \ \ \ \ \ \ \ \ \ \ \ \ \ \ \ \ \ \ \ \ \ \ \square
\end{align}
\end{thm}

\subsection{The basis of the \texorpdfstring{$W_{N_{C_l^{(1)}}(k\Lambda_{0})}$}{BWCN(kLambda0)}}
As in the case of $B_l^{(1)}$ we can prove:
\begin{thm}\label{prop:S22CN} The set $\mathfrak{B}_{W_{N_{C_l^{(1)}}(k{\Lambda}_{0})}}=\left\{bv_{N_{C_l^{(1)}}(k\Lambda_{0})}:b \in B_{W_{N_{C_l^{(1)}}\left(k\Lambda_{0}\right)}}\right\}$, where
\begin{equation}\label{SkupNCl}
B_{W_{N_{C_l^{(1)}}\left(k\Lambda_{0}\right)}}= \bigcup_{\substack{n_{r_{1}^{(1)},1}\leq \ldots \leq n_{1,1}\\\substack{ \ldots \\n_{r_{l-1}^{(1)},l-1}\leq \ldots \leq n_{1,l-1} }\\n_{r_{l}^{(1)},l}\leq \ldots \leq n_{1,l} }}\left(\text{or, equivalently} \ \ \ 
\bigcup_{\substack{r_{1}^{(1)}\geq \cdots \geq 0\\\substack{\cdots \\r_{l-1}^{(1)}\geq \cdots \geq 
0}\\r_{l}^{(1)}\geq \cdots \geq 
0}}\right)
\end{equation}
\begin{equation*}\left\{b\right.=b(\alpha_{1})\cdots b(\alpha_{l})
=x_{n_{r_{1}^{(1)},1}\alpha_{1}}(m_{r_{1}^{(1)},1})\cdots x_{n_{1,1}\alpha_{1}}(m_{1,1})\cdots
x_{n_{r_{l}^{(1)},l}\alpha_{l}}(m_{r_{l}^{(1)},l})\cdots  x_{n_{1,l}\alpha_{l}}(m_{1,l}):\end{equation*}
\begin{align}\nonumber
\left|
\begin{array}{l}
m_{p,l}\leq  -n_{p,l} - \sum_{p>p'>0} 2 \ \text{min}\{n_{p,l}, n_{p',l}\}, \  1\leq  p\leq r_{l}^{(1)};\\
m_{p+1,l}\leq  m_{p,l}-2n_{p,l} \  \text{if} \ n_{p,l}=n_{p+1,l}, \  1\leq  p\leq r_{l}^{(1)}-1;\\
m_{p,l-1}\leq  -n_{p,l-1}+ \sum_{q=1}^{r_{l}^{(1)}}\text{min}\left\{2n_{q,l},n_{p,l-1}\right\}- \sum_{p>p'>0} 
2 \ \text{min}\{n_{p,l}, n_{p',l}\},\\
\ \ \ \ \ \ \ \ \ \ \ \ \ \ \ \ \ \ \ \ \ \ \ \ \ \ \ \ \ \ \ \ \ \ \ \ \ \ \ \ \ \ \ \ \ \ \ \ \ \ \ \ \ \ \ \ \ \ \ \ \ \ \ \ \ \ \ \ \ \ \ \ \ \ \ \ \ \ \ \ \  1\leq  p\leq r_{l}^{(1)};\\
m_{p+1,l-1} \leq   m_{p,l-1}-2n_{p,l-1} \  \text{if} \ n_{p+1,l-1}=n_{p,l-1}, \ 1\leq  p\leq r_{l-1}^{(1)}-1;\\
m_{p,i}\leq  -n_{p,i}+ \sum_{q=1}^{r_{i+1}^{(1)}}\text{min}\left\{n_{q,i+1},n_{p,i}\right\}- \sum_{p>p'>0} 
2 \ \text{min}\{n_{p,i}, n_{p',i}\},\\
\ \ \ \ \ \ \ \ \ \ \ \ \ \ \ \ \ \ \ \ \ \ \ \ \ \ \ \ \ \ \ \ \ \ \ \ \ \ \ \ \ \ \ \ \ \ \ \ \ \ \ \ \ \ \ \ \ \ \ \ \ \ \ \ 1\leq  p\leq r_{i}^{(1)}, \ 1 \leq i \leq l-2;\\
m_{p+1,i} \leq   m_{p,i}-2n_{p,i} \  \text{if} \ n_{p+1,i}=n_{p,i}, \ 1\leq  p\leq r_{i}^{(1)}-1, \ 1 \leq i \leq l-2\\
\end{array}
\right\},
\end{align}
is the base of the principal subspace $W_{N_{C_l^{(1)}}\left(k\Lambda_{0}\right)}$.
\end{thm}
\begin{flushright}
$\square$
\end{flushright}

\subsection{Characters of the principal subspace \texorpdfstring{$W_{N_{C_l^{(1)}}(k\Lambda_{0})}$}{WNCl(1)(kLambda0)}}
From the definition of the set $\mathfrak{B}_{W_{N_{C_l^{(1)}}(k\Lambda_0)}}$ and (\ref{S710Cl}- \ref{S712Cl}), (\ref{S7KB}) follows the character formula of the principal subspace $W_{N_{C_l^{(1)}}(k\Lambda_0)}$:
\begin{thm}\label{uvodCl2}
\begin{align}\label{characterCl} 
&\mathrm{ch} \  W_{N_{C_l^{(1)}}(k\Lambda_{0})}&\\
\nonumber
= &\sum_{\substack{r^{(1)}_{1}\geq \ldots \geq r^{(2u_1)}_{1}\geq 0\\ u_1\geq0 }}
\frac{q^{r^{(1)^{2}}_{1}+\cdots +r^{(2u_1)^{2}}_{1}}}{(q)_{r^{(1)}_{1}-r^{(2)}_{1}}\cdots (q)_{r^{(2u_1)}_{1}}}y^{r_1}_{1}&\\
\nonumber
 &\sum_{\substack{r^{(1)}_{2}\geq \ldots \geq r^{(2u_2)}_{2}\geq 0\\ u_2\geq 0}}
\frac{q^{r^{(1)^{2}}_{2}+\cdots +r^{(2u_2)^{2}}_{2}-r_1^{(1)}r_2^{(1)}-\cdots -r_1^{(2u_2)}r_2^{(2u_2)}}}{(q)_{r^{(1)}_{2}-r^{(2)}_{2}}\cdots (q)_{r^{(2u_2)}_{2}}}y^{r_2}_{2}&\\
\nonumber
&\ \ \ \ \ \ \ \ \cdots& \\
\nonumber
& \sum_{\substack{r^{(1)}_{l-1}\geq \ldots \geq r^{(2u_{l-1})}_{l-1}\geq 0\\ u_{l-1}\geq0}}
\frac{q^{r^{(1)^{2}}_{l-1}+\cdots +r^{(2u_{l-1})^{2}}_{l-1}-r_{l-2}^{(1)}r_{l-1}^{(1)}-\cdots -r_{l-2}^{(2u_{l-1})}r_{l-1}^{(2u_{l-1})}}}{(q)_{r^{(1)}_{l-1}-r^{(2)}_{l-1}}\cdots (q)_{r^{(2u_{l-1})}_{l-1}}}y^{r_{l-1}}_{l-1}&\\
\nonumber
&\sum_{\substack{r^{(1)}_{l}\geq \ldots \geq r^{(u_{l})}_{l}\geq  0\\ u_{l}\geq 0}}
\frac{q^{r^{(1)^{2}}_{l}+\cdots +r^{(u_{l})^{2}}_{l}-r_{l}^{(1)}(r_{l-1}^{(1)}+r_{l-1}^{(2)})
-\cdots -r_{l}^{(u_{l})}(r_{l-1}^{(2u_{l}-1)}+r_{l-1}^{(2u_{l})})}}{(q)_{r^{(1)}_{l}-r^{(2)}_{l}}\cdots (q)_{r^{(u_{l})}_{l}}}
y^{r_l}_{l}.&
\end{align}
\end{thm}
\begin{flushright}
$\square$
\end{flushright}
By Poincar\'{e}-Birkhoff-Witt theorem, we obtain the base of the universal enveloping algebra $U(\mathcal{L}(\mathfrak{n}_+)_{<0})$:
\begin{align}\label{102Cl}
&x_{\alpha_1}(m^1_1)\cdots  x_{\alpha_1}(m^{s_1}_1)x_{\alpha_1+\alpha_2}(m^1_2)\cdots x_{\alpha_1+\alpha_2}(m^{s_2}_2)\cdots x_{\alpha_1+\alpha_2+\cdots + \alpha_{l-1}}(m^1_{l-1})\cdots &\\
\nonumber
& \cdots x_{\alpha_1+\alpha_2+\cdots + \alpha_l}(m^{s_{l-1}}_{l-1})x_{\alpha_1+2\alpha_2+\cdots +2\alpha_l}(m^1_{l})\cdots   x_{\alpha_1+2\alpha_2+\cdots +2\alpha_l}(m^{s_{l }}_{l })\cdots &\\
\nonumber
& \cdots x_{\alpha_1+\alpha_2+\cdots +2\alpha_l}(m^1_{2l-2})\cdots   x_{\alpha_1+\alpha_2+\cdots +2\alpha_l}(m^{s_{2l-2}}_{2l-2})\cdots &\\
\nonumber
&  x_{2\alpha_1+2\alpha_2+\cdots +\alpha_l}(m^1_{2l-1})\cdots   x_{2\alpha_1+2\alpha_2+\cdots +\alpha_l}(m^{s_{2l-1}}_{2l-1})\cdots &\\
\nonumber
&\cdots x_{ \alpha_{l-1}}(m^1_{l^2-3})\cdots x_{ \alpha_{l-1}}(m^{s_{l^2-3}}_{l^2-3})\cdots x_{ 2\alpha_{l-1}+\alpha_l}(m^1_{l^2-1})\cdots x_{ 2\alpha_{l-1}+\alpha_l}(m^{s_{l^2-1}}_{l^2-1})&\\
\nonumber
& x_{ \alpha_l}(m^1_{l^2})\cdots x_{ \alpha_l}(m^{s_{l^2}}_{l^2}),&
\end{align}
with $m_i^1\leq \cdots \leq m_i^{s_i}$, $s_i \in \mathbb{N}$ for $i=1,\ldots,l^2$.
Now, we also have a following character formula
\begin{equation} 
\label{KN3Cl}
\textrm{ch} \ W_{N_{C_l^{(1)}}(k\Lambda_{0})}\\
\end{equation} 
\begin{equation*} 
=\prod_{m > 0} \frac{1}{(1-q^my_1)}\frac{1}{(1-q^my_1y_2)}\cdots \frac{1}{(1-q^my_1\cdots y_{l-1})}\frac{1}{(1-q^my_1y_2^2\cdots y_l^2)} 
\end{equation*} 
\begin{equation*}
\ \ \ \ \ \ \ \cdots \frac{1}{(1-q^my_1y_2\cdots y_l^2)}\frac{1}{(1-q^my^2_1y^2_2\cdots y_l )} 
\end{equation*}
\begin{equation*} 
\frac{1}{(1-q^my_2)}\frac{1}{(1-q^my_2y_3)}\cdots \frac{1}{(1-q^my_2\cdots y_{l-1})}\frac{1}{(1-q^my_2y_3^2\cdots y_l^2)} 
\end{equation*} 
\begin{equation*}
\ \ \ \ \ \ \ \cdots \frac{1}{(1-q^my_2y_3\cdots y_l^2)}\frac{1}{(1-q^my^2_2y^2_3\cdots y_l )} 
\end{equation*} 
\begin{equation*}
\ \ \ \ \ \ \ \ \cdots 
\end{equation*} 
\begin{equation*} 
\frac{1}{(1-q^{l-1})}\frac{1}{(1-q^my_{l-1}y^2_l)} \frac{1}{(1-q^my^2_{l-1}y_l)} \frac{1}{(1-q^my_l)} 
\end{equation*} 
From (\ref{characterCl}) and (\ref{KN3Cl}) now follows:
\begin{thm}
\begin{equation*} 
\prod_{m > 0} \frac{1}{(1-q^my_1)}\frac{1}{(1-q^my_1y_2)}\cdots \frac{1}{(1-q^my_1\cdots y_{l-1})}\frac{1}{(1-q^my_1y_2^2\cdots y_l^2)} 
\end{equation*} 
\begin{equation*} 
\ \ \ \ \ \ \ \cdots \frac{1}{(1-q^my_1y_2\cdots y_l^2)}\frac{1}{(1-q^my^2_1y^2_2\cdots y_l )} 
\end{equation*} 
\begin{equation*} 
\frac{1}{(1-q^my_2)}\frac{1}{(1-q^my_2y_3)}\cdots \frac{1}{(1-q^my_2\cdots y_{l-1})}\frac{1}{(1-q^my_2y_3^2\cdots y_l^2)} 
\end{equation*} 
\begin{equation*}
\ \ \ \ \ \ \ \cdots \frac{1}{(1-q^my_2y_3\cdots y_l^2)}\frac{1}{(1-q^my^2_2y^2_3\cdots y_l )} 
\end{equation*} 
\begin{equation*}
\ \ \ \ \ \ \ \ \cdots
\end{equation*} 
\begin{equation*} 
\frac{1}{(1-q^{l-1})}\frac{1}{(1-q^my_{l-1}y^2_l)} \frac{1}{(1-q^my^2_{l-1}y_l)} \frac{1}{(1-q^my_l)} 
\end{equation*}
\begin{equation*}
= \sum_{\substack{r^{(1)}_{1}\geq \ldots \geq r^{(2u_1)}_{1}\geq 0\\ u_1\geq0 }}
\frac{q^{r^{(1)^{2}}_{1}+\cdots +r^{(2u_1)^{2}}_{1}}}{(q)_{r^{(1)}_{1}-r^{(2)}_{1}}\cdots (q)_{r^{(2u_1)}_{1}}}y^{r_1}_{1} 
\end{equation*} 
\begin{equation*}
\sum_{\substack{r^{(1)}_{2}\geq \ldots \geq r^{(2u_2)}_{2}\geq 0\\ u_2\geq 0}}
\frac{q^{r^{(1)^{2}}_{2}+\cdots +r^{(2u_2)^{2}}_{2}-r_1^{(1)}r_2^{(1)}-\cdots -r_1^{(2u_2)}r_2^{(2u_2)}}}{(q)_{r^{(1)}_{2}-r^{(2)}_{2}}\cdots (q)_{r^{(2u_2)}_{2}}}y^{r_2}_{2} 
\end{equation*} 
\begin{equation*}
\ \ \ \ \ \ \ \ \cdots  
\end{equation*} 
\begin{equation*}
\sum_{\substack{r^{(1)}_{l-1}\geq \ldots \geq r^{(2u_{l-1})}_{l-1}\geq 0\\ u_{l-1}\geq0}}
\frac{q^{r^{(1)^{2}}_{l-1}+\cdots +r^{(2u_{l-1})^{2}}_{l-1}-r_{l-2}^{(1)}r_{l-1}^{(1)}-\cdots -r_{l-2}^{(2u_{l-1})}r_{l-1}^{(2u_{l-1})}}}{(q)_{r^{(1)}_{l-1}-r^{(2)}_{l-1}}\cdots (q)_{r^{(2u_{l-1})}_{l-1}}}y^{r_{l-1}}_{l-1}  
\end{equation*} 
\begin{equation*}
\sum_{\substack{r^{(1)}_{l}\geq \ldots \geq r^{(u_{l})}_{l}\geq  0\\ u_{l}\geq 0}}
\frac{q^{r^{(1)^{2}}_{l}+\cdots +r^{(u_{l})^{2}}_{l}-r_{l}^{(1)}(r_{l-1}^{(1)}+r_{l-1}^{(2)})
-\cdots -r_{l}^{(u_{l})}(r_{l-1}^{(2u_{l}-1)}+r_{l-1}^{(2u_{l})})}}{(q)_{r^{(1)}_{l}-r^{(2)}_{l}}\cdots (q)_{r^{(u_{l})}_{l}}}
y^{r_l}_{l}.
\end{equation*}
\end{thm}
\begin{flushright}
$\square$
\end{flushright}

\section*{Acknowledgement}
We would like to thank Mirko Primc for his help, and for important comments and advice concerning the earlier version of the paper. This work has been supported in part by the Croatian Science Foundation under the project 2634. and by University of Rijeka research grant 13.14.1.2.02. 

\end{document}